\documentclass[11pt,reqno]{amsart} 
\usepackage{a4wide}
\numberwithin{equation}{section}
\usepackage{mathrsfs}
\usepackage{amsfonts}
\usepackage{amsmath}
\usepackage{stmaryrd}
\usepackage{amssymb}
\usepackage{amsthm}
\usepackage{url}
\usepackage{amsfonts}
\usepackage{amscd}
\usepackage{indentfirst}
\usepackage{enumerate}
\usepackage{amsmath,amsfonts,amssymb,amsthm}
\usepackage{amsmath,amssymb,amsthm,amscd}
\usepackage{graphicx,mathrsfs}
\usepackage{appendix}
\usepackage[T1]{fontenc}
\usepackage[utf8]{inputenc}
\usepackage{wrapfig}
\usepackage{multicol}
\usepackage{mathrsfs}
\usepackage{tikz}
\usepackage[usenames,dvipsnames]{pstricks}
\usepackage{epsfig,subfigure}
\usepackage{pst-grad} 
\usepackage{pst-plot} 
\usepackage{textpos} 
\usepackage{tikz,calc}
\usepackage{empheq}
\usepackage[numbers,sort&compress]{natbib}
\usepackage[labelformat=empty]{caption}
\usepackage{sidecap}
\usepackage{caption}
\usepackage{empheq}
\usepackage{cancel}
\usepackage{pgfplots}
\usepackage{floatflt}
\usepackage{tcolorbox}
\usepackage{graphics,graphicx,color,multicol}
\newcommand*\widefbox[1]{\fbox{\hspace{2em}#1\hspace{2em}}}

\usepackage{color}
 \usepackage[colorlinks, linkcolor=blue, citecolor=blue]{hyperref}

\newtheorem{teo}{Theorem}[section]
\newtheorem{lem}[teo]{Lemma}
\newtheorem{prop}[teo]{Proposition}
\newtheorem{cor}[teo]{Corollary}
\newtheorem{rem}[teo]{Remark}

\newcommand\R{\mathbb R}

\newcommand\mbb\mathbb
\newcommand\mbf\mathbf
\newcommand\mcal\mathcal
\newcommand\mfrak\mathfrak
\newcommand\mrm\mathrm
\newcommand\msf\mathsf
\renewcommand\a\alpha
\renewcommand\b\beta
\newcommand\g\gamma
\newcommand\G\Gamma
\renewcommand\d\delta
\newcommand\D\Delta
\newcommand\e\varepsilon
\newcommand\z\zeta
\renewcommand\t\theta
\newcommand\Th\Theta
\newcommand\la\lambda
\newcommand\La\Lambda
\newcommand\s\sigma
\newcommand\si\varsigma
\newcommand\Si\Sigma
\newcommand\ups\upsilon
\newcommand\U\Upsilon
\newcommand\ph\varphi
\renewcommand\o\omega
\renewcommand\O\Omega
\newcommand\wt\widetilde
\newcommand\wh\widehat
\newcommand\ol\overline
\newcommand\ul\underline
\newcommand\mr\mathring
\newcommand\ub\underbrace
\newcommand\pa\partial
\newcommand\n\nabla
\newcommand\fa\forall
\newcommand\ex\exists
\newcommand\es\emptyset
\newcommand\wk\rightharpoonup
\newcommand\inc\hookrightarrow
\newcommand\linf\varliminf
\newcommand\lsup\varlimsup
\newcommand\os\overset
\newcommand\us\underset
\newcommand\sr\stackrel
\newcommand\Ot\Leftarrow
\newcommand\To\Rightarrow
\newcommand\map\mapsto
\newcommand\ot\leftarrow
\newcommand\lot\longleftarrow
\newcommand\lto\longrightarrow
\newcommand\tot\leftrightarrow
\newcommand\ltot\longleftrightarrow
\newcommand\sm\backslash
\renewcommand\Cup\bigcup
\renewcommand\Cap\bigcap
\newcommand\sub\subset
\newcommand\Sub\Subset
\newcommand\sne\subsetneq
\newcommand\bus\supset
\newcommand\Bus\Supset

\newcommand\eq\equiv
\newcommand\ox\otimes
\newcommand\Ox\bigotimes
\newcommand\pl\oplus
\newcommand\Pl\bigoplus
\newcommand\x\times
\renewcommand\c\circ
\newcommand\q\quad
\renewcommand\l\left
\renewcommand\r\right
\newcommand\fr\frac

\allowdisplaybreaks

\def\sideremark#1{\ifvmode\leavevmode\fi\vadjust{\vbox to0pt{\vss
 \hbox to 0pt{\hskip\hsize\hskip1em
 \vbox{\hsize2.1cm\tiny\raggedright\pretolerance10000
  \noindent #1\hfill}\hss}\vbox to15pt{\vfil}\vss}}}%

\begin{document}
\title[Qualitative analysis on the critical points of the Robin function]
{Qualitative analysis on the critical points \\ of the Robin function}
\author[F. Gladiali, M. Grossi, P. Luo and S. Yan]{Francesca Gladiali, Massimo Grossi, Peng Luo, Shusen Yan}

\address[Francesca Gladiali]{Dipartimento di Chimica e Farmacia, Universit\`a di Sassari, via Piandanna 4 - 07100 Sassari, Italy}
 \email{fgladiali@uniss.it}

\address[Massimo Grossi]{Dipartimento di Matematica Guido Castelnuovo, Universit$\grave{a}$ Sapienza, P.le Aldo Moro 5, 00185 Roma, Italy}
 \email{massimo.grossi@uniroma1.it}

 \address[Peng Luo]{School of Mathematics and Statistics, Central China Normal University, Wuhan 430079,  China }
 \email{pluo@ccnu.edu.cn}

\address[Shusen Yan]{School of Mathematics and Statistics, Central China Normal University, Wuhan 430079, China}
\email{syan@ccnu.edu.cn}

\begin{abstract}
Let $\O\subset\R^N$ be a smooth bounded domain with $N\ge2$ and $\O_\e=\O\backslash B(P,\e)$ where $B(P,\e)$ is the ball centered at $P\in\O$ and radius $\e$.
In this paper, we establish the number, location and non-degeneracy of critical points of the Robin function in $\O_\e$ for $\e$ small enough. We will show that the location of $P$ plays a crucial role on the existence and multiplicity of the critical points. The proof of our result is a consequence of delicate estimates on the Green function near to $\partial B(P,\e)$. Some applications to compute the exact number of solutions of related well-studied nonlinear elliptic problems will be showed.
\end{abstract}
\maketitle
\section{Introduction and main results}
\setcounter{equation}{0}

Let $D\subset\R^N$, $N\ge2$ be a smooth domain. For $(x,y)\in D\times D$, $x\ne y$, denote by $G_D(x,y)$ the Green function in $D$. It verifies
\begin{equation*}
\begin{cases}
-\D_x G_D(x,y)=\delta_x(y)&\hbox{in }D,\\[2mm]
G_D(x,y)=0&\hbox{on }\partial D,
\end{cases}
\end{equation*}
in the sense of distribution. We have the classical representation formula
\begin{equation}\label{i0}
G_D(x,y)=S(x,y)-H_D(x,y),
\end{equation}
where $S(x,y)$ is the  {\em fundamental solution} given by
\begin{equation*}
S(x,y)=
\begin{cases}
-\frac{1}{2\pi}\ln \big|x-y\big|&~\mbox{if}~N=2,\\[2mm]
\frac{C_N}{|x-y|^{N-2}} &~\mbox{if}~N\geq 3,
\end{cases}
\end{equation*}
where $C_N:=\frac{1}{N(N-2)\omega_N}$, with $\omega_N$ the volume of the unit ball in $\R^N$. $H_D(x,y)$ is the {\em regular part of the Green function} which is harmonic in both variables $x$ and $y$. The $Robin$ function is defined as
\begin{equation*}
\mathcal{R}_D(x):=H_D(x,x)~\,~\mbox{in}~D.
\end{equation*}
Note that our definition differs, up to a multiplicative constant, from that of some other authors (see for example \cite{bf} and \cite{cf2} where $\mathcal{R}_D(x)=2\pi H_D(x,x)$ for $N=2$). However, since we are interested in the critical points of the Robin function, this difference plays no role in our results.

The $Robin$ function plays a fundamental role in a great number of problems (see \cite{bf,f} and the references therein). It also plays a role in the theory of conformal mappings and is closely related to the inner radius function (see \cite{gu}) and to some
geometric quantities such as the capacity and transfinite diameter of sets (see \cite{bf,f}).
 For elliptic problems involving critical Sobolev exponent \cite{Glangetas,Rey},  the number of solutions is  linked to
 the the number of non-degenerate critical points of the  Robin function.
 Despite the great interest on the Robin function, many questions are still unanswered and we are far from a complete understanding of its properties.

The only bounded domain where the Robin function is explicitly known is the ball centered at a point $Q\in\R^N$; in this case the Robin function is $radial$ and $Q$ is the only critical point (it turns out to be non-degenerate). The computation of the number of critical points as well as other geometric properties (for example the shape of level sets) in general domains of $\R^N$ interested a great number of experts in PDEs, but anyway very few results are known.
One additional difficulty is that it is not known if  the Robin function satisfies some differential equation. This is known only for planar simply-connected domains (\cite{bf}) where it solves the Liouville equation.

To our knowledge, one of the first result  in general domains is \cite{cf2}, in which  Caffarelli and  Friedman
  proved that the Robin function admits only one non-degenerate critical point in convex and bounded domain in $\R^2$. Note that here a crucial role is played by the Liouville equation.
   Later, the existence and uniqueness of critical points of Robin function for a convex and bounded domain in higher dimension was proved in \cite{ct}. However, the non-degeneracy of this critical point is still open. Also some results on the non-degeneracy of the critical points of the Robin function for some symmetric domain can be found in \cite{g2}.
For non-convex domains, for example domains with ``rich'' topology, we cannot expect the uniqueness of the critical point of the Robin function. But how the topology of the domain impacts the number of  critical points of Robin function is still unclear and seems to be a very difficult issue.

In this paper we study what happens when we remove a small hole to a domain $\O\subset\R^N$.  More precisely, denoting by $B(x_0,r)$ the ball centered at $x_0$ of radius $r$, set
\begin{equation}\label{def}
\Omega_\e=\Omega\backslash B(P,\e)\,~\mbox{with}\,~P\in\Omega.
\end{equation}

\begin{figure}[hbt]
\begin{minipage}[t]{0.25\textwidth}
\centering
\includegraphics[width=\textwidth]{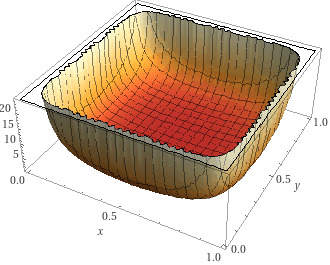}
{\em Robin function in $\O$}
\end{minipage}
\qquad\qquad\qquad
\begin{minipage}[t]{0.25\textwidth}
\centering
\includegraphics[width=\textwidth]{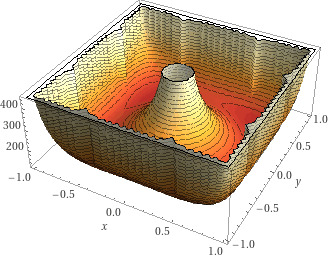}
{\em Robin function in $\O_\e$}
\end{minipage}
\end{figure}

Note that the Robin function $\mathcal{R}_{\O_\e}$ blows up at $\partial B(P,\e)$. So, for $\e$ small enough, $\mathcal{R}_\O$ and $\mathcal{R}_{\O_\e}$ look like very differently near $P$.

Our aim is to study the number of critical points of the Robin function $\mathcal{R}_{\O_\e}$ as well as their non-degeneracy.
Observe that the regular part $H_{\O_\e}(x,y)$ satisfies
$$
\begin{cases}
\D H_{\O_\e}(x,y)=0&\mbox{in}\ \O_\e,\\
H_{\O_\e}(x,y)=S(x,y)&\mbox{on}\ \partial\O_\e.
\end{cases}
$$
Hence, by the standard regularity theory, we have that
\begin{equation}\label{eqH}
H_{\O_\e}(x,y)\to  H_{\O}(x,y)
\end{equation}
in any compact set $K\subset\overline\O\setminus B(P,r)$ as $\e\to0$ for some small {\em fixed} $r>0$. Setting $x=y$ in \eqref{eqH} we get that $\mathcal{R}_{\O_\e}(x)\to\mathcal{R}_\O(x)$ for any $x\in K$, which is a good information on the behavior of $\mathcal{R}_{\O_\e}$ {\em far away} from $P$. The behavior of $\mathcal{R}_{\O_\e}$ close to $\partial B_\e$ is much more complicated and is the most delicate problem to be addressed in this paper. Here a careful use of the maximum principle for harmonic functions will be crucial. Actually, sharp estimates up to $\partial B(P,\e)$ will allow us to prove the existence of critical points for $\mathcal{R}_{\O_\e}$ which converge to $P$ as $\e\to0$.

\vskip 0.2cm

Our first result emphasizes the role of the center of the hole $B(P,\e)$. Indeed the scenario is very different depending on whether $P$ is a critical point of $\mathcal{R}_\O$ or not.

In all the paper we denote, for $x\in\O_\e$, by $O\big(f(\e,x)\big)$ a quantity such that
$$O\big(f(\e,x)\big)\le C|f(\e,x)|,$$
 where $C$ is a constant independent of $\e$ and $x$.

\begin{teo}\label{I6}
Suppose $\Omega$ is a bounded smooth domain in $\R^N$, $N\ge2$ with $P\in\Omega$.
\begin{floatingfigure}[r]{.30\textwidth}
\includegraphics[width=.35\textwidth]{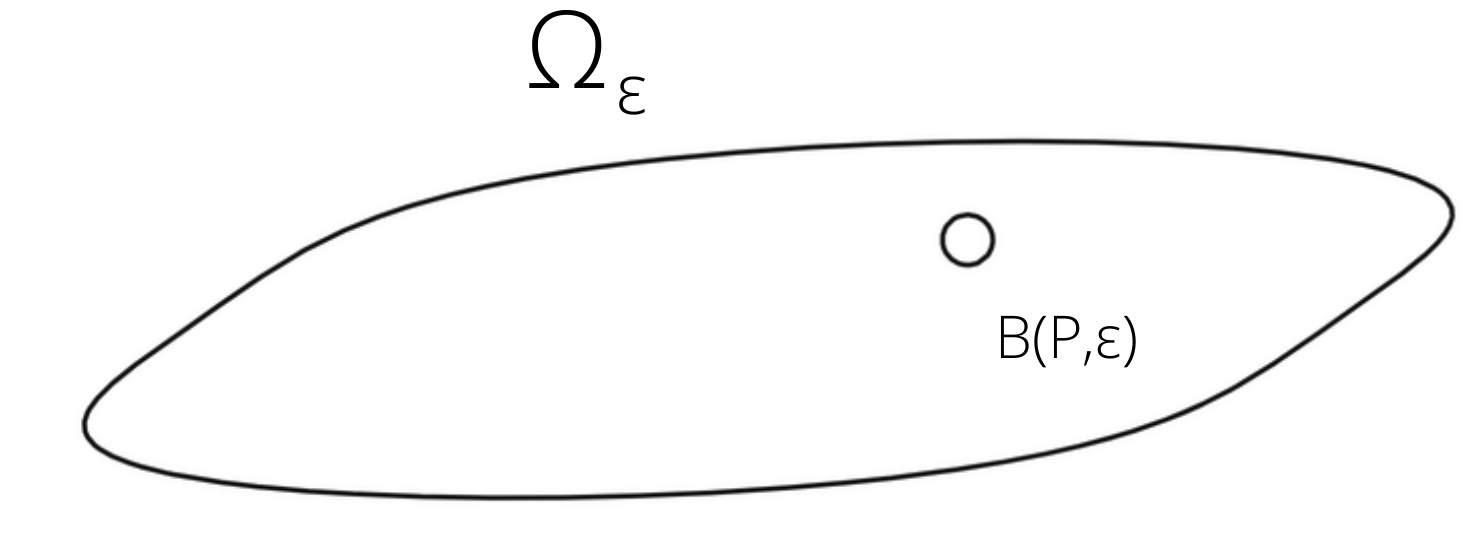}
\end{floatingfigure}
If
\[\nabla \mathcal{R}_\O(P)\neq 0,\]
\noindent then for $\varepsilon$ small enough,
\begin{equation*}\label{equa}
\!\!\!\!\!\!\!\!\!\!\!\!\sharp\Big\{\mbox{critical points of  $\mathcal{R}_{\O_\e}$ in $B(P,r)\setminus B(P,\e)$}\Big\}=1,
\end{equation*}
\noindent
where $B(P,r)\subset\O$ is chosen not containing any critical  point of $\mathcal{R}_\O$.
Moreover the critical point $x_\e\in B(P,r)$ of $\mathcal{R}_{\O_\e}$\hbox{ verifies, for $\e$ small enough,}
\begin{itemize}
\item[$(P_1)$]
\vskip-0.1cm
\begin{equation}\label{11-1-01}
x_\e=P+
\begin{cases}
 \e^\frac{N-2}{2N-3}\left(
\Big(\frac{2 }{ N\omega_N |\nabla \mathcal{R}_\O(P)|^{2N-2}}\Big)^{\frac{1}{2N-3}} \nabla \mathcal{R}_\O(P) +o(1)\right)&~\mbox{for}~N\geq 3,\\[4mm]
 r_\e \left( \frac{\nabla \mathcal{R}_\O(P)}{ \pi^2 |\nabla \mathcal{R}_\O(P)|^{2}}   +o(1)\right) &~\mbox{for}~N=2,
\end{cases}
\end{equation}
where $r_\e$ is the unique solution of  the equation
 \begin{equation}\label{re}
r-\frac{\ln r}{\ln \e}=0~~\mbox{in}~(0,\infty).
\end{equation}
\vskip 0.1cm
\item[$(P_2)$]
$x_\e$ is a non-degenerate critical point  with
$\mbox{index}_{x_\e}\big(\nabla \mathcal{R}_{\O_\e}\big)=(-1)^{N+1}$.
\vskip 0.2cm
\item[$(P_3)$] $\mathcal{R}_{\O_\e}(x_\e)\to\mathcal{R}_{\O}(P)$.
\end{itemize}
\end{teo}

\begin{rem}\label{Ir}
The condition $\nabla \mathcal{R}_\O(P)\neq 0$ cannot be removed. Indeed, if $\O=B(0,R)$ we know that $0$ is the unique critical point of $\mathcal{R}_\O$ and  in the annulus $\O_\e=B(0,R)\backslash B(0,\e)$ we have that $\mathcal{R}_{\O_\e}$ is radial with respect to the origin. Since $\mathcal{R}_{\O_\e}\Big|_{\partial \O_\e}=+\infty$ we have that the set of minima of $\mathcal{R}_{\O_\e}$ is a sphere and then $\mathcal{R}_{\O_\e}$ admits infinitely many minima.
\end{rem}
\begin{rem}
Although $\mathcal{R}_{\O_\e}\Big|_{\partial  B_\e}=+\infty$, we have by $(P_3)$ that $\mathcal{R}_{\O_\e}(x_\e)\le C$. Roughly speaking this means that $x_\e$ is not ``so close'' to $\partial B(P,\e)$ (actually by $(P_1)$ we have that $\frac{|x_\e-P|}{\e}\to+ \infty$).

\end{rem}
\begin{rem}\label{R1}
Let us give the idea of the proof of the theorem. Our starting point is the following basic representation formula for the gradient of the Robin function (see \cite{sc}, and \cite{gsc} for $N=2,3$ and \cite{bp} for any $N\ge2$),
\begin{equation}\label{I7}
\begin{split}
 \nabla\mathcal{R}_{\O_\e}(x)=
&\int_{\partial\Omega_\e}\nu_\e(y)\left(\frac{\partial G_{\O_\e}(x,y)}{\partial\nu_y}\right)^2d\sigma_y\\=
&\int_{\partial\Omega}\nu(y)\left(\frac{\partial G_{\O_\e}(x,y)}{\partial\nu_y}\right)^2d\sigma_y-\int_{\partial B(P,\e)}\frac{y-P}\e\left(\frac{\partial G_{\O_\e}(x,y)}{\partial\nu_y}\right)^2d\sigma_y,
\end{split}
\end{equation}
where $\nu_\e(y)$ and $\nu(y)$ are the outer unit normal to $\partial\O_\e$ and $\partial\O$ respectively. By \eqref{I7} we will derive some $C^1$-estimates which are crucial to prove our results and, in our opinion, have an independent interest. Let us start to discuss the case $N\ge3$,  where we get, \underline{uniformly} for $x\in\O_\e$,
\begin{tcolorbox}[colback=white,colframe=black]
\begin{equation}\label{I8}
\nabla\mathcal{R}_{\O_\e}(x)=\nabla\mathcal{R}_\O(x) +\!\!\!\!\!
\underbrace{\nabla\mathcal{R}_{\R^N\setminus B(P,\e)}(x)}_{=-\frac{2\e^{N-2}}{N\omega_N}\frac{x-P}{\left(|x-P|^2-\e^2\right)^{N-1}}}\!\!\!\!\!\!+O\left(\frac{\e^{N-2}}{|x-P|^{N-1}}\right)+O(\e).
 \end{equation}
\end{tcolorbox}
The previous estimate is a {\bf second order} expansion of the Robin function in $\O_\e$.\\
 It turns out that $\nabla\mathcal{R}_\O(x)$ and $\nabla\mathcal{R}_{\R^N\setminus B(P,\e)}(x)$ are the leading terms of the expansion of $\nabla \mathcal{R}_{\O_\e}(x)$.

Note that from \eqref{I8} we get that $\nabla\mathcal{R}_{\O_\e}(x)\to\nabla\mathcal{R}_\O(x)$ uniformly on the compact sets of $\O$ not containing $P$. So, under some non-degeneracy assumptions on the critical points of $\mathcal{R}_\O(x)$, we get that the number of critical points of $\mathcal{R}_{\O_\e}(x)$ far away from $P$ is the same as $\mathcal{R}_\O(x)$.
Next let us study what happens when $x\to P$. Here the analysis is very delicate but
formally, using the uniform convergence in \eqref{I8}, we get
$$\nabla\mathcal{R}_{\O_\e}\left(P+\e^\frac{N-2}{2N-3}y\right)\sim\nabla\mathcal{R}_\O(P)-\frac2{N\omega_N}\frac y{|y|^{2N-2}}=\nabla\left(
\sum_{j=1}^N\frac{\partial\mathcal{R}_\O(P)}{\partial x_j}y_j-\frac2{N\omega_N(4-2N)}\frac1{|y|^{2N-4}}\right).$$
Since the function $F(y)=\displaystyle\sum_{j=1}^N\frac{\partial\mathcal{R}_\O(P)}{\partial x_j}y_j-\frac2{N\omega_N(4-2N)}\frac1{|y|^{2N-4}}$ admits a unique non-degenerate critical point with index  $(-1)^{N+1}$ we get that the same holds for $\mathcal{R}_{\O_\e}$, which proves $(P_1)$ and $(P_2)$ for $N\ge3$.
\vskip 0.1cm

The case  $N=2$ is a little bit more complicated because $\mathcal{R}_{\R^2\setminus B(P,\e)}(x)$ does not goes to $0$ as $\e\to0$. Actually, an additional term appears in the expansion,
\begin{tcolorbox}[colback=white,colframe=black]
\begin{equation}\label{I9}
\!\!\!\!\nabla\mathcal{R}_{\O_\e}(x)=\nabla\mathcal{R}_\O(x) +
\underbrace{\nabla\mathcal{R}_{\R^2\setminus B(P,\e)}(x)}_{=-\frac1\pi\frac{x-P}{|x-P|^2-\e^2}}+\frac 1{\pi}\left(1-\frac {\ln |x-P|}{\ln \e}\right)\frac{x-P}{|x-P|^2}+O\left(\frac1{|x-P| |\ln\e|}\right).
\end{equation}
\end{tcolorbox}
However, up to some technicalities, the proof follows the same line as in the case $N\ge3$.\\
Finally let us point out that the maximum principle plays a crucial role to get uniform estimates up to $\partial\O$ in \eqref{I8} and \eqref{I9}.
\end{rem}
\begin{rem}
An interesting asymptotic formula for the Robin function in $\O_\e$ is the following (see \cite{bf}, p.198-199): for every $x\ne P$ and $N\ge3$,
\begin{equation}\label{Irem}
\mathcal{R}_{\O_\e}(x)=\mathcal{R}_\O(x)+\frac{G_\O^2(x,P)}{N(N-2)
\omega_N\left(\e^{2-N}-\mathcal{R}_\O(P)\right)}+O\left(\e^{N-1}\right),
\end{equation}
(an analogous formula holds for $N=2$). \eqref{Irem} is a consequence of the Schiffer-Spencer formula (see \cite{ss}  for $N=2$ and \cite{O} for $N\ge3$) but the remainder term $O\left(\e^{N-1}\right)$ is not uniform with respect to $x$ (as stated at page 771 in  \cite{O}).

In Proposition \ref{Re1} we prove the following,
\begin{tcolorbox}[colback=white,colframe=black]
\begin{equation*}
\begin{split}
&\mathcal{R}_{\O_\e}(x)=\mathcal{R}_{\O}(x)+ \mathcal{R}_{B^c_\e}(x)+\begin{cases} O\Big(\frac{\e^{N-2}}{|x-P|^{N-2}}\Big)+O(\e)\,\,\,~&\mbox{for}~N\geq 3,
\\[3mm]
\frac 1{2\pi}\ln\frac\e{|x-P|^2}+O\left(\frac1{|\ln\e|}\right)~&\mbox{for}~N=2,
\end{cases}
\end{split}
\end{equation*}
\end{tcolorbox}\noindent
where the remainder terms are uniform with respect to $x\in\O_{\e}$. This can be seen as an extension of  \eqref{Irem}.
\end{rem}
Theorem \ref{I6} states the $uniqueness$ and $non-degeneracy$ of the critical points of $\mathcal{R}_{\O_\e}$ near the hole $B(P,\e)$. Under a non-degeneracy condition on the critical points of $\mathcal{R}_{\O}$ we can compute the exact number of the critical points of $\mathcal{R}_{\O_\e}$ in $\O_\e$.
\begin{cor}\label{th1-2dd}
Suppose that $\Omega$ and $\Omega_\e$ are the domains as in Theorem \ref{I6}.
If $\nabla \mathcal{R}_\O(P)\neq 0$ and all critical points of $\mathcal{R}_\O$ in $\Omega$  are  non-degenerate,
then for $\varepsilon$ small enough, all critical points of $\mathcal{R}_{\O_\e}$ are non-degenerate and
\begin{equation*}
\sharp\Big\{\mbox{critical points of $\mathcal{R}_{\O_\e}$ in $\Omega_\e$}\Big\}=\sharp\big\{\hbox{critical points of $\mathcal{R}_\O$ in $\O$}\big\}+1.
\end{equation*}
\end{cor}

As previously mentioned, the Robin function of the ball $B(0,R)$ has a unique non-degenerate critical point. So Corollary \ref{th1-2dd} applies and then if $P\ne0$ the Robin function  $\mathcal{R}_{\O_\e}$ has $two$ non-degenerate critical points.
On the other hand, if $P=0$, we are in the situation described in Remark \ref{Ir}.
Hence a nice consequence of the previous corollary is the following one.
\begin{cor}
Assume that $\O=B(0,R)\subset\R^N$, $N\ge2$ and $\O_\e=B(0,R)\backslash B(P,\e)$. Then, for $\e$ small enough,
\begin{equation*}
\sharp\Big\{\mbox{critical points of $\mathcal{R}_{\O_\e}$ in $\Omega_\e$}\Big\}=
\begin{cases}
2&\hbox{if }P\ne0,\\
\infty&\hbox{if }P=0,
\end{cases}
\end{equation*}
and if $P\ne0$ the two critical points are non-degenerate.
\begin{figure}[hbt]
\begin{minipage}[t]{0.25\linewidth}
\centering
\includegraphics[width=\textwidth]{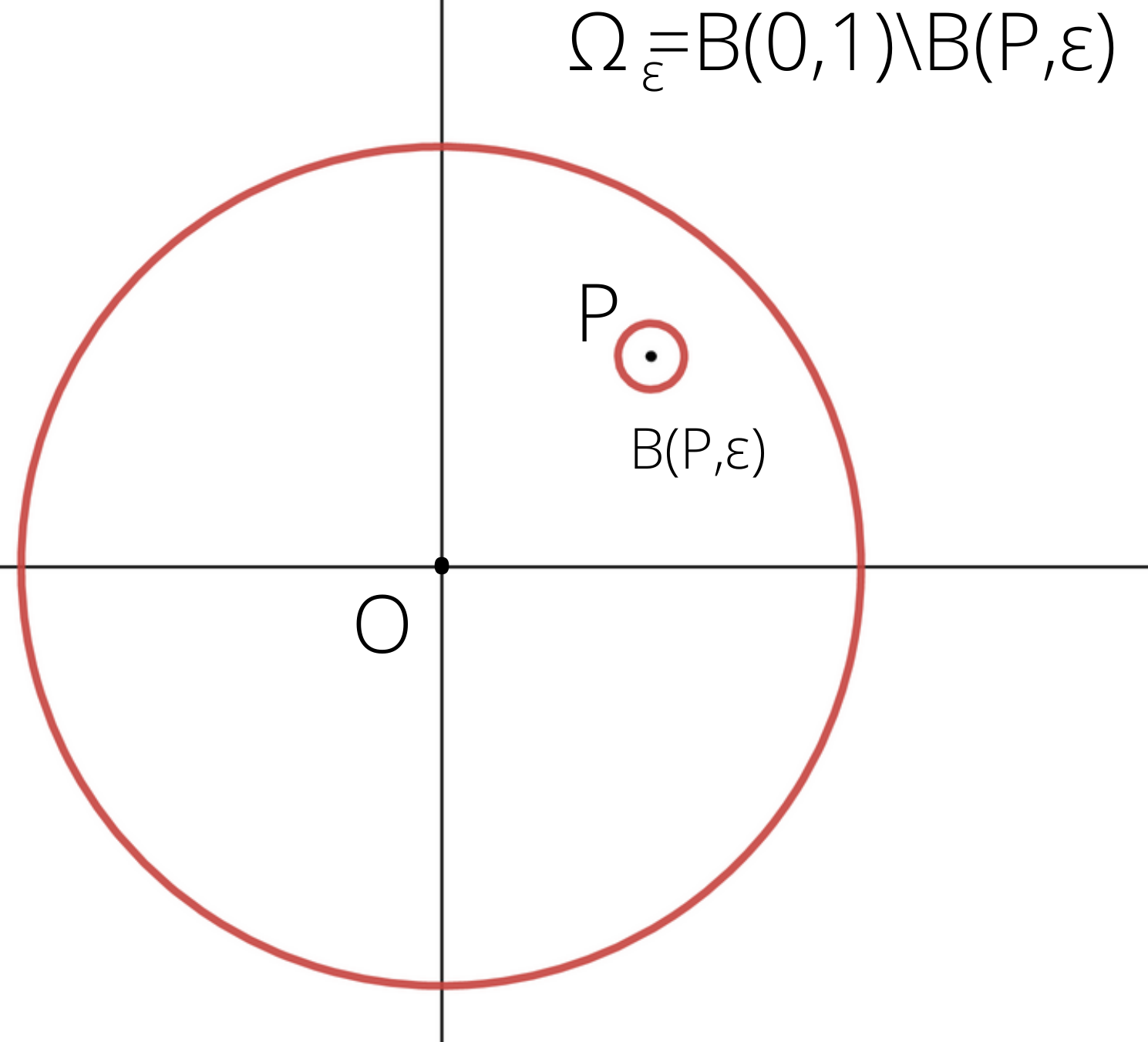}
{\scriptsize $P\ne0$ (two critical points)}
\end{minipage}
\qquad\qquad\qquad
\begin{minipage}[t]{0.25\linewidth}
\centering
\includegraphics[width=\textwidth]{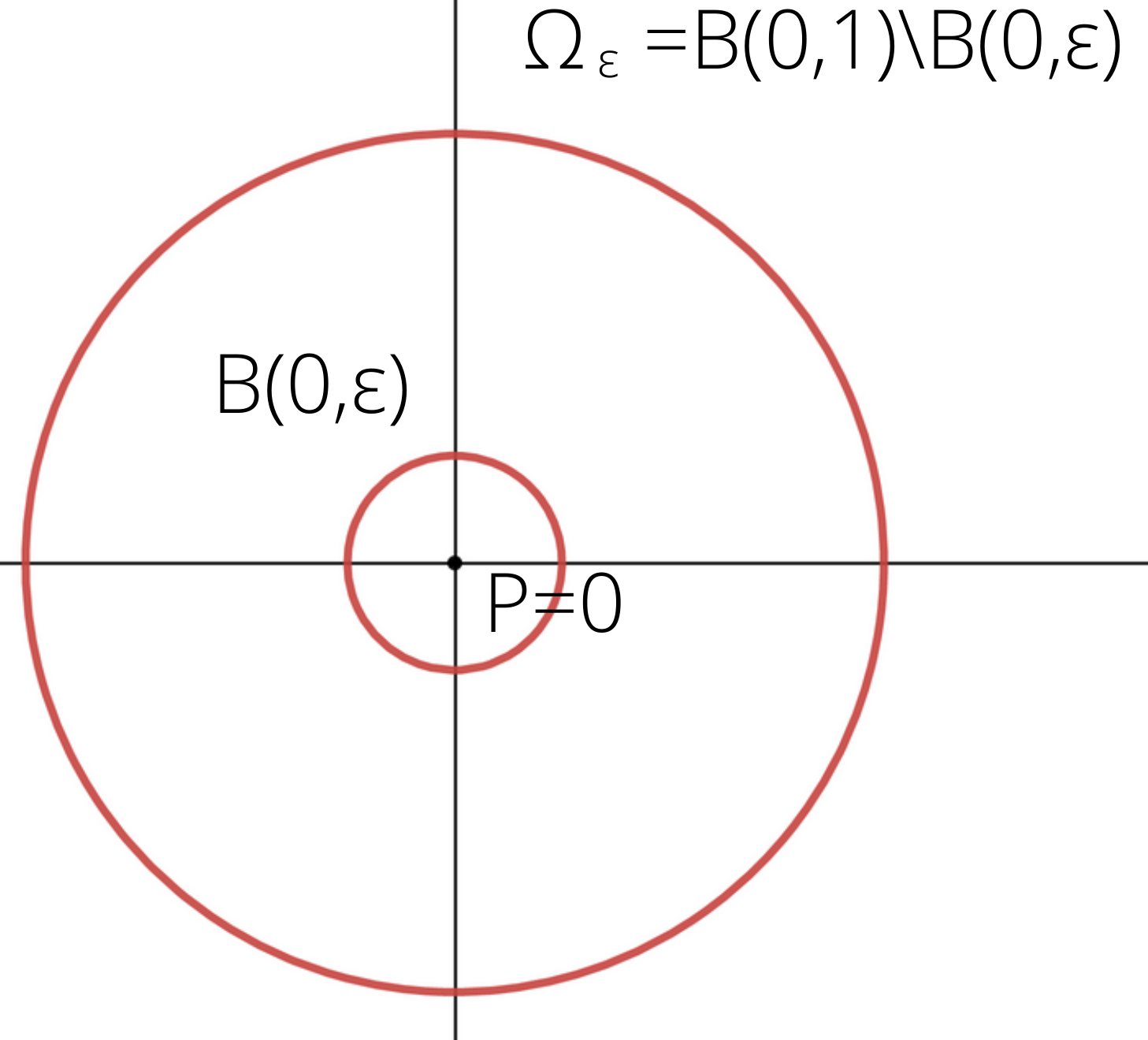}
{\scriptsize \hbox{$P=0$ (infinitely many critical points)}}
\end{minipage}
\end{figure}
\end{cor}

In next theorem we study what happens when $\nabla \mathcal{R}_\O(P)=0$. This case is more delicate and it seems very hard to give a complete answer. However, in some cases it is possible to compute the number of the critical points, as stated in the following.
\vskip0.2cm
 \begin{teo}\label{th1-2}
Suppose $\Omega$ is a bounded smooth domain in $\R^N$, $N\ge2$ and $P\in\Omega$. If
\begin{itemize}
\item $\nabla \mathcal{R}_\O(P)=0$.
\item The critical point $P$ is non-degenerate, i.e. $det\big(Hess\big(\mathcal{R}_\O(P)\big)\big)\neq 0$.
\item The Hessian matrix $Hess\big(\mathcal{R}_\O(P)\big)$ has $m\le N$ positive  eigenvalues
$0<\lambda_1\leq \lambda_{2} \leq \dots \leq \lambda_m$ with associated eigenvectors $v_1,\dots,v_m$ with $|v_i|=1$ for $i=1,\cdots,m$.
\end{itemize}
Assume that $B(P,r)\subset\O$ admits $P$ as the only critical point of $\mathcal{R}_\O$.
If the eigenvalue $\lambda_l$ is simple for some $l\in \{1,\cdots,m\}$, then we have
two non-degenerate critical  points $x_{l,\e}^\pm$ of $\mathcal{R}_{\O_\e}(x)$ in $B(P,r)\setminus B(P,\e)$
which satisfy
\begin{equation}\label{11-1-02}
x_{l,\e}^\pm=P\pm
\begin{cases}
\left(\frac{ 2+o(1) }{N\omega_N\la_l}\right)^\frac{1}{2N-2}
\varepsilon^{\frac{N-2}{2N-2}}v_l&\hbox{if }N\ge3,\\[4mm]
\widehat r_{\e,l}\Big(1+o(1)\Big)v_l~&\mbox{if}~N=2,
\end{cases}
\end{equation}
where
 $v_l$  is the $l$-th eigenvector associated to $\la_l$ with $|v_l|=1$, $\widehat r_{\e,l}$ is the
 unique solution of
 \begin{equation*}
 r^2-\frac{\ln r}{\lambda_l\pi \ln \e}=0~~\mbox{in}~(0,\infty).
 \end{equation*}
 Finally it holds
  \begin{equation}\label{Irob}
\mathcal{R}_{\O_\e}(x_{l,\e}^\pm)\to \mathcal{R}_\O(P).
 \end{equation}
Moreover, if all the positive eigenvalues of $\mathcal{R}_\O(P)$ are simple we have that
\begin{equation}\label{In}
\sharp\Big\{\mbox{critical points of  $\mathcal{R}_{\O_\e}(x)$ in $B(P,r)\setminus B(P,\e)$}\Big\}=2m
\end{equation}
and all critical points satisfy \eqref{11-1-02} for $l=1,\cdots,m$.
\end{teo}
\begin{rem}
The condition $det\big(Hess\mathcal{R}_\O(P)\big)\neq 0$ is widely verified. In \cite{mipi} it was proved that it holds up small perturbations of the domain $\O$.
Note that only positive eigenvalue of the Hessian matrix of $\mathcal{R}_\O(P)$ ``generates'' critical points for $\mathcal{R}_\O$ (see Proposition \ref{d1}). Hence saddle points of
$\mathcal{R}_\O$ give less contribution to the number of critical points of $\mathcal{R}_{\O_\e}$.
\end{rem}
\begin{cor}\label{Ic}
Let $\O$ be a convex and symmetric domain (see Gidas, Ni and Nirenberg \cite{gnn}) with respect to the origin.
We have that
\vskip 0.1cm
\noindent\textup{(i)} If $P\neq 0$, then  $\mathcal{R}_{\O_\e}(x)$ admits exactly two non-degenerate critical points in $\Omega_\e$.
\vskip 0.1cm
\noindent\textup{(ii)} If $P= 0$ and all the eigenvalues of $Hess\big(\mathcal{R}_\O(0)\big)$ are simple then $\mathcal{R}_{\O_\e}(x)$ admits exactly $2N$ non-degenerate critical points in $\Omega_\e$.\\
\begin{minipage}{1.5\linewidth}
\begin{tikzpicture}[scale=0.5]
\begin{axis} [axis y line=left,
y axis line style={opacity=0},
axis x line=left,
x axis line style={opacity=0},
axis equal, axis lines* = center,
xtick = \empty, ytick = \empty,legend style={anchor=north}
shift = {(axis cs: 917, 0)}]

\addplot[only marks,color=red,mark size=0.8pt] coordinates {({sin(0)},{75*cos(0)})};
\addplot[only marks,color=red,mark size=0.8pt] coordinates {({sin(0)},{-75*cos(0)})};
\addplot[only marks,color=red, mark size=0.8pt] coordinates {({75*cos(0)},{sin(0)})};
\addplot[only marks,color=red, mark size=0.8pt] coordinates {({-75*cos(0)},{sin(0)})};
\addplot[only marks,mark size=0.9pt] coordinates {(0,0)};

 \addplot
[domain=0:360,variable=\t,
samples=400,smooth]
({(184*cos(t)},{(120*sin(t)});
 \addplot
[domain=0:360,variable=\t,
samples=700,smooth]
({(50*cos(t)},{(50*sin(t)});
\node[label={0:{$\textcolor{red}{x_{2,\e}^+\to0}$}}] at (axis cs: {sin(0)},{75*cos(0)}) {};
\node[label={270:{$\textcolor{red}{x_{1,\e}^+\to0}$}}] at (axis cs: {85*cos(0)},{sin(0)}) {};
\node[label={270:{$\textcolor{red}{x_{1,\e}^-\to0}$}}] at (axis cs: {-85*cos(0)},{50*sin(0)}) {};
\node[label={0:{$\textcolor{red}{x_{2,\e}^-\to0}$}}] at (axis cs: {sin(0)},{-75*cos(0)}) {};
\node[label={30:{$\partial B(0,\e)$}}] at (axis cs: {30*cos(0)},{20*cos(0)}) {};
\node[label={10:{$\partial\O$}}] at (axis cs: {133*cos(0)},{125*sin(0)}) {};
\node[label={0:\tiny{$\!\!\!P\equiv0$}}] at (axis cs: {sin(0)},{sin(0)}) {};
\end{axis}\end{tikzpicture}
\end{minipage}

\end{cor}
\begin{rem}
An example of a domain $\O$ which satisfies the conditions in Theorem \ref{th1-2} for $N\geq 2$ is the following
(see Section \ref{es}),
$$\Omega_\delta=\Big\{x\in \R^N,~ \displaystyle\sum^N_{i=1}x_i^2\Big(1+\alpha_i \delta\Big)^2 <1~~\mbox{with}~~\delta>0\hbox{ and }0<\alpha_1\leq \alpha_2\le\cdots \leq \alpha_N\Big\}.$$
We will give a precise description of the Robin function $\mathcal{R}_{\O_\delta}$ for $\delta$ small in Theorem \ref{th1-3a}, Section \ref{es}.
\end{rem}
\begin{rem}
The proof of Theorem \ref{th1-2} uses again the estimates  \eqref{I8} and  \eqref{I9}. In this case we get that
\begin{equation*}
\frac{\partial\mathcal{R}_{\O_\e}(P+\e^\frac{N-2}{2N-2}y)}{\partial x_i}\sim\sum^N_{j=1}\frac{\partial^2 \mathcal{R}_\O(P)}{\partial x_i\partial x_j}y_j -\frac2{N\omega_N}\frac{y_i} {|y|^{2N-2}}\,\,\,~\mbox{for}~N\geq 2.
\end{equation*}
In other words, after diagonalization we get that
$$\nabla\mathcal{R}_{\O_\e}\left(P+\e^\frac{N-2}{2N-2}y\right)\sim\nabla\left(
\sum_{j=1}^N\lambda_jy_j^2-\frac2{N\omega_N(4-2N)}\frac1{|y|^{2N-4}}\right)\,\,\,~\mbox{for}~N\geq 3.$$
Note that if $0<\lambda_1<\lambda_2<\cdots<\lambda_m$ and $\lambda_j<0$ for $j=m+1,\cdots,N$, then the function $F(y)=\displaystyle\sum_{j=1}^N\lambda_jy_j^2-\frac2{N\omega_N(4-2N)}\frac1{|y|^{2N-4}}$ admits $2m$ non-degenerate critical points and the claim follows as in Remark \ref{R1}. On the other hand, if some positive eigenvalue is multiple, say $\lambda_1=\lambda_2=\cdots=\lambda_k>0$, then the function $F$ is spherically symmetric in $(y_1,y_2,\cdots,y_k)$. This implies that there is a set of critical points of $F$ given by a sphere $S^k$. Since we have a manifold of critical points, the number of the critical points depends on the approximation function and it leads to a {\bf third order} expansion of $\mathcal{R}_{\O_\e}(x)$.  Further considerations on multiple eigenvalues deserve to be studied apart.
However, when $\O$ is a symmetric domain, we obtain partial results on the critical points of $\mathcal{R}_{\O_\e}(x)$ (see Theorem \ref{th1-2a}).

\end{rem}

\begin{rem}
Our results can be iterated to handle the case in which $k~(k\geq 2)$  small holes are removed from $\O$.  Moreover, using similar ideas,  our main theorems are true if we replace $B(P,\e)$ by a small convex set.
\end{rem}

As in the case $\nabla\mathcal{R}_{\O}(P)=0$ we have the following corollary.
\begin{cor}\label{th1-2tt}
Suppose $\Omega$ is a bounded smooth domain in $\R^N$, $N\ge2$ and $P\in\Omega$.
If $\nabla \mathcal{R}_\O(P)=0$, all critical points of $ \mathcal{R}_\O$ are non-degenerate
and $Hess\big(\mathcal{R}_\O(P)\big)$ has $m\leq N$ positive eigenvalues which are all simple, for small $\e$ it holds
\begin{equation*}
\begin{split}
 \sharp\big\{\hbox{critical points of $\mathcal{R}_{\O_\e}(x)$ in $\O_\e$}\big\}= \sharp\big\{\hbox{critical points of $\mathcal{R}_\O(x)$ in $\O$}\big\}
+2m -1.
\end{split}
\end{equation*}
Finally all critical points of $\mathcal{R}_{\O_\e}(x)$ are non-degenerate.
\end{cor}

As said before the non-degeneracy of critical points
of the Robin function plays an important role in PDEs. Now  we would like to give applications of our results to
some elliptic problems. For example, let us consider the following,
\begin{equation}\label{B-N}
\begin{cases}
-\Delta u= u^{p}  &{\text{in}~\O_\e},\\
u>0  &{\text{in}~\O_\e},\\
u=0  &{\text{on}~\partial\O_\e},
\end{cases}
\end{equation}
where the solution $u_{\e,p}$ satisfies either
\begin{equation}\label{aa1}
\lim_{p\rightarrow +\infty} p \int_{\Omega}|\nabla u_{\e,p}|^2=8\pi e,~ ~\mbox{for} ~~N=2,
\end{equation}
or
\begin{equation}\label{aa}
\lim_{p\rightarrow \frac{N+2}{N-2}^-}
\frac{\displaystyle\int_{\O_\e}|\nabla u_{\e,p}|^2}{\left(\displaystyle\int_{\O_\e}u_{\e,p}^{p+1}dx\right)^{\frac{1}{p+1}}}=S,~~\mbox{for}~~N\geq 3,
\end{equation}
with $S$ the best constant in Sobolev inequality. We have following results.
\begin{teo}\label{th1-3}
Suppose $N=2$ or $N\geq 4$, $\Omega\subset \R^N$ is a convex domain and the critical point of $\mathcal{R}_\O(x)$ is non-degenerate (the uniqueness was proved in \cite{cf,ct}, the non-degeneracy in \cite{cf} for $N=2$). Then there exists some $\e_0>0$ such that for any fixed $\e\in (0,\e_0]$, $p$ large for $N=2$ or $\frac{N+2}{N-2}-p>0$ small for $N\geq 4$,
\vskip 0.2cm
\noindent\textup{(1)} if $\nabla \mathcal{R}_\O(P)\neq 0$ we have exact  two solutions of \eqref{B-N} satisfying \eqref{aa1} for $N=2$ or  \eqref{aa} for $N\geq 4$;
\vskip 0.2cm
\noindent\textup{(2)}  if $\nabla \mathcal{R}_\O(P)=0$ and all the eigenvalues of $\nabla^2\mathcal{R}_\O(P)$ are simple, then we have  exact  $2N$ solutions of \eqref{B-N} satisfying \eqref{aa1} for $N=2$ or  \eqref{aa} for $N\geq 4$.
\end{teo}
\begin{SCfigure}[1.5]
\includegraphics[scale=0.15]{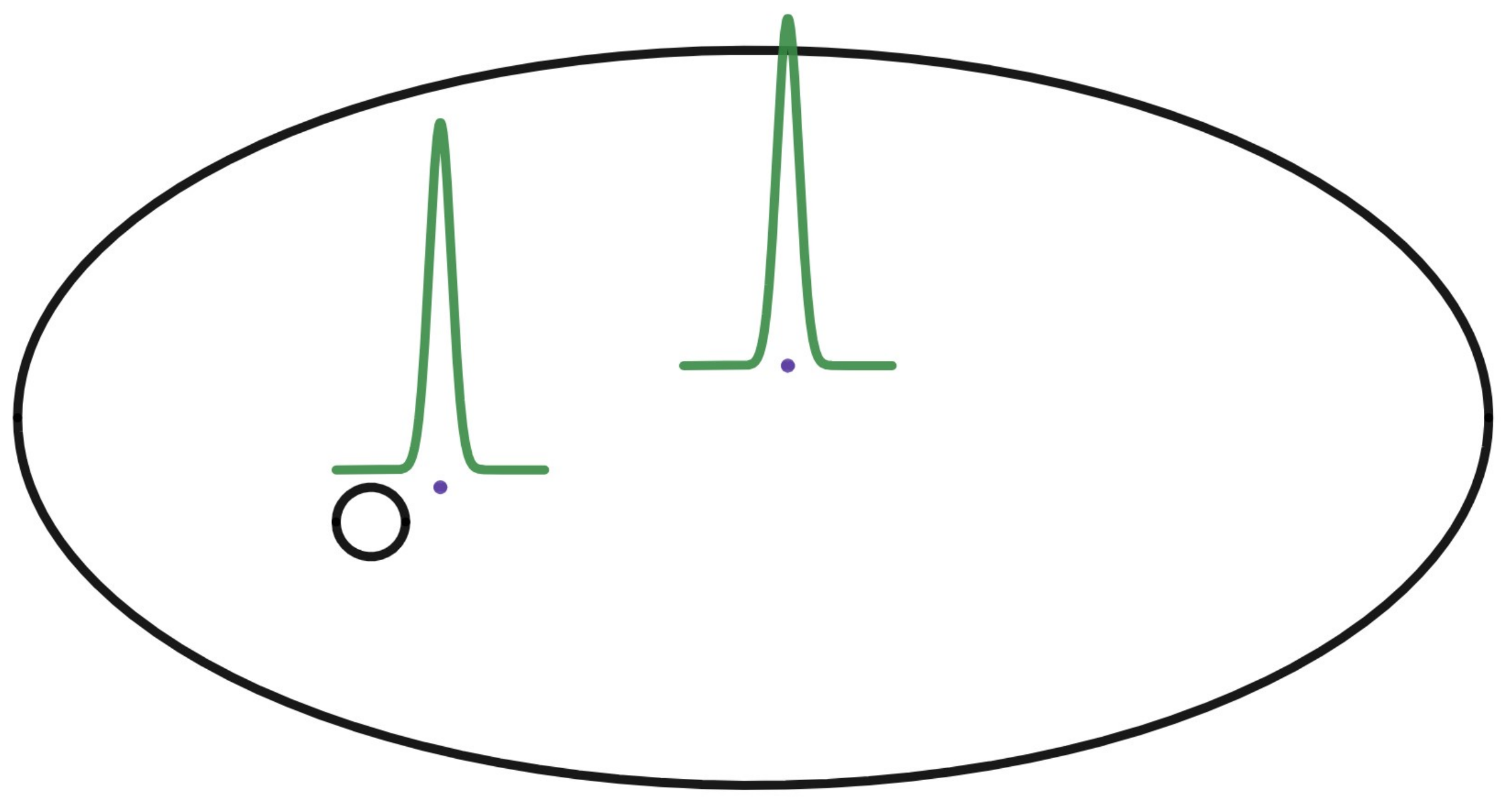}
\caption{{\em Picture of two solutions concentrating at critical points of $\mathcal R_{\O_\e}$}}
\end{SCfigure}
\begin{proof}[Proof of Theorem \ref{th1-3}]
Firstly, let us fix $\e\in(0,\e_0]$ such that Theorems \ref{I6} and \ref{th1-2} apply.
Then if $\nabla\mathcal R_\O(P)\neq 0$, from Corollary \ref{th1-2dd} we get that $\mathcal R_{\O_\e}$ admits exact  two critical points in $\Omega_\e$ which are non-degenerate. If  $\nabla \mathcal{R}_\O(P)=0$ and all the eigenvalues of $\nabla^2\mathcal{R}_\O(P)$ are simple, then Theorem \ref{th1-2} gives us that $\mathcal R_{\O_\e}$ admits exact  $2N$ critical points in $\Omega_\e$ which are non-degenerate.
Next it is known by \cite{DIP,h,Rey} that the solutions of \eqref{B-N} with
\eqref{aa1} or \eqref{aa} concentrate at critical points of $\mathcal R_{\O_\e}$  when $p$ is large for $N=2$ or $\frac{N+2}{N-2}-p>0$ small for $N\geq 4$. Moreover from \cite{blr,GILY}, we find that, using the non-degeneracy assumption of the critical points of $\mathcal R_{\O_\e}$, we have the {\em local} uniqueness of these solutions.
\end{proof}
\begin{rem}
Observe that in above corollary, the assumption $N\ge4$ instead of the natural one $N\ge3$ is due to technical reason in proving the uniqueness result in \cite{blr}. If the uniqueness result in  \cite{blr} is extended to $N=3$ we will get the claim also in this case.

Similar applications can also be given on Brezis-Nirenberg problem \cite{bn,Glangetas}, planar vortex patch in incompressible steady flow \cite{CGPY2019,CPY2015} and plasma problem \cite{cf,CPY2010} for some non-convex domains.
\end{rem}
The paper is organized as follows. In Section \ref{s2} we prove some lemmas which estimate the regular part $H_{\O_\e}$ in terms of $H_\O$ and $H_{\R^N\setminus B(P,\e)}$. Here the maximum principle for harmonic functions allows to get {\em uniform} estimates for $\mathcal{R}_{\O_\e}$ up to $\partial B(P,\e)$. These will be the basic tool to give the expansion of $\mathcal{R}_{\O_\e}$ and its derivatives which will be proved in Section \ref{s4}.
In Section \ref{s5} we consider the case  $\nabla\mathcal{R}_\O(P)\ne0$ and prove
Theorem \ref{I6} and Corollary \ref{th1-2dd}. In Section \ref{s6} we consider the  case  $\nabla\mathcal{R}_\O(P)=0$ and prove Theorem \ref{th1-2} and Corollaries \ref{Ic} and \ref{th1-2tt}. In Section \ref{es} we will give an example of domains which satisfy the assumptions of Theorem \ref{th1-2}.  Finally in the Appendix we recall some known properties of the Robin function in the exterior of the ball as well as some useful identities involving the Green function.

\section{Uniform estimates on the regular part of the Green function}\label{s2}
Set
$$\Omega_\e=\Omega\backslash B_\e~\mbox{where}~B_\e=B(P,\e)~\mbox{with}~P\in \Omega.$$
Observe that $B_\e\subset \Omega$ for $\e$ sufficiently small and we will always assume this is the case. We denote by
$$B^c_\e:=\R^N\backslash B_\e$$
and  without loss of generality, we take $P=0\in \Omega$.

In this section we will prove two crucial lemmas that will be repeatedly used in the proof of our expansion of the Robin function and its derivatives. In order to clarify their role, let us write down the following representation formula for the gradient of the Robin function proved in \cite{bp}, p.170:
for $x\in \Omega_\e$ and letting $\nu_y$ be the outer unit normal to the boundary of the domain,
\begin{equation}\label{10-09-21}
\begin{split}
\nabla \mathcal{R}_{\O_\e}(x)=&\int_{\partial\Omega_\e}\nu(y)\left(\frac{\partial G_{\O_\e}(y,x)}{\partial\nu_x}\right)^2d\sigma_y=
\int_{\partial\Omega_\e}\nu(y)\left(\frac{\partial G_{\O_\e}(x,y)}{\partial\nu_y}\right)^2d\sigma_y
\\=
&\int_{\partial\Omega}\nu(y)\left(\frac{\partial G_{\O_\e}(x,y)}{\partial\nu_y}\right)^2d\sigma-\int_{\partial B_\e }\frac y\e\left(\frac{\partial G_{\O_\e}(x,y)}{\partial\nu_y}\right)^2d\sigma_y
\end{split}
\end{equation}
=\big(using the identities $G_{\O_\e}=G_{\O}+H_{\O}-H_{\O_\e}$ and $G_{\O_\e}=G_{B^c_\e}+H_{B^c_\e}-H_{\O_\e}$\big)
\begin{subequations}\nonumber
\begin{empheq}[box=\widefbox]{align}
&\underbrace{\int_{\partial\Omega}\nu(y)\left(\frac{\partial G_\O(x,y)}{\partial\nu_y}\right)^2d\sigma_y}_{=\nabla \mathcal{R}_{\O}(x)}+\int_{\partial\Omega}\nu(y)\left(\frac{\partial \big(H_{\O}(x, y)-H_{\O_\e}(x,y)\big)}{\partial\nu_y}\right)^2d\sigma_y\nonumber\\
&+2\int_{\partial\Omega}\nu(y) \frac{\partial G_\O(x,y)}{\partial\nu_y} \frac{\partial \big(H_{\O}(x, y)-H_{\O_\e}(x,y)\big)}{\partial\nu_y} d\sigma_y\nonumber
\end{empheq}
\end{subequations}
+
\begin{subequations}
\begin{empheq}[box=\widefbox]{align}
&\underbrace{-\int_{\partial B_\e }\frac y\e\left(\frac{\partial G_{B^c_\e}(x,y)}{\partial\nu_y}\right)^2d\sigma_y}_{=\nabla \mathcal{R}_{B^c_\e}(x)\hbox{ by Lemma \ref{Lap}}}-\int_{\partial B_\e }\frac y\e\left(\frac{\partial \big(H_{B^c_\e}(x, y)-H_{\O_\e}(x,y)\big)}{\partial\nu_y}\right)^2d\sigma_y\nonumber\\
&-2\int_{\partial B_\e }\frac y\e\frac{\partial G_{B^c_\e}(x,y)}{\partial\nu_y} \frac{\partial \big(H_{B^c_\e}(x, y)-H_{\O_\e}(x,y)\big)}{\partial\nu_y} d\sigma_y.\nonumber
\end{empheq}
\end{subequations}
Hence in order to estimate the integrals in the boxes we need the behavior of $H_{\O}(x, y)-H_{\O_\e}(x,y)$ on $\partial\O$ and $H_{B^c_\e}(x, y)-H_{\O_\e}(x,y)$ on $\partial B_\e$ respectively. It will be done in the next lemmas.
\begin{lem}\label{llt}
For any $x\in \Omega_\e$ and $y\in \bar \Omega_\e$ , with $|y|\ge C_0>0$ and $i=1,\cdots,N$, it holds
\begin{equation}\label{a07-24-8t}
\frac{\partial \big(H_{\Omega}(x, y)-H_{\Omega_\e}(x,y)\big)}{\partial y_i} =
\begin{cases}
O\left(\frac{\e^{N-2}}{|x|^{N-2}}\right)+O(\e)\,\,\,~&\mbox{for}~N\geq 3,\\[3mm]
O\Big(\left|
 \frac{\ln|x|}{\ln \e}\right|\Big)+
O\Big(\frac{1}{|\ln \e|}\Big)\,\,\,~&\mbox{for}~N=2.
\end{cases}\end{equation}
\end{lem}
\begin{proof}
Let us point out that the functions $H_{\Omega}(x,y)$ and  $H_{\Omega_\e}(x,y)$ are well defined if $x\in \Omega_\e$.\\
For any $y\in \bar \Omega$ with $|y|\ge C_0>0$,  we have that
\begin{equation*}
 \begin{cases}
 \Delta_x  \big(H_{\Omega}(x, y)-H_{\Omega_\e}(x,y)\big) =0 ~&\mbox{for}~x\in \Omega\backslash B_\e,\\[2mm]
H_{\Omega}(x, y)-H_{\Omega_\e}(x,y)= 0 ~&\mbox{for}~x\in \partial \Omega,\\[2mm]
H_{\Omega}(x, y)-H_{\Omega_\e}(x,y) = -  G_{\Omega} (x,y)~&\mbox{for}~x\in \partial B_\e.
 \end{cases}
 \end{equation*}
By the representation formula for harmonic function,  we obtain
\begin{equation*}
H_{\Omega}(x, y)-H_{\Omega_\e}(x,y)=\int_{\partial B_\e}\frac{\partial G_{\O_\e}(x,t)}{\partial\nu_t}G_{\Omega} (y,t)d\sigma_t.
\end{equation*}
Since $G_{\Omega} (y,t)$ has no singularity if $t\in \partial B_\e$  and $|y |\ge C_0>0$, we see that
$\nabla_y\big(H_{\Omega}(x, y)-H_{\Omega_\e}(x,y)\big)$ is well defined for $y\in \bar \Omega$ with $|y|\ge C_0>0$.\\
Now fix $y\in \Omega$ with $|y|\ge C_0>0$ and observe that
\begin{equation*}
 \begin{cases}
 \Delta_x \frac{\partial \big(H_{\Omega}(x, y)-H_{\Omega_\e}(x,y)\big)}{\partial y_i}=0 ~&\mbox{for}~x\in \Omega\backslash B_\e,\\[1mm]
\frac{\partial \big(H_{\Omega}(x, y)-H_{\Omega_\e}(x,y)\big)}{\partial y_i}= 0 ~&\mbox{for}~x\in \partial \Omega,\\[1mm]
\frac{\partial \big(H_{\Omega}(x, y)-H_{\Omega_\e}(x,y)\big)}{\partial y_i}= -\frac{\partial G_{\Omega} (x,y)}{\partial y_i}~&\mbox{for}~x\in \partial B_\e.
 \end{cases}
 \end{equation*}
 For $N\ge3$ we consider, in $\Omega_\e$, the function $b_\e(x,y):=\frac{\partial \big(H_{\Omega}(x, y)-H_{\Omega_\e}(x,y)\big)}{\partial y_i}+\frac{\e^{N-2}}{|x|^{N-2}}\frac{\partial G_{\Omega}(0,y)}{\partial y_i}$. We
 have
 \begin{equation*}
 \begin{cases}
 \Delta_x
 b_\e(x,y)=0 ~&\mbox{for}~x\in \Omega\backslash B_\e,\\[3mm]
b_\e(x,y)=\frac{\e^{N-2}}{|x|^{N-2}}\frac{\partial G_{\Omega}(0,y)}{\partial y_i}=O\left(\e^{N-2}\right) ~&\mbox{for}~x\in \partial \Omega,\\[3mm]
b_\e(x,y)=\underbrace{
 -\frac{\partial G_{\Omega} (x,y)}{\partial y_i}+\frac{\partial G_{\Omega}(0,y)}{\partial y_i}}_{=O(\e)\text{ uniformly for $|y|\ge C_0>0$}}
 ~&\mbox{for}~x\in \partial B_\e.
 \end{cases}
  \end{equation*}
So, by the maximum principle, we get that
$$\frac{\partial \big(H_{\Omega}(x, y)-H_{\Omega_\e}(x,y)\big)}{\partial y_i}+\frac{\e^{N-2}}{|x|^{N-2}}\frac{\partial G_{\Omega}(0,y)}{\partial y_i}=O\big(\e\big),$$
for any $x\in\O_\e$ which implies \eqref{a07-24-8t} for $N\ge3$.

\vskip0.2cm
 For $N=2$ we consider $b_\e(x,y):=\frac{\partial \big(H_{\Omega}(x,y)-H_{\Omega_\e}(x,y)\big)}{\partial y_i}+\frac{\ln|x|}{\ln\e}\frac{\partial G_{\Omega}(0,y)}{\partial y_i}$, so that, for any $|y|\ge C_0>0$,
 \begin{equation*}
 \begin{cases}
 \Delta_xb_\e(x,y)=0 ~&\mbox{for}~x\in \Omega\backslash B_\e,\\[3mm]
b_\e(x,y)=\frac{\ln|x|}{\ln\e}\frac{\partial G_{\Omega}(0,y)}{\partial y_i} =O\left(\frac1{|\ln\e|}\right)~&\mbox{for}~x\in \partial \Omega,\\[3mm]
b_\e(x,y)=
 -\frac{\partial G_{\Omega} (x,y)}{\partial y_i}+\frac{\partial G_{\Omega}(0,y)}{\partial y_i}=O\big(\e\big)~&\mbox{for}~x\in \partial B_\e,
 \end{cases}
  \end{equation*}
and exactly as for $N\ge3$ we get $b_\e(x,y)=O\left(\frac1{|\ln\e|}\right)$
for any $x\in\O_\e$, $|y|\ge C_0>0$. So the claim follows.
\end{proof}
Next lemma concerns the estimate for $H_{B^c_\e}(x, y)-H_{\O_\e}(x,y)$ as $y\in\partial B_\e$. Since the corresponding integrals of \eqref{10-09-21} are harder to estimate, we will need to write the leading term of the expansion as $\e\to0$.
\begin{lem}\label{imp}
For any $x\in \Omega_\e$,  $y\in \partial B_\e$ we have
\begin{equation}\label{a07-24-8tmm*}
 \nabla_y \big(H_{B^c_\e}(x, y)-H_{\Omega_\e}(x,y)\big)=\phi_\e(x,y)+
\begin{cases}
O(1)~&\mbox{for}~N\geq 3,\\[3mm]
O\Big(\frac 1{|\ln \e|}\Big)&\mbox{for}~N=2,
\end{cases}\end{equation}
where the function $\phi_\e(x,y)$ is given by
\begin{equation*}
\phi_{\e}(x,y)=
\begin{cases}
(2-N)\frac{y}{\e^2}\left(H_{\Omega}(x,0)-H_{\Omega}(0,0)\frac{\e^{N-2}}{|x|^{N-2}}\right)&\mbox{for }N\geq 3,\\[3mm]
\frac{y}{\e^2}\left[\frac1{2\pi}\left(\frac{G_{\Omega}(x,0)}{-\frac1{2\pi}\ln\e-H_{\Omega}(0,0)}-1\right)
-2\left(\nabla_y H_{\Omega}(x,0)\cdot y-\nabla_y H_{\Omega}(0,0)\cdot y\frac{\ln|x|}{\ln\e}\right)
\right]~&\mbox{for }N=2.
 \end{cases}
\end{equation*}
\end{lem}
\begin{rem}
It is possible to improve the estimate \eqref{a07-24-8tmm*} for $N\ge3$ in order to have a lower order term like $o(1)$. This can be achieved adding other suitable terms to
$\phi_{\e}$ like in the $2$-dimensional case.  However the remainder term $O(1)$ will be enough for our aims.
\end{rem}
\begin{proof}
As in the proof of Lemma \ref{llt}, we can prove that $\nabla_y\big(H_{B^c_\e}(x, y)-H_{\Omega_\e}(x,y)\big)$ is well defined for $x\in \Omega_\e$ and $y\in \partial B_\e$
Then for $N\geq 2$ and $i=1,\cdots,N$,
\begin{equation*}
\begin{cases}
\Delta_x \frac{\partial \big(H_{B^c_\e}(x, y)-H_{\Omega_\e}(x,y)\big)}{\partial   y_i}=0 ~&\mbox{for}~x\in  \Omega\backslash B_\e,\\[1mm]
\frac{\partial \big(H_{B^c_\e}(x, y)-H_{\Omega_\e}(x,y)\big)}{\partial   y_i}= -\frac{\partial  G_{ B_\e^c}(x,y)}{\partial y_i} ~&\mbox{for}~x\in \partial\Omega,\\[1mm]
\frac{\partial \big(H_{B^c_\e}(x, y)-H_{\Omega_\e}(x,y)\big)}{\partial   y_i}=0 ~&\mbox{for}~x\in \partial B_\e.
\end{cases}
\end{equation*}
Next we consider, in $\Omega_\e$, the functions $b_{\e,i}(x,y):=\frac{\partial \big(H_{B^c_\e}(x, y)-H_{\Omega_\e}(x,y)\big)}{\partial   y_i}-\phi_{\e,i}(x,y)$ for $i=1,\cdots,N$ and  $\phi_{\e,i}(x,y)$ being the $i$-th component of $\phi_{\e}(x,y)$. Then denoting by $\hat b_\e(x,y)=\big( b_{\e,1}(x,y),\dots, b_{\e,N}(x,y)$ we have
 \begin{equation*}
 \begin{cases}
 \Delta_x
 b_{\e,i}(x,y)=0 ~&\mbox{for}~x\in \Omega\backslash B_\e,\\[3mm]
b_{\e,i}(x,y)=-\frac{\partial  G_{ B_\e^c}(x,y)}{\partial y_i}-\phi_{\e,i}(x,y)~&\mbox{for}~x\in \partial \Omega,\\[3mm]
b_{\e,i}(x,y)=-\phi_{\e,i}(x,y)
 ~&\mbox{for}~x\in \partial B_\e.
 \end{cases}
  \end{equation*}
{\bf Case 1: $N=2$}
\vskip0.2cm
In this case we have, for $x\in \partial B_\e$ (recall that $y\in \partial B_\e$),
 \[ \begin{split}
 |\hat b_\e(x,y)|&=|\phi_\e(x,y)|=\frac{1}{\e}\left[-\frac1{2\pi}\underbrace{\left(\frac{-\frac1{2\pi}\ln\e-H_{\Omega}(x,0)}{-\frac1{2\pi}\ln\e-H_{\Omega}(0,0)}-1\right)}_{=O\Big(\frac {|x|}{|\ln \e|}\Big)=O\Big(\frac {\e}{|\ln \e|}\Big)}
-2\underbrace{\Big(\nabla_y H_{\Omega}(x,0)-\nabla_y H_{\Omega}(0,0)\Big)\cdot y}_{=O\Big( |x|\cdot|y|\Big)=O\Big(\e^2\Big)}
\right]\\&=O\Big(\frac {1}{|\ln \e|}\Big).
\end{split}
\]
On the other hand for $x\in \partial \Omega$, using \eqref{ader}, we have
   \[ \begin{split}
 b_{\e,i}(x,y)&=\frac{1}{2\pi}\left[\frac{y_i-x_i}{|x|^2+\e^2-2x\cdot y}-\frac{|x|^2y_i-{\e^2}x_i}{\e^2\left(|x|^2+\e^2-2 x\cdot y\right)}+\frac {y_i}{\e^2}
+2\frac {y_i}{\e^2}\frac {x\cdot y}{|x|^2}+\frac {y_i}{\e^2}O\Big(\frac{|y|}{|\ln\e|}\Big)\right]\\
&=\frac{1}{2\pi}\underbrace{\left[\frac{2\e^2|x|^2y_i+2\e^2(x\cdot y)y_i-4(x\cdot y)^2y_i}{\e^2|x|^2\left(|x|^2+\e^2-2 x\cdot y\right)}
\right]}_{=O(\e)}+O\Big(\frac{1}{|\ln\e|}\Big)=O\Big(\frac{1}{|\ln\e|}\Big).
\end{split}
\]
So by the maximum principle we get $\frac{\partial \big(H_{B^c_\e}(x, y)-H_{\Omega_\e}(x,y)\big)}{\partial   y_i}-\phi_{\e,i}(x,y)=O\Big(\frac{1}{|\ln\e|}\Big)$ which gives the claim.
\vskip0.2cm
\noindent{\bf Case 2: $N\ge3$}
\vskip0.2cm
In this case we have, for $x\in \partial B_\e$,
\[ |\hat b_\e(x,y)|=|\phi_\e(x,y)|=(N-2)\frac 1\e\underbrace{\Big|H_{\Omega}(x,0)-H_{\Omega}(0,0)\Big|}_{=O(|x|)=O(\e)}=O(1),
\]
while, when $x\in \partial \Omega$ using \eqref{ader} with $|y|=\e$,
  \[ \begin{split}
 b_{\e,i}(x,y)&=-C_N(N-2)\left[\underbrace{\frac{x_i-y_i}{|x-y|^N}}_{=O(1)}+
 \e^{-2}\underbrace{\frac{|x|^2y_i-\e^2x_i}{\big(|x|^2+O(\e)\big)^{\frac N2}}}_{=|x|^{2-N}y_i+O(\e^2)}-\frac {y_i}{\e^2}\frac 1{|x|^{N-2}}+O\Big(\e^{N-3}\Big)\right]\\
&=O(1).
\end{split}
\]
As in the previous case the maximum principle
gives the claim.
\end{proof}
Lemma \ref{llt} and Lemma \ref{imp} will allow to give sharp estimates for $\mathcal{R}_{\O_\e}$ and $\nabla\mathcal{R}_{\O_\e}$ in $\O_\e$. For what concerns $H_{B^c_{\e}}(x,y)-H_{\Omega_\e}(x,y)$ we need additional information on its second derivative. For this it will be useful to use the following known result for harmonic function.
\begin{lem}\label{lem2-3}
Let $u(x)$ is a $harmonic$ function in a domain $D\subset\R^N$, $N\ge2$ and $B(x,r)\subset\subset D$, then
\begin{equation*}
\big|\nabla u(x)\big|\leq \frac{N}{r}\sup_{\partial B(x,r)} |u(x)|.
\end{equation*}
\end{lem}
\begin{proof}
See \cite{gt}, page 22.
\end{proof}
By the previous lemma we deduce the following corollary.
\begin{cor}
For any $x\in\O_\e$ we have that,
\begin{itemize}
\item
if $|y|\ge C_0$, then
\begin{equation}\label{sec1}
\frac{\partial^2\big(H_{\O}(x, y)-H_{\O_\e}(x,y)\big)}{\partial x_j\partial y_i} =
\begin{cases}
O\Big(
 \frac{\e^{N-2}}{dist (x,\partial\O_\e)|x|^{N-2}}\Big)+O\left(\frac\e{dist (x,\partial\O_\e)}\right)&\mbox{for}~N\geq 3,\\[3mm]
O\left(\left|
 \frac{\ln|x|}{dist (x,\partial\O_\e)\ln \e}\right|\right)+
O\left(\frac{1}{dist (x,\partial\O_\e)|\ln \e|}\right)&\mbox{for}~N=2,
\end{cases}
\end{equation}
\item if $y\in\partial B_\e$, then
\begin{equation}\label{sec2}
\frac{\partial^2\big(H_{B^c_\e}(x, y)-H_{\O_\e}(x,y)\big)}{\partial x_j\partial   y_i}=
\frac{\partial\phi_{\e,i}(x,y)}{\partial x_j}+
\begin{cases}
O\left(\frac1{dist (x,\partial\O_\e)}\right)~&\mbox{for}~N\geq 3,\\[3mm]
O\left(\frac 1{dist (x,\partial\O_\e)|\ln \e|}\right)&\mbox{for}~N=2,
\end{cases}
\end{equation}
\end{itemize}
where $\phi_\e$ is the function introduced in Lemma \ref{imp}.
\end{cor}
\begin{proof}
Let $N=2$ and $|y|\ge C_0$. From \eqref{a07-24-8t} and by Lemma \ref{lem2-3} with $r=dist (x,\partial\O_\e)$ we get
\begin{equation*}
\begin{split}
 \left|\frac{\partial^2\big(H_{\O}(x, y)-H_{\O_\e}(x,y)\big)}{\partial x_j\partial y_i}\right|\le &\frac C{dist (x,\partial\O_\e)}\sup_{\partial B\left(x,dist (x,\partial\O_\e)\right)}\left|\frac{\partial\big(H_{\O}(x, y)-H_{\O_\e}(x,y)\big)}{\partial y_i}\right|\\=
&O\left(\left|
 \frac{\ln|x|}{dist (x,\partial\O_\e)\ln \e}\right|\right)+
O\left(\frac{1}{dist (x,\partial\O_\e)|\ln \e|}\right),
\end{split}
\end{equation*}
which gives the claim. In the same way we get, for $N\ge3$,
\begin{equation*}
\begin{split}
 &\left|\frac{\partial}{\partial x_j}\left(\frac{\partial\big(H_{\O}(x, y)-H_{\O_\e}(x,y)\big)}{\partial y_i}\right)\right|=O\left(
 \frac{\e^{N-2}}{dist (x,\partial\O_\e)|x|^{N-2}}+ \frac\e{dist (x,\partial\O_\e)}\right),
\end{split}
\end{equation*}
which proves \eqref{sec1} for $N\ge 3$. In the same way applying Lemma \ref{lem2-3} to the function $\nabla_y \left(H_{B^c_\e}(x, y)-H_{\O_\e}(x,y)\right)-\phi_\e(x,y)$ and using \eqref{a07-24-8tmm*},
we have  \eqref{sec2}.
\end{proof}
\section{Estimates on the Robin function $\mathcal{R}_{\O_\e}$ and its derivatives}\label{s4}
In this section we prove an asymptotic estimate for $\mathcal{R}_{\O_\e}$ and its derivatives in $\Omega_\e=\Omega\backslash B_\e$. It is worth to remark that we get  uniform estimates up to $\partial B_\e$ for $\e$ small. These allow us to find the additional critical point for  $\mathcal{R}_{\O_\e}$ (which will be actually close to $\partial B_\e$), but we believe that these estimates are interesting themselves.
There is a common strategy in the proof of the estimates both for $\mathcal{R}_{\O_\e}$ and its derivatives. We start using some representation formula and after some manipulations (as in \eqref{10-09-21}) we reduce our estimate to some boundary integrals.  Lastly we use the Lemmas of Section \ref{s2} to conclude.

\vskip0.2cm~

\subsection{Estimate of  $\mathcal{R}_{\O_\e}$}~

\vskip 0.2cm

The main result of this section is the following.
\begin{prop}\label{Re1}
We have that,  for any $x\in\O_\e$,
\begin{tcolorbox}[colback=white,colframe=black]
\begin{equation}\label{esro}
\begin{split}
&\mathcal{R}_{\O_\e}(x)=\mathcal{R}_{\O}(x)+ \mathcal{R}_{B^c_\e}(x)+\begin{cases} O\Big(\frac{\e^{N-2}}{|x|^{N-2}}\Big)+O(\e)\,\,\,~&\mbox{for}~N\geq 3,
\\[3mm]
\frac 1{2\pi}\ln\frac{|x|^2}\e+O\left(\big|\frac{\ln|x|}{\ln\e}\big|\right)~&\mbox{for}~N=2.
\end{cases}
\end{split}
\end{equation}
\end{tcolorbox}
\end{prop}
\begin{rem}\label{rorem}
If $N=2$, taking in account the explicit expression of $\mathcal{R}_{B^c_\e}$ (see \eqref{robin-esterno-palla}), we have that \eqref{esro} can be written in this way,
$$\mathcal{R}_{\O_\e}(x)=\mathcal{R}_{\O}(x)+\frac1{2\pi}
\ln\frac{|x|^2}{|x|^2-\e^2}+O\left(\left|\frac{\ln|x|}{\ln\e}\right|\right).$$
\end{rem}
\begin{proof}
Starting by \eqref{i0} and the representation formula for harmonic function we have
\begin{equation*}
H_{\O_\e}(x,t)=-\int_{\partial\Omega_\e}\frac{\partial G_{\O_\e}(x,y)}{\partial\nu_y}S(y,t)d\sigma_y
\end{equation*}
and so
\begin{equation}\label{rob}
\begin{split}
\mathcal{R}_{\O_\e}(x)=& -\int_{\partial\Omega_\e}\frac{\partial G_{\O_\e}(x,y)}{\partial\nu_y}S(x,y)d\sigma_y\\=&-\int_{\partial\Omega}\frac{\partial G_{\O_\e}(x,y)}{\partial\nu_y}S(x,y)d\sigma_y-\int_{\partial B_\e}\frac{\partial G_{\O_\e}(x,y)}{\partial\nu_y}S(x,y)d\sigma_y\\=
&\underbrace{-\int_{\partial\Omega}\frac{\partial G_{\O}(x,y)}{\partial\nu_y}S(x,y)d\sigma_y}_{=\mathcal{R}_\O(x)}-\underbrace{\int_{\partial\Omega}\frac{\partial \big(H_{\O}(x, y)-H_{\O_\e}(x,y)\big)}{\partial\nu_y}S(x,y)d\sigma_y}_{K_{1,\e}}\\
&-\underbrace{\int_{\partial B_\e}\frac{\partial G_{B_\e^c}(x,y)}{\partial\nu_y}S(x,y)d\sigma_y}_{=\hbox{ see \eqref{lap2}}}-\underbrace{\int_{\partial B_\e}\frac{\partial \big(H_{B_\e^c}(x, y)-H_{\O_\e}(x,y)\big)}{\partial\nu_y}S(x,y)d\sigma_y}_{K_{2,\e}}\\=
&\mathcal{R}_\O(x)-K_{1,\e}-K_{2,\e}+
\begin{cases}
\mathcal{R}_{B_\e^c}(x)&\hbox{ if }N\ge3,\\[3mm]
\mathcal{R}_{B_\e^c}(x) + \frac1{2\pi}\ln\frac{|x|}\e&\hbox{ if }N=2.
\end{cases}
\end{split}
\end{equation}
\noindent\underline{Computation of $K_{1,\e}$}
\vskip0.2cm
By \eqref{a07-24-8t} we get that
\begin{equation*}
\begin{split}
K_{1,\e}=&\int_{\partial\Omega}\frac{\partial \big(H_{\O}(x, y)-H_{\O_\e}(x,y)\big)}{\partial\nu_y}S(x,y)d\sigma_y\\
= & O\Big(\big|\nabla\big(H_{\O}(x, y)-H_{\O_\e}(x,y)\big)\big|\Big)\underbrace{\int_{\partial\Omega}S(x,y)d\sigma_y}_{=O(1)}\\=
&\begin{cases}
O\Big(
 \frac{\e^{N-2}}{|x|^{N-2}}\Big)+O(\e)\,\,\,~&\mbox{for}~N\geq 3,\\[3mm]
O\Big(\left|
 \frac{\ln|x|}{\ln \e}\right|\Big)+
O\Big(\frac{1}{|\ln \e|}\Big)\,\,\,~&\mbox{for}~N=2.
\end{cases}
\end{split}
\end{equation*}
\vskip0.2cm
\noindent\underline{Computation of $K_{2,\e}$}
\vskip0.2cm
First we observe that
\begin{equation}\label{new}
\begin{split}
&\frac{\partial \big(H_{B^c_\e}(x, y)-H_{\O_\e}(x,y)\big)}{\partial\nu_y}\Big|_{\partial B_\e}=\phi_\e(x,y)\cdot \frac{-y}{\e}+ \begin{cases}O(1)&\mbox{if}\ N\ge3,\\[2mm]
O\Big(\frac 1{|\ln \e|}\Big)&\mbox{if}\ N=2
\end{cases}
\\=
& \begin{cases}
 \frac{N-2}\e\left(H_{\Omega}(x,0)-H_{\Omega}(0,0)\frac{\e^{N-2}}{|x|^{N-2}}\right)+O(1)&\mbox{if}\ N\ge3,\\[2mm]
-\frac1\e\left[\frac1{2\pi}\left(\frac{G_{\Omega}(x,0)}{-\frac1{2\pi}\ln\e-H_{\Omega}(0,0)}-1\right)
-2\left(\nabla_y H_{\Omega}(x,0)\cdot y-\nabla_y H_{\Omega}(0,0)\cdot y\frac{\ln|x|}{\ln\e}\right)\right]+O\Big(\frac 1{|\ln \e|}\Big)&\mbox{if}\ N=2.
\end{cases}
 \end{split}
\end{equation}
{\bf Case $N\ge3$}.
By \eqref{new} we get that
\begin{equation*}
\begin{split}
K_{2,\e}=&\int_{\partial B_\e}\frac{\partial \big(H_{B_\e^c}(x, y)-H_{\O_\e}(x,y)\big)}{\partial\nu_y}\frac{C_N}{|x-y|^{N-2}}d\sigma_y\\=
&\left[\frac{N-2}\e\left(H_{\Omega}(x,0)-H_{\Omega}(0,0)\frac{\e^{N-2}}{|x|^{N-2}}\right)+O(1)\right]\underbrace{\int_{\partial B_\e}\frac{C_N}{|x-y|^{N-2}}d\sigma_y}_{\hbox{by \eqref{ap5}}=\frac1{N-2}\frac{\e^{N-1}}{|x|^{N-2}}}\\=
&O\left(\frac{\e^{N-2}}{|x|^{N-2}}\right),
\end{split}
\end{equation*}
 and so the claim follows for $N\ge3$.
\vskip0.2cm
{\bf Case $N=2$}.
In the same way we get by \eqref{new},
\begin{equation*}
\begin{split}
K_{2,\e}=&-\frac1{2\pi}\int_{\partial B_\e}\frac{\partial \big(H_{B_\e^c}(x, y)-H_{\O_\e}(x,y)\big)}{\partial\nu_y}\ln|x-y|d\sigma_y\\=
&\frac1{4\pi^2\e}\underbrace{\left(\frac{G_{\Omega}(x,0)}{
-\frac1{2\pi}\ln\e-H_{\Omega}(0,0)}-1\right)}_{=-1
+O\left(\big|\frac{\ln|x|}{\ln\e}\big|\right)}\underbrace{\int_{\partial B_\e}\ln|x-y|d\sigma_y}_{=2\pi\e\ln|x|\hbox{ by \eqref{ap6}}}\\
&-\frac 1{\pi\e}\left(\nabla_y H_{\Omega}(x,0)-\nabla_y H_{\Omega}(0,0)\frac{\ln|x|}{\ln\e}\right)\cdot\underbrace{\int_{\partial B_\e}y\ln|x-y|d\sigma_y}_{=O\left(\e^2\big|\ln|x|\big|\right)\hbox{ by \eqref{ap7}}}
\\& +O\left(\frac1{|\ln\e|}\right)\underbrace{\int_{\partial B_\e}\big|\ln|x-y|\big|d\sigma_y}_{=O\left(\e\big|\ln|x|\big|\right)\hbox{ by \eqref{ap7}}}= -\frac1{2\pi}\ln|x|+O\left(\Big| \frac{\ln|x|}{\ln\e}\Big|\right).
\end{split}
\end{equation*}

By \eqref{rob} and the estimates for $K_{1,\e}$ and $K_{2,\e}$ the claim follows.
\end{proof}
\vskip0.05cm~
\subsection{Estimate of  $\nabla\mathcal{R}_{\O_\e}$}~
\vskip0.2cm
The main result of this section is the following.
\begin{prop}\label{prnab}
We have that for any $x\in\O_\e$,
\begin{tcolorbox}[colback=white,colframe=black]
\begin{equation}\label{super}
\begin{split}
&\nabla \mathcal{R}_{\O_\e}(x)=\nabla \mathcal{R}_{\O}(x)+\nabla \mathcal{R}_{B^c_\e}(x)+\begin{cases} O\left(\frac{\e^{N-2}}{|x|^{N-1}}\right)+O(\e)\,\,\,~&\mbox{for}~N\geq 3,
\\[3mm]
\frac 1{\pi}\left(1-\frac {\ln |x|}{\ln \e}\right)\frac x{|x|^2}+O\left(\frac1{|x||\ln\e|}\right)~&\mbox{for}~N=2.
\end{cases}
\end{split}
\end{equation}
\end{tcolorbox}
\end{prop}
\begin{proof}
Recalling \eqref{10-09-21} we have that
\begin{equation}\label{cnab}
\begin{split}
&\nabla \mathcal{R}_{\O_\e}(x)=\nabla \mathcal{R}_{\O}(x)+\nabla \mathcal{R}_{B^c_\e}(x)\\
&+\underbrace{\int_{\partial\Omega}\nu(y)\left(\frac{\partial \big(H_{\O}(x, y)-H_{\O_\e}(x,y)\big)}{\partial\nu_y}\right)^2d\sigma_y}_{I_{1,\e}}+\underbrace{2\int_{\partial\Omega}\nu(y) \frac{\partial G_\O(x,y)}{\partial\nu_y} \frac{\partial \big(H_{\O}(x, y)-H_{\O_\e}(x,y)\big)}{\partial\nu_y} d\sigma_y}_{I_{2,\e}}\\
&-\underbrace{\int_{\partial B_\e }\frac y\e\left(\frac{\partial \big(H_{B^c_\e}(x, y)-H_{\O_\e}(x,y)\big)}{\partial\nu_y}\right)^2d\sigma_y}_{I_{3,\e}}-\underbrace{2\int_{\partial B_\e }\frac y\e\frac{\partial G_{B^c_\e}(x,y)}{\partial\nu_y} \frac{\partial \big(H_{B^c_\e}(x, y)-H_{\O_\e}(x,y)\big)}{\partial\nu_y} d\sigma_y}_{I_{4,\e}}.
\end{split}
\end{equation}
We will show that, for $N\ge3$, the integrals $I_{1,\e},\cdots,I_{4,\e}$ are lower order terms with respect to $\nabla \mathcal{R}_{\O}$ and $\nabla \mathcal{R}_{B^c_\e}$. If $N=2$ the situation is more complicated, because both integrals $I_{3,\e}$ and $I_{4,\e}$ give  a contribution.
\vskip0.2cm
\noindent\underline{Computation of $I_{1,\e}$}
\vskip0.2cm
By \eqref{a07-24-8t} we get that
\begin{equation*}
I_{1,\e}=
\begin{cases}
O\left(\frac{\e^{2(N-2)}}{|x|^{2(N-2)}}\right)+O\big(\e^2\big)\,\,\,~&\mbox{for}~N\geq 3,\\[3mm]
O\left(\left|
 \frac{\ln^2|x|}{\ln^2\e}\right|\right)+
O\Big(\frac{1}{|\ln \e|^2}\Big)\,\,\,~&\mbox{for}~N=2.
\end{cases}\end{equation*}
\vskip0.2cm\noindent
\underline{Computation of $I_{2,\e}$}
\vskip0.2cm
By \eqref{a07-24-8t} we get that
\begin{equation*}
\begin{split}
 |I_{2,\e}|=& O\Big(\big|\nabla\big(H_{\O}(x, y)-H_{\O_\e}(x,y)\big)\big|\Big)
\underbrace{\int_{\partial\Omega}\left|\frac{\partial G_\O(x,y)}{\partial\nu_y}\right|d\sigma_y}_{=1}\\=&
\begin{cases}
O\Big(
 \frac{\e^{N-2}}{|x|^{N-2}}\Big)+O\big(\e\big)\,\,\,~&\mbox{for}~N\geq 3,\\[3mm]
O\Big(\left|
 \frac{\ln|x|}{\ln \e}\right|\Big)+
O\Big(\frac{1}{|\ln \e|}\Big)\,\,\,~&\mbox{for}~N=2.
\end{cases}
\end{split}
\end{equation*}
\vskip0.2cm\noindent
\underline{Computation of $I_{3,\e}$}
\vskip0.2cm
Then we look at the cases $N\ge3$ and $N=2$ separately.
\vskip0.2cm\noindent
{\bf Case $N\ge3$}
\vskip0.2cm
By \eqref{new} we get
\begin{equation*}
\begin{split}
I_{3,\e}=&
\int_{\partial B_\e}\frac y\e\left(\frac{\partial \big(H_{B^c_\e}(x, y)-H_{\Omega_\e}(x,y)\big)}{\partial\nu_y}\right)^2d\sigma_y\\
=&(N-2)^2\int_{\partial B_\e}\frac y\e\left[\frac1\e\left(H_{\Omega}(x,0)-H_{\Omega}(0,0)\frac{\e^{N-2}}{|x|^{N-2}}\right)+O(1)\right]^2
d\sigma_y\\
\end{split}
\end{equation*}

\begin{equation*}
\begin{split}
=&\frac{(N-2)^2}{\e^2}\left(H_{\Omega}(x,0)-H_{\Omega}(0,0)\frac{\e^{N-2}}{|x|^{N-2}}\right)^2\underbrace{\int_{\partial B_\e}\frac y\e d\sigma_y}_{=0}+O\Big(\frac1{\e}\Big)\underbrace{\int_{\partial B_\e}\left|\frac y\e\right|d\sigma_y}_{=O(\e^{N-1})}\\
&+O(1)\int_{\partial B_\e}\left|\frac y\e\right|d\sigma_y=O\Big(\e^{N-2}\Big).
\end{split}
\end{equation*}

{\bf Case $N=2$}
\vskip0.2cm
Again by \eqref{new} we have that
\begin{equation*}
\begin{split}
I_{3,\e}=&
\int_{\partial B_\e}\frac y\e\left(\frac{\partial \big(H_{B^c_\e}(x, y)-H_{\Omega_\e}(x,y)\big)}{\partial\nu_y}\right)^2d\sigma_y=\int_{\partial B_\e}\frac y\e\left(\phi_\e(x,y)\cdot\nu_y+O\Big(\frac 1{|\ln \e|}\Big)\right)^2d\sigma_y\\=
&\int_{\partial B_\e}\frac y\e\left(\phi_\e(x,y)\cdot\nu_y\right)^2d\sigma_y+O\Big(\frac 1{|\ln \e|}\Big)\underbrace{\int_{\partial B_\e}|\phi_\e(x,y)|d\sigma_y}_{=O(1)}+O\Big(\frac {\e}{(\ln \e)^2}\Big)
\\=
&\int_{\partial B_\e}\frac y\e\left(\phi_\e(x,y)\cdot\nu_y\right)^2d\sigma_y+O\Big(\frac 1{|\ln \e|}\Big)=\frac{1}{4\pi^2\e^2}
\left(\frac{G_{\Omega}(x,0)}{-\frac1{2\pi}\ln\e-H_{\Omega}(0,0)}-1\right)^2\underbrace{\int_{\partial B_\e}\frac y\e d\sigma_y}_{=0}\\
&+\frac{2}{\pi\e^2}\left(1-\frac {\ln |x|}{\ln \e} + O\Big(\frac 1{|\ln \e|}\Big) \right)\int_{\partial B_\e}\frac y\e\left[\nabla_y H_{\Omega}(x,0)\cdot y-\nabla_y H_{\Omega}(0,0)\cdot y\frac{\ln|x|}{\ln\e}
\right]d\sigma_y\\
&+\underbrace{\int_{\partial B_\e}\frac {|y|}\e O(1)d\sigma_y}_{=O(\e)}+O\Big(\frac 1{|\ln \e|}\Big) \\=
&\frac{2}{\pi\e^3}\left(1-\frac {\ln |x|}{\ln \e} + O\Big(\frac 1{|\ln \e|}\Big) \right)
\sum_{j=1}^2\left( \frac{\partial H_{\Omega}(x,0)}{\partial y_j}
-\frac{\partial H_{\Omega}(0,0)}{\partial y_j}\frac{\ln|x|}{\ln\e}
\right)\underbrace{\int _{\partial B_\e}yy_j d\sigma_y}_{=\pi \e^3 \delta_j^i}+O\Big(\frac 1{|\ln \e|}\Big)\\=
&2\left(1-\frac {\ln |x|}{\ln \e} \right)\left(\nabla_y H_{\Omega}(x,0)-\nabla _y  H_{\Omega}(0,0)\frac{\ln|x|}{\ln\e}\right)+O\Big(\frac 1{|\ln \e|}\Big)\\=
&2\nabla_y H_{\Omega}(x,0)+O\left(\left|\frac {\ln |x|}{\ln \e}\right|\right)+O\Big(\frac 1{|\ln \e|}\Big).
\end{split}
\end{equation*}
\underline{Computation of $I_{4,\e}$}
\vskip0.2cm
As in the previous step let us consider the case $N\ge3$ firstly.
\vskip0.2cm\noindent
{\bf Case $N\ge3$}
\vskip0.2cm
Recalling \eqref{new} and using \eqref{ap1} and \eqref{ap2} we have
\begin{equation*}
\begin{split}
I_{4,\e}=&2\int_{\partial B_\e}\frac y\e \frac{\partial G_{B^c_\e}(x,y)}{\partial\nu_y} \frac{\partial \big(H_{B^c_\e}(x, y)-H_{\Omega_\e}(x,y)\big)}{\partial\nu_y}d\sigma_y\\
=&2\frac{N-2}{\e}\underbrace{\left(H_{\Omega}(x,0)-H_{\Omega}(0,0)\frac{\e^{N-2}}{|x|^{N-2}}\right)}_{=O(1)}\underbrace{\frac{1}{\e}\int_{\partial B_\e}y \frac{\partial G_{B^c_\e}(x,y)}{\partial\nu_y}d\sigma_y}_{=
O\left(\frac{\e^{N-1}}{|x|^{N-1}}\right)\hbox{ by }\eqref{ap2}
}\\
&+O(1)\underbrace{\int_{\partial B_\e}\frac{\partial G_{B^c_\e}(x,y)}{\partial\nu_y}d\sigma_y}_{=O\left(\frac{\e^{N-2}}{|x|^{N-2}}\right)\hbox{ by }\eqref{ap1}
}=O\left(\frac{\e^{N-2}}{|x|^{N-1}}\right).
\end{split}
\end{equation*}
{\bf Case $N=2$}
\vskip0.2cm
Using \eqref{new} we have
\begin{equation}\label{akappahat-N=2}
\begin{split}
I_{4,\e}=&2\int_{\partial B_\e}\frac y\e\frac{\partial G_{B^c_\e}(x,y)}{\partial\nu_y} \frac{\partial \big(H_{B^c_\e}(x, y)-H_{\Omega_\e}(x,y)\big)}{\partial\nu_y}d\sigma_y\\
=&2\int_{\partial B_\e}\frac{y}{\e} \frac{\partial G_{B^c_\e}(x,y)}{\partial\nu_y} \left(\phi_\e(x,y)\cdot \nu_y+O\Big(\frac 1{|\ln \e|}\Big)\right)d\sigma_y\\
=&\frac{1}{\pi\e}
\left(1-\frac {\ln |x|}{\ln \e} +O\Big( \frac 1{| \ln \e|}\Big) \right)\underbrace{\int_{\partial B_\e}\frac{y}{\e} \frac{\partial G_{B^c_\e}(x,y)}{\partial\nu_y}d\sigma_y}_{=-\frac{x}{|x|^2}\e \hbox{ by }\eqref{ap2}}
\\+
&\underbrace{\frac 4{\e}\int_{\partial B_\e}\frac{y}{\e} \frac{\partial G_{B^c_\e}(x,y)}{\partial\nu_y}\left(\nabla_y H_{\Omega}(x,0)\cdot y-\nabla_y H_{\Omega}(0,0)\cdot y\frac{\ln|x|}{\ln\e}\right)d\sigma_y}_{:=\widetilde{K}_{4,\e}}\\
&+O\Big(\frac 1{|\ln \e|}\Big)\underbrace{\int_{\partial B_\e}\frac{|y|}{\e} \left|\frac{\partial G_{B^c_\e}(x,y)}{\partial\nu_y}\right|d\sigma_y}_{=O(1) \hbox{ by }\eqref{ap1}}.
\end{split}
\end{equation}
Also, by \eqref{ap10}, we have $\displaystyle\int _{\partial B_\e}yy_j \frac{\partial G_{B^c_\e}(x,y)}{\partial\nu_y} d\sigma_y =-\e^4\frac{xx_j}{|x|^4}-\delta_j^i\frac {\e^2}2\left(1-\frac {\e^2}{|x|^2}\right)$
and then
\begin{equation}\label{bkappahat-N=2}
\begin{split}
\widetilde K_{4,\e}=&-2\left(1-\frac{\e^2}{|x|^2}\right)\left( \nabla_y H_{\Omega}(x,0)
-\nabla_y H_{\Omega}(0,0)\frac{\ln|x|}{\ln\e}
\right)\\
&-4\e^2 \frac x{|x|^4}\sum_{j=1}^2\left( \frac{\partial H_{\Omega}(x,0)}{\partial y_j}-\frac{\partial H_{\Omega}(0,0)}{\partial y_j}\frac{\ln|x|}{\ln\e}\right)x_j\\=&
-2\left( \nabla_y H_{\Omega}(x,0)
-\nabla_y H_{\Omega}(0,0)\frac{\ln|x|}{\ln\e}\right)+O\left(\frac{\e^2}{|x|^2}\right).
\end{split}
\end{equation}
Hence from \eqref{akappahat-N=2} and \eqref{bkappahat-N=2}, we find
\begin{equation*}
\begin{split}
I_{4,\e}&=-\frac 1{\pi}\left(1-\frac {\ln |x|}{\ln \e}\right)\frac x{|x|^2}-2\nabla_y H_{\Omega}(x,0)+O\left(\frac1{|x|\cdot|\ln\e|}\right).
\end{split}
\end{equation*}
Collecting the estimates for $I_{1,\e},\cdots,I_{4,\e}$ we have the expansion for  $\nabla\mathcal{R}_{\O_\e}$.
\end{proof}
\vskip0.1cm~
\subsection{Estimate of  $\nabla^2\mathcal{R}_{\O_\e}$}~
\vskip0.2cm
These estimates will be crucial to prove the uniqueness of the critical point of $\mathcal{R}_{\O_\e}$ close to $\partial B(0,\e)$.
 The basic result of this section is the following.
\begin{prop}\label{prop-p2}
 For any $i,j=1,\cdots,N$ and for $x\in\O_\e$ such that $dist(x,\partial \Omega)\geq C>0$, we have
 \begin{tcolorbox}[colback=white,colframe=black]
 \begin{equation}\label{t07-24-11b}
 \begin{split}
&\frac{\partial^2 \mathcal{R}_{\O_\e}(x)}{\partial x_i\partial x_j}=\frac{\partial^2 \mathcal{R}_\O(x)}{\partial x_i\partial x_j}+\frac{\partial^2 \mathcal{R}_{B_\e^c}(x)}{\partial x_i\partial x_j}+\\
&\begin{cases}
O\left(\frac{\e^{N-2}}{dist (x,\partial B_\e)|x|^{N-2}}+\frac{\e}{dist (x,\partial B_\e)}+\frac{\e^{N-2}}{|x|^N}+\frac{\e^{N-2}|x|}{dist(x,\partial B_\e)^{N}}+\frac{\e^{N-2}|x|^2}{dist(x,\partial B_\e)^{N+1}}\right)&\mbox{for}~~N\geq 3,\\[3mm]
\frac1\pi\left(1-\frac{\ln|x|}{\ln\e}\right)\frac{\partial}{\partial x_j}\left(\frac{x_i}{|x|^2}\right)+O\left( \frac1{|x|^2|\ln\e|}+\left|
 \frac{\ln|x|}{dist (x,\partial B_\e)\ln \e}\right|\right)+\\
O\left(\frac{|x|}{dist(x,\partial B_\e)^2|\ln \e|}+\frac{|x|^2}{dist(x,\partial B_\e)^3|\ln \e|}\right)&\mbox{for}~~N=2.
 \end{cases}
 \end{split}
\end{equation}
\end{tcolorbox}
 \end{prop}
 \begin{proof}
Differentiating formula \eqref{cnab} for $i,j=1,\cdots,N$ we have
\begin{equation*}
\begin{split}
&\frac{\partial^2 \mathcal{R}_{\O_\e}(x)}{\partial x_i\partial x_j}=\frac{\partial^2 \mathcal{R}_{\O}(x)}{\partial x_i\partial x_j}+\frac{\partial^2 \mathcal{R}_{B_\e^c}(x)}{\partial x_i\partial x_j}\\
&+2\underbrace{\int_{\partial\Omega}\nu_i(y)\left(\frac{\partial^2 \big(H_{\O}(x, y)-H_{\O_\e}(x,y)\big)}{\partial x_j\partial \nu_y}\right)\left(\frac{\partial \big(H_{\O}(x, y)-H_{\O_\e}(x,y)\big)}{\partial\nu_y}\right)d\sigma_y}_{J_{1,\e}}\\
&+2\underbrace{\int_{\partial\Omega}\nu_i(y) \frac{\partial^2 G_\O(x,y)}{\partial x_j\partial \nu_y} \frac{\partial \big(H_{\O}(x, y)-H_{\O_\e}(x,y)\big)}{\partial\nu_y} d\sigma_y}_{J_{2,\e}}+2\underbrace{\int_{\partial\Omega}\nu_i(y) \frac{\partial G_\O(x,y)}{\partial\nu_y} \frac{\partial^2 \big(H_{\O}(x, y)-H_{\O_\e}(x,y)\big)}{\partial x_j\partial\nu_y} d\sigma_y}_{J_{3,\e}}\\
&-2\underbrace{\int_{\partial B_\e}\frac{y_i}\e\left(\frac{\partial^2 \big(H_{B_\e^c}(x, y)-H_{\O_\e}(x,y)\big)}{\partial x_j\partial \nu_y}\right)\left(\frac{\partial \big(H_{B_\e^c}(x, y)-H_{\O_\e}(x,y)\big)}{\partial\nu_y}\right)d\sigma_y}_{J_{4,\e}}\\
&-2\underbrace{\int_{\partial B_\e}\frac{y_i}\e\frac{\partial^2 G_{B_\e^c}(x,y)}{\partial x_j\partial \nu_y} \frac{\partial \big(H_{B_\e^c}(x, y)-H_{\O_\e}(x,y)\big)}{\partial\nu_y} d\sigma_y}_{J_{5,\e}}-2\underbrace{\int_{\partial B_\e}\frac{y_i}\e\frac{\partial G_{B_\e^c}(x,y)}{\partial\nu_y} \frac{\partial^2 \big(H_{B_\e^c}(x, y)-H_{\O_\e}(x,y)\big)}{\partial x_j\partial\nu_y} d\sigma_y}_{J_{6,\e}}.
\end{split}
\end{equation*}
We have to estimate the integrals $J_{1,\e},\cdots,J_{6,\e}$. The computations are very similar to those of Proposition \ref{prnab}.\\
\underline{Computation of $J_{1,\e}$}
\vskip0.2cm
By Lemma \ref{llt} and \eqref{sec1} we get
\begin{equation*}
\begin{split}
& J_{1,\e}=
\begin{cases}
O\Big(
 \frac{\e^{2N-4}}{dist (x,\partial B_\e)|x|^{2N-4}}\Big)+O\left(\frac{\e^2}{dist (x,\partial B_\e)}\right) &\mbox{for}~N\geq 3,\\[3mm]
O\left(
 \frac{\ln^2|x|}{dist (x,\partial B_\e)\ln^2\e}\right) &\mbox{for}~N=2.
\end{cases}
 \end{split}
\end{equation*}
\underline{Computation of $J_{2,\e}$}\\
Here we use the assumption that $x$ satisfies $dist(x,\partial \Omega)\geq C>0$. We need it to have
\begin{equation*}
\frac{\partial^2 G_\O(x,y)}{\partial x_j\partial \nu_y}
=O(1)\, \mbox{ for}~~\,~y\in \partial \Omega.
\end{equation*}
Using Lemma \ref{llt} we immedialtely have
\begin{equation*}
\begin{split}
&\left|J_{2,\e}\right|\le\int_{\partial\Omega}\left|\frac{\partial^2 G_\O(x,y)}{\partial x_j\partial \nu_y}\right|\left|\frac{\partial \big(H_{\O}(x, y)-H_{\O_\e}(x,y)\big)}{\partial\nu_y}\right| d\sigma_y
=\begin{cases}
O\Big(
 \frac{\e^{N-2}}{|x|^{N-2}}\Big)+(\e)~&\mbox{for}~N\geq 3,\\[3mm]
O\left(\left|
 \frac{\ln|x|}{\ln \e}\right|\right)~&\mbox{for}~N=2.
\end{cases}
 \end{split}
\end{equation*}
\underline{Computation of $J_{3,\e}$}
\vskip0.2cm
Recalling that $\int_{\partial\Omega}\left|\frac{\partial G_\O(x,y)}{\partial\nu_y}\right|d\sigma_y=1$ we have
\begin{equation*}
\begin{split}
J_{3,\e}=&\int_{\partial\Omega}\left|\frac{\partial G_\O(x,y)}{\partial\nu_y}\right|\left|\frac{\partial^2 \big(H_{\O}(x, y)-H_{\O_\e}(x,y)\big)}{\partial x_j\partial\nu_y}\right| d\sigma_y\\=
&\begin{cases}
O\Big(
 \frac{\e^{N-2}}{dist (x,\partial B_\e)|x|^{N-2}}\Big)+O\left(\frac{\e}{dist (x,\partial B_\e)}\right)\,\,\,~&\mbox{for}~N\geq 3,\\[3mm]
O\left(\left|
 \frac{\ln|x|}{dist (x,\partial B_\e)\ln \e}\right|\right)~&\mbox{for}~N=2.
\end{cases}
 \end{split}
\end{equation*}
\underline{Computation of $J_{4,\e}$}
\vskip0.2cm
We will show that
\begin{equation}\label{sectild}
\begin{split}
J_{4,\e}=\begin{cases}
 O\left(\frac{\e^{N-2}}{dist (x,\partial B_\e)}\right)+O\left(\frac{\e^{2N-4}}{|x|^{N-1}}\right) ~&\mbox{for}~N\geq 3,\\[3mm]
  \frac{\partial^2H_\O(x,0)}{\partial x_j\partial y_i}+O\left(\frac1{dist (x,\partial B_\e)|\ln \e|}\right)+O\left(\big| \frac{\ln|x|}{\ln \e}\big|\right)~&\mbox{for}~N=2.
\end{cases}
 \end{split}
\end{equation}
By \eqref{sec2} we derive that
\begin{equation}\label{sti2}
\begin{split}
 &\frac{\partial^2\big(H_{\O_\e}(x, y)-H_{B^c_\e}(x,y)\big)}{\partial x_j\partial\nu_y}\Big|_{ y\in\partial B_\e}\\=
&\begin{cases}
 \frac{N-2}\e\left(\frac{\partial H(x,0)}{\partial x_j}-(2-N)H_{\Omega}(0,0)\frac {\e^{N-2}x_j}{|x|^N}\right)+O\left(\frac1{dist (x,\partial B_\e)}\right)&\mbox{for}~N\geq 3,\\[3mm]
-\frac1\e\frac1{-\ln\e-2\pi H_\O(0,0)}\frac{\partial G_\O(x,0)}{\partial x_j}+\frac 2\e \sum_{l=1}^2\left( \frac{\partial ^2H_\O(x,0)}{\partial x_j\partial y_l}-\frac{\partial H_\O(0,0)}{\partial y_l}\frac{x_j}{|x|^2\ln \e}\right)y_l+O\left(\frac 1{dist (x,\partial B_\e)|\ln \e|}\right)
~&\mbox{for}~N=2.
\end{cases}
 \end{split}
\end{equation}
Hence for $N\ge3$ we have (see \eqref{new})
\begin{equation*}
\begin{split}
J_{4,\e}=& \frac{(N-2)^2}{\e^2}
\int_{\partial B_\e}\frac{y_i}\e\left(H_{\Omega}(x,0)-H_{\Omega}(0,0)\frac{\e^{N-2}}{|x|^{N-2}}+O(\e)\right) \\&
~~~~~~~
\cdot \left(\frac{\partial H(x,0)}{\partial x_j}-(2-N)H_{\Omega}(0,0)\frac{\e^{N-2}x_j}{|x|^N}
+ O\left(\frac\e{dist (x,\partial B_\e)}\right)\right)d\s_y\\=
&\frac{(N-2)^2}{\e^2}\left(H_{\Omega}(x,0)-H_{\Omega}(0,0)\frac{\e^{N-2}}{|x|^{N-2}}\right)\left(\frac{\partial H(x,0)}{\partial x_j}-(2-N)H_{\Omega}(0,0)\frac{\e^{N-2}x_j}{|x|^{N}
}\right)
\underbrace{\int_{\partial B_\e}\frac{y_i}\e}_{=0} \\
&+\frac1{\e^2}\int_{\partial B_\e}\left(O\left(\frac\e{dist (x,\partial B_\e)}\right)+O\left(\frac{\e^{N-1}}{|x|^{N-1}}\right)\right)
=O\left(\frac{\e^{N-2}}{dist (x,\partial B_\e)}\right)+O\left(\frac{\e^{2N-4}}{|x|^{N-1}}\right),
 \end{split}
\end{equation*}
which gives \eqref{sectild} for $N\ge3$.
\vskip0.2cm\noindent
 If $N=2$ we have, again using \eqref{new},
\begin{equation*}
\begin{split}
J_{4,\e}=& \int_{\partial B_\e}\frac{y_i}{\e^3}\left[\frac1{-\ln\e-2\pi H_\O(0,0)}\frac{\partial G_\O(x,0)}{\partial x_j}-2\left(\sum _{l=1}^2\frac{\partial^2H_\O(x,0)}{\partial x_j\partial y_l}y_l-\nabla_yH_\O(0,0)\cdot y\frac{x_j}{|x|^2\ln\e}\right) \right.\\
&\left.\,\,\,+O\left(\frac\e{dist (x,\partial B_\e)|\ln \e|}\right)\right]\left[\frac1{2\pi}\left(\frac{G_{\Omega}(x,0)}{-\frac1{2\pi}\ln\e-H_{\Omega}(0,0)}-1\right)
+O(\e)\right]d\sigma_y \\
=&\frac1{2\pi\e^3}\left(\frac1{-\ln\e-2\pi H_\O(0,0)}\frac{\partial G_\O(x,0)}{\partial x_j}\right)\left(\frac{G_{\Omega}(x,0)}{-\frac1{2\pi}\ln\e-H_{\Omega}(0,0)}-1\right)\underbrace{\int_{\partial B_\e}y_i d\sigma_y}_{=0}\\
&-\frac1{\pi\e^3}\left(\frac{G_{\Omega}(x,0)}{-\frac1{2\pi}\ln\e-H_{\Omega}(0,0)}-1\right)\sum_{l=1}^2\left(\frac{\partial^2H_\O(x,0)}{\partial x_j\partial y_l}-\frac{\partial H_\O(0,0)}{\partial y_l}\frac{x_j}{|x|^2\ln\e}\right)
\underbrace{\int_{\partial B_\e}y_iy_l d\sigma_y}_{\pi\e^3\delta_i^l}\\
&-\frac1{2\pi\e^3}\left(\frac{G_{\Omega}(x,0)}{-\frac1{2\pi}\ln\e-H_{\Omega}(0,0)}-1\right)O\left(\frac\e{dist (x,\partial B_\e)|\ln \e|}\right)\underbrace{\int_{\partial B_\e}|y_i|d\sigma_y}_{=O(\e^2)} \\
&+ O(\e)\left(\frac1{-\ln\e-2\pi H_\O(0,0)}\frac{\partial G_\O(x,0)}{\partial x_j}\right)\underbrace{\int_{\partial B_\e}\frac{|y_i|}{\e^3}d\sigma_y }_{=O\left(\frac1\e\right)} \\=
&\frac1{\pi\e^3}\left(1-\frac{\ln|x|}{\ln\e}+O\left(\frac1{|x|\cdot |\ln\e|}\right)\right)\frac{\partial^2H_\O(x,0)}{\partial x_j\partial y_i}\underbrace{\int_{\partial B_\e}y_i^2d\sigma_y}_{=\pi\e^3}\\
&+O\left(\frac1{dist (x,\partial B_\e)|\ln \e|}+\Big| \frac{\ln|x|}{ \ln \e }\Big|\right)\\
=& \frac{\partial^2H_\O(x,0)}{\partial x_j\partial y_i}+O\left(\frac1{dist (x,\partial B_\e)|\ln \e|}\right)+O\left(\Big| \frac{\ln|x|}{ \ln \e }\Big|\right),
\end{split}
\end{equation*}
which proves \eqref{sectild}.
\vskip0.2cm
\noindent\underline{Computation of $J_{5,\e}$}
\vskip0.2cm
We will show that
\begin{equation*}
\begin{split}
J_{5,\e}=\begin{cases}
O\left(\frac{\e^{N-2}}{|x|^N}\right)+O\left(\frac{\e^{N-2}|x|}{dist(x,\partial B_\e)^{N}}\right)+O\left(\frac{\e^{N-2}|x|^2}{dist(x,\partial B_\e)^{N+1}}\right)
~&\mbox{for}~N\geq 3,\\[3mm]
\frac1{2\pi} \left(-1+\frac{\ln|x|}{\ln\e}+O\Big(\frac1{|\ln \e|}\Big)\right)\frac{\partial}{\partial x_j}\left(\frac{x_i}{|x|^2}\right)+\\O\left(\frac{|x|}{|\ln \e|dist(x,\partial B_\e)^2}\right)+O\left(\frac{|x|^2}{|\ln \e|dist(x,\partial B_\e)^3}\right)+O\left(\frac {\e^2}{|x|^3}\right)~&\mbox{for}~N=2.
\end{cases}
 \end{split}
\end{equation*}
Here we will use that, for $N\ge2$,
\begin{equation}\label{stG}
\begin{split}
N\omega_N\frac{\partial^2 G_{B_\e^c}(x,y)}{\partial x_j\partial \nu_y}\Big|_{|y|=\e}=&
-\frac{2x_j}{\e |x-y|^N}+N\frac{|x|^2-\e^2}{\e|x-y|^{N+2}}(x_j-y_j)\\
=&O\left(\frac{|x|}{\e |x-y|^N}\right)+O\left(\frac{|x|^2}{\e|x-y|^{N+1}}\right).
\end{split}
\end{equation}
We have that, by \eqref{new} and for $N\ge3$,
\begin{equation}\label{j5}
\begin{split}
 J_{5,\e}=&\int_{\partial B_\e}\frac{y_i}\e\frac{\partial^2 G_{B_\e^c}(x,y)}{\partial x_j\partial \nu_y}\left(\frac{N-2}\e\left(H_{\Omega}(x,0)-H_{\Omega}(0,0)\frac{\e^{N-2}}{|x|^{N-2}}\right)+O(1)\right)
d\sigma_y \\
=&\frac{N-2}{\e^2}\left(H_{\Omega}(x,0)-H_{\Omega}(0,0)\frac{\e^{N-2}}{|x|^{N-2}}\right)\underbrace{\int_{\partial B_\e}y_i\frac{\partial^2 G_{B_\e^c}(x,y)}{\partial x_j\partial \nu_y}}_{=\e^N\left(\frac {\delta^i_{j}|x|^2-Nx_ix_j}{|x|^{N+2}}\right)}
+\int_{\partial B_\e}O\left(\Big| \frac{\partial^2 G_{B_\e^c}(x,y)}{\partial x_j\partial y_i}\Big| \right)\\[1mm]
=& \left(\hbox{differentiating \eqref{ap2}, and using \eqref{stG}}\right)\\[1mm]=
&O\left(\frac{\e^{N-2}}{|x|^N}\right)+O\left(\frac{|x|}{\e}\right)\int_{\partial B_\e}\frac{1}{ |x-y|^N}d\sigma_y+O\left(\frac{|x|^2}\e\right)\int_{\partial B_\e}\frac{1}{ |x-y|^{N+1}}d\sigma_y\\
&(\hbox{and observing that, for $y\in\partial B_\e$ we get }dist(x,\partial B_\e)\le|x-y|,)\\
=&O\left(\frac{\e^{N-2}}{|x|^N}\right)+O\left(\frac{\e^{N-2}|x|}{dist(x,\partial B_\e)^{N}}\right)+O\left(\frac{\e^{N-2}|x|^2}{dist(x,\partial B_\e)^{N+1}}\right),
\end{split}
\end{equation}
which gives the claim.
If $N=2$ we have, using \eqref{sec2}
 \begin{equation*}
\begin{split}
J_{5,\e}=&-\int_{\partial B_\e}\frac{y_i}{\e^2}\frac{\partial^2 G_{B_\e^c}(x,y)}{\partial x_j\partial \nu_y}\left[\frac1{2\pi}\left(\frac{G_{\Omega}(x,0)}{-\frac1{2\pi}\ln\e-H_{\Omega}(0,0)}-1\right)
\right.\\
&\left.\,\,\,
-2\left(\nabla_y H_{\Omega}(x,0)\cdot y-\nabla_y H_{\Omega}(0,0)\cdot y\frac{\ln|x|}{\ln\e}\right)+O\Big(\frac\e{|\ln \e|}\Big)\right]d\sigma_y\\=
&-\frac1{2\pi\e^2}\underbrace{\left(\frac{G_{\Omega}(x,0)}{-\frac1{2\pi}\ln\e-H_{\Omega}(0,0)}-1
\right)}_{=-1+\frac{\ln|x|}{\ln\e}+O\left(\frac1{|\ln \e|}\right)}\underbrace{\int_{\partial B_\e}y_i\frac{\partial^2 G_{B_\e^c}(x,y)}{\partial x_j\partial \nu_y}}_{=-\frac{\partial}{\partial x_j}\left(\frac{x_i}{|x|^2}\e^2\right)}+\frac2{\e^2}\sum_{l=1}^2\frac{\partial H_{\Omega}(x,0)}{\partial x_l}\underbrace{\int_{\partial B_\e}y_iy_l\frac{\partial^2 G_{B_\e^c}(x,y)}{\partial x_j\partial \nu_y}}_{=O\left(\frac {\e^4}{|x|^3}\right)}\\
&-\frac2{\e^2}\sum_{l=1}^2\frac{\partial H_{\Omega}(0,0)}{\partial x_l}\frac{\ln|x|}{\ln\e}\underbrace{\int_{\partial B_\e}y_iy_l\frac{\partial^2 G_{B_\e^c}(x,y)}{\partial x_j\partial \nu_y}}_{=-\frac{\partial}{\partial x_j}\left(\frac {\e^2}{|x|^2}\left(\e^2\frac {x_ix_l}{|x|^2}+\frac {\delta^i_l}2(|x|^2-\e^2)\right)\right)=O\left(\frac {\e^4}{|x|^3}\right)}+O\Big(\frac1{|\ln \e|}\Big)\int_{\partial B_\e}\left|\frac{\partial^2 G_{B_\e^c}(x,y)}{\partial x_j\partial \nu_y}\right|\\[1mm]=
&\left(\hbox{differentiating \eqref{ap10} and using \eqref{stG} as in \eqref{j5}}\right)\\[1mm]
&\frac1{2\pi}\left(-1+\frac{\ln|x|}{\ln\e}+O\left(\frac1{|\ln \e|}\right)\right)\frac{\partial}{\partial x_j}\left(\frac{x_i}{|x|^2}\right)+O\left(\frac{|x|}{|\ln \e|dist(x,\partial B_\e)^2}\right)+O\left(\frac{|x|^2}{|\ln \e|dist(x,\partial B_\e)^3}\right)\\
&+O\left(\frac {\e^2}{|x|^3}\right).
\end{split}
\end{equation*}

\underline{Computation of $J_{6,\e}$}
\vskip0.2cm
We will show that
\begin{equation*}
\begin{split}
J_{6,\e}=\begin{cases}
O\left(\frac{\e^{N-2}}{|x|^{N-1}}\right)+O\left(\frac{\e^{N-2}}{dist (x,\partial B_\e)|x|^{N-2}}\right)&\mbox{for}~N\geq 3,\\[3mm]
 -\frac{\partial^2H_\O(x,0)}{\partial x_j\partial y_i}+O\left(\frac1{|x|^2\cdot| \ln\e|}\right)
+O\left(\frac1{dist (x,\partial B_\e)|\ln\e|}\right)&\mbox{for}~N=2.
\end{cases}
 \end{split}
\end{equation*}
We have, by \eqref{sti2} and $N\ge3$,
\begin{equation*}
\begin{split}
J_{6,\e}=&\int_{\partial B_\e}\frac{y_i}\e\frac{\partial G_{B_\e^c}(x,y)}{\partial\nu_y}\left(\frac{N-2}\e\left(\frac{\partial H_\O(x,0)}{\partial x_j}-(2-N)H_{\Omega}(0,0)
\frac{\e^{N-2}x_j}{|x|^{N}}
\right)+O\left(\frac1{dist (x,\partial\O_\e}\right)\right)d\sigma_y\\=
&\frac{N-2}{\e^2}\left(\frac{\partial H(x,0)}{\partial x_j}-(2-N)H_{\Omega}(0,0)
\frac{\e^{N-2}x_j}{|x|^{N}}
\right)\underbrace{\int_{\partial B_\e}y_i\frac{\partial G_{B_\e^c}(x,y)}{\partial\nu_y}d\sigma_y}_{=O\left(\frac {\e^N}{|x|^{N-1}}\right) \hbox{ by \eqref{ap2}}}\\&
+O\left(\frac1{dist (x,\partial\O_\e}\right)\underbrace{\int_{\partial B_\e}\left|\frac{\partial G_{B_\e^c}(x,y)}{\partial\nu_y}\right|d\sigma_y}_{=O\left(\frac{\e^{N-2}}{|x|^{N-2}}\right)\hbox{ by \eqref{ap1}}}=O\left(\frac{\e^{N-2}}{|x|^{N-1}}\right)+O\left(\frac{\e^{N-2}}{dist (x,\partial B_\e)|x|^{N-2}}\right),
 \end{split}
\end{equation*}
and for $N=2$,
\begin{equation*}
\begin{split}
 J_{6,\e}=&-\int_{\partial B_\e}\frac{y_i}{\e^2}\frac{\partial G_{B_\e^c}(x,y)}{\partial\nu_y}\left[\frac1{-\ln\e-2\pi H_\O(0,0)}\frac{\partial G_\O(x,0)}{\partial x_j}-2\left(\sum_{l=1}^2\frac{\partial^2H_\O(x,0)}{\partial x_j\partial y_l}y_l-\nabla_yH_\O(0,0)\cdot y\frac{x_j}{|x|^2\ln\e}\right) \right.\\ \,\,\,
&\left.+O\left(\frac\e{dist (x,\partial B_\e)|\ln \e|}\right)\right]d\sigma_y\\=
&\frac1{\e^2}\frac1{\ln\e+2\pi H_\O(0,0)}\underbrace{\frac{\partial G_\O(x,0)}{\partial x_j}}_{=O\left(\frac1{|x|}\right)}\underbrace{\int_{\partial B_\e}y_i\frac{\partial G_{B_\e^c}(x,y)}{\partial\nu_y}}_{=O\left(\frac{\e^2}{|x|}\right)\hbox{ by \eqref{ap2}}}+\frac2{\e^2}\frac{\partial^2 H_\O(x,0)}{\partial x_j\partial y_l}\underbrace{\int_{\partial B_\e}y_iy_l\frac{\partial G_{B_\e^c}(x,y)}{\partial\nu_y}d\sigma_y}_{=-\frac{\delta_i^l}2\e^2+O\left(\frac{\e^4}{|x|^2}\right)\hbox{ by \eqref{ap10}}}\\
&+ \frac2{\e^2}\nabla_yH_\O(0,0)\frac{x_j}{|x|^2\ln\e}\underbrace{\int_{\partial B_\e}y_iy_j\frac{\partial G_{B_\e^c}(x,y)}{\partial\nu_y}d\sigma_y}_{=-\frac{\delta_i^j}2\e^2+O\left(\frac{\e^4}{|x|^2}\right)\hbox{ by \eqref{ap10}}}+O\left(\frac1{dist (x,\partial B_\e)|\ln\e|}\right)\underbrace{\int_{\partial B_\e}\left|\frac{\partial G_{B_\e^c}(x,y)}{\partial\nu_y}\right|d\sigma_y}_{=1 \hbox{ by \eqref{ap1}}}\\=
&-\frac{\partial^2H_\O(x,0)}{\partial x_j\partial y_i}+O\left(\frac1{|x|^2\cdot |\ln\e|}\right)
+O\left(\frac1{dist (x,\partial B_\e)|\ln\e|}\right),
 \end{split}
\end{equation*}
which proves the claim.
\vskip0.2cm
Now we look at the cases $N=2$ and $N\ge3$ separately.
If $N\ge3$, collecting  the previous estimates we have that
\begin{equation*}
\begin{split}
& 2J_{1,\e}+2J_{2,\e}+2J_{3,\e}-2J_{4,\e}-2J_{5,\e}-2J_{6,\e} \\=& O\Big(
 \frac{\e^{2N-4}}{dist (x,\partial B_\e)|x|^{2N-4}}\Big)+O\left(\frac{\e^2}{dist (x,\partial B_\e)}\right)+O\Big(\frac{\e^{N-2}}{|x|^{N-2}}\Big)+O(\e)\\
 &+O\Big(\frac{\e^{N-2}}{dist (x,\partial B_\e)|x|^{N-2}}\Big)+O\left(\frac{\e}{dist (x,\partial B_\e)}\right)+O\left(\frac{\e^{N-2}}{dist (x,\partial B_\e)}\right)+O\left(\frac{\e^{2N-4}}{|x|^{N-1}}\right)\\
&+O\left(\frac{\e^{N-2}}{|x|^N}\right)+O\left(\frac{\e^{N-2}|x|}{dist(x,\partial B_\e)^{N}}\right)+O\left(\frac{\e^{N-2}|x|^2}{dist(x,\partial B_\e)^{N+1}}\right)+O\left(\frac{\e^{N-2}}{|x|^{N-1}}\right)+O\left(\frac{\e^{N-2}}{dist (x,\partial B_\e)|x|^{N-2}}\right)\\
=
&O\Big(\frac{\e^{N-2}}{dist (x,\partial B_\e)|x|^{N-2}}\Big)+O\left(\frac{\e}{dist (x,\partial B_\e)}\right)+O\left(\frac{\e^{N-2}}{|x|^N}\right)+O\left(\frac{\e^{N-2}|x|}{dist(x,\partial B_\e)^{N}}\right)+O\left(\frac{\e^{N-2}|x|^2}{dist(x,\partial B_\e)^{N+1}}\right).
 \end{split}
\end{equation*}
If $N=2$, collecting  the previous estimates we have that

\begin{equation*}
\begin{split}
&2J_{1,\e}+2J_{2,\e}+2J_{3,\e}-2J_{4,\e}-2J_{5,\e}-2J_{6,\e}\\=&
O\left(
 \frac{\ln^2|x|}{dist (x,\partial B_\e)\ln^2\e}\right)+
O\left(\frac{1}{dist (x,\partial B_\e)|\ln\e|^2}\right)+O\Big(\left|
 \frac{\ln|x|}{\ln \e}\right|\Big)\\
&+O\Big(\frac{1}{|\ln \e|}\Big)+
O\Big(\frac{1+\big|\ln|x|\big|}{dist (x,\partial B_\e)|\ln \e|}\Big)\cancel{-2\frac{\partial^2H_\O(x,0)}{\partial x_j\partial y_i}}
\\&+
\frac1{\pi}\left(1-\frac{\ln|x|}{\ln\e}+O\left(\frac1{|\ln \e|}\right)\right)\frac{\partial}{\partial x_j}\left(\frac{x_i}{|x|^2}\right) +O\left(\frac{|x|}{|\ln \e|dist(x,\partial B_\e)^2}\right)\\
&+O\left(\frac{|x|^2}{|\ln \e|dist(x,\partial B_\e)^3}\right)+O\left(\frac {\e^2}{|x|^3}\right)+\cancel{2\frac{\partial^2H_\O(x,0)}{\partial x_j\partial y_i}}
+O\left(\frac1{|x|^2\cdot|\ln\e|}\right)\\
=&\frac1\pi\left(1-\frac{\ln|x|}{\ln\e}\right)\frac{\partial}{\partial x_j}\left(\frac{x_i}{|x|^2}\right)+O\left(\frac1{|x|^2|\ln\e|}+\frac1{dist (x,\partial B_\e)|\ln\e|}+\frac{\ln |x|}{dist (x,\partial B_\e)|\ln\e|}\right)\\
&+O\left(\frac{|x|}{dist(x,\partial B_\e)^2|\ln \e|}+\frac{|x|^2}{dist(x,\partial B_\e)^3|\ln \e|}\right),
 \end{split}
\end{equation*}
which ends the proof.
\end{proof}
We will apply the $C^2$ estimates of $\mathcal{R}_{\O_\e}$ at the {\em critical points} of $\mathcal{R}_{\O_\e}$. This leads to the following corollary.
\begin{cor}\label{corC2}
Set
$$\mathcal{D}_{\e,q,c}=\left\{x\in\O_\e\hbox{ such that }dist(x,\partial\O)\ge C>0
\hbox{ and for some }c>0, 0<q\le1,\begin{cases}
|x|\ge c\e^q&\hbox{ if }N\ge3\\
|x|\ge c r_\e^q&\hbox{ if }N=2
\end{cases}\right\}.
$$
Then for any $i,j=1,\cdots,N$ and $x\in\mathcal{D}_{\e,q,c}$  we have that
\begin{equation}\label{3-28}
 \begin{split}
\frac{\partial^2 \mathcal{R}_{\O_\e}(x)}{\partial x_i\partial x_j}=&\frac{\partial^2 \mathcal{R}_\O(x)}{\partial x_i\partial x_j} \\
&+\begin{cases}
-\left(\frac2{N\omega_N}+o(1)\right)\frac{2(1-N)x_ix_j+\delta_i^j|x|^2}{|x|^{2N}}\e^{N-2}
+O\left(\e^{1-q}+\e^{N-2-Nq}\right)&\mbox{for}~~N\geq 3,\\[3mm]
-\frac1\pi\frac{\ln|x|}{\ln\e}\frac{-2x_ix_j+\delta_i^j
|x|^2}{|x|^4}+O\left(\frac1{r_\e^{2q}|\ln\e|}+r_\e^{1-q}\right)+o(1)&\mbox{for}~~N=2.
 \end{cases}
 \end{split}
\end{equation}
 \end{cor}
 \begin{rem}
We will apply this corollary with $q\le\frac{N-2}{2N-3}$ for $N\ge3$ and any $q<1$ for $N=2$.
 \end{rem}
\begin{proof}
First of all we have that, for $N\ge2$ and since $|x|\ge c \e^q$,
\begin{equation*}
\begin{split}
\frac{\partial^2\mathcal{R}_{B_\e^c}(x)}{\partial x_i\partial x_j}=&-\frac2{N\omega_N}\frac{2(1-N)x_ix_j+\delta_i^j(|x|^2-\e^2)}{(|x|^2-\e^2)^N}\e^{N-2}
\\=&-\left(\frac2{N\omega_N}+O\left(\frac{\e^2}{|x|^{2N}}\right)\right)\frac{2(1-N)x_ix_j+\delta_i^j|x|^2}{|x|^{2N}}\e^{N-2}.
\end{split}\end{equation*}
 Next we observe that
$dist (x,\partial B_\e)=|x|-\e\sim|x|\hbox{ since $x\in\mathcal{D}_{\e,q,c}$ and by Remark \ref{cre}}$.
So we have that \eqref{t07-24-11b} becomes
\begin{equation*}
 \begin{split}
\frac{\partial^2 \mathcal{R}_{\O_\e}(x)}{\partial x_i\partial x_j}=&\frac{\partial^2 \mathcal{R}_\O(x)}{\partial x_i\partial x_j} \\
&+ \begin{cases}
-\left(\frac2{N\omega_N}+o(1)\right)\frac{2(1-N)x_ix_j+\delta_i^j|x|^2}{|x|^{2N}}\e^{N-2}
+O\left(\e^{1-q}+\e^{N-2-Nq}\right)&\mbox{for}~~N\geq 3,\\[3mm]
-\frac1{\pi} \frac{\ln|x|}{\ln\e} \frac{\partial}{\partial x_j}\left(\frac{x_i}{|x|^2}\right)+O\left(\frac1{r_\e^{2q}|\ln\e|}\right)+O\left(r_\e^{1-q}\right)
+\underbrace{O\left(\frac{\e^2}{|x|^6}\right)}_{=o(1)}&\mbox{for}~~N=2,
 \end{cases}
 \end{split}
\end{equation*}
which gives the claim.
\end{proof}

\section{The case $\nabla\mathcal{R}_\O(0)\ne0$. Proof of  Theorem \ref{I6} and Corollary \ref{th1-2dd}.}\label{s5}
\vskip 0.2cm
We start this section with a {\em necessary condition} on the location of the critical points of $\mathcal{R}_{\O_\e}$. Basically it is a consequence of Proposition \ref{prnab}.
\begin{prop}\label{cor}Set
\begin{equation}\label{y0}
 y_0=\frac{2^\frac1{2N-3}}{(N\omega_N)^\frac1{2N-3}|\nabla \mathcal{R}_\O(0)|^\frac{2N-2}{2N-3}}\nabla \mathcal{R}_\O(0).
\end{equation}
 If $x_\e$ is a critical point of $\mathcal{R}_{\O_\e}(x)$, then for $\e\to0$ we have that either
  \begin{equation}\label{cor1}
 x_\e\to x_0\neq 0\hbox{ with }\nabla \mathcal{R}_\O(x_0)=0
 \end{equation}
 or
  \begin{equation}\label{??}
  \begin{cases}
x_\e=\big(y_0+o(1)\big)\e^\frac{N-2}{2N-3}&
\mbox{if}\ N\ge3,\\ \\
x_\e=\left[y_0+o(1)\right]r_\e&\mbox{if}\ N=2,
\end{cases}
 \end{equation}
 where $r_\e$ is defined in \eqref{re}.
 \end{prop}
 \begin{proof}
 Let $x_\e$ be a critical point of $\mathcal{R}_{\O_\e}$ and first consider $N\ge3$.
By \eqref{super} we get
  \begin{equation}\label{c3}
0=\nabla  \mathcal{R}_\O(x_\e)-\frac{2\e^{N-2}}{N\omega_N}\frac {x_\e}{\left(|x_\e|^2-\e^2\right)^{N-1}}+O\left(\frac{\e^{N-2}}{|x_\e|^{N-1}}\right)+O\Big(\e\Big).
 \end{equation}
If $x_\e\to x_0\ne0$ we have that  \eqref{cor1} holds. So let us suppose that
$$x_\e\to0.$$
By \eqref{super} we have
 \begin{equation}\label{c4}
  \begin{split}
 &\nabla\mathcal{R}_\O(0)+o(1)=\frac{2\e^{N-2}}{N\omega_N|x_\e|^{N-2}}\left[
 \frac{ {x_\e}}{|x_\e|^{N}\left(1-\frac{\e^2}{|x_\e|^2}\right)^{N-1}}+O\left(\frac 1{|x_\e|}\right)\right]
\end{split}
 \end{equation}
 and, since
  \begin{equation*}
  \begin{split}
 &
 \left|\frac{ {x_\e}}{|x_\e|^{N}\left(1-\frac{\e^2}{|x_\e|^2}\right)^{N-1}}+O\left(\frac 1{|x_\e|}\right)\right|\ge\frac1{|x_\e|^{N-1}}-\frac C{|x_\e|}\to+\infty\hbox{ if }N\ge3,
\end{split}
 \end{equation*}
 then $\lim\limits_{\e\to0}\frac{\e}{|x_\e|}=0$ and \eqref{c4} becomes
 \begin{equation*}
  \begin{split}
 &\nabla\mathcal{R}_\O(0)+o(1)=\frac{2\e^{N-2}}{N\omega_N|x_\e|^{2N-3}}\left[\frac {x_\e}{|x_\e|}+o(1)\right]\Rightarrow\\
 &\lim\limits_{\e\to0}\frac{\e^{N-2}}{|x_\e|^{2N-3}}=\frac{N\omega_N}2|\nabla\mathcal{R}_\O(0)|,
\end{split}
\end{equation*}
which gives the claim.
\vskip 0.2cm
Next we consider the case $N=2$. As for $N\ge3$ we get that, if $x_\e\to x_0\ne0$ by \eqref{super} we derive that \eqref{cor1} holds. So assume that $x_\e\to 0$. By \eqref{super} we have
\begin{equation}\label{gradiente-robin}
\begin{split}\nabla   \mathcal{R}_\O(0)+o(1)&=\frac{ x_\e }{\pi|x_\e|^2}\left( \frac{ \ln |x_\e| }{\ln \e}+\frac{\e^2 }{|x_\e|^2-\e^2}+ O\left(\frac 1{|\ln \e|}\right)\right) .
\end{split}\end{equation}
Since $\frac 1{|x_\e|}\to+\infty$ and at least one among $ \frac{ (x_\e)_i }{|x_\e|}\neq 0$, using that the two terms $\frac{ \ln |x_\e| }{\ln \e}$, $\frac{\e^2 }{|x_\e|^2-\e^2}$ have the same sign, we get that
\begin{equation}\label{4.7-bis}
 \frac{ \ln |x_\e| }{\ln \e}=o(1)\hbox{ and }\frac{\e^2 }{|x_\e|^2-\e^2}=o(1),
\end{equation}
which implies also that
\begin{equation}\label{4.7-ter}\frac{\e^2 }{|x_\e|^3}=o(1)  \hbox{ and then }\frac{\e^2 }{|x_\e|\big(|x_\e|^2-\e^2\big)}=o(1),\end{equation}
(indeed, if by contradiction $\frac{\e^2 }{|x_\e|^3}\ge k>0$ then $2-3\frac{\ln|x_\e|}{\ln\e}\le\frac{\ln k}{\ln\e}$, a contradiction as $\e\to0$.)
So that \eqref{gradiente-robin} becomes
\begin{equation}\label{c7}
\begin{split}
\pi|\nabla\mathcal{R}_\O(0)|(1+o(1))&=\frac1{|x_\e|}\frac{ \ln |x_\e| }{\ln \e}\left(1+O\left(\frac 1{|\ln |x_\e||}\right)\right)\\
&=
\frac{r_\e}{|x_\e|}\frac{\ln |x_\e| }{\ln r_\e}\left(1+o(1)\right)=
\left[\frac{\ln\frac{|x_\e|}{r_\e}}{\frac{|x_\e|}{r_\e}\ln r_\e}+\frac1{\frac{|x_\e|}{r_\e}}\right]\big(1+o(1)\big).
\end{split}
\end{equation}
Since the both terms $\frac{\ln\frac{|x_\e|}{r_\e}}{\frac{|x_\e|}{r_\e}\ln r_\e}$ and $\frac{|x_\e|}{r_\e}$ are positive and  $r_\e\to 0$, it cannot happen that $\frac{|x_\e|}{r_\e}\to 0$. Moreover, since $\nabla\mathcal{R}_\O(0)\ne0$ it is not possible that $\frac{|x_\e|}{r_\e}\to+\infty.$
Then $\frac{|x_\e|}{r_\e}\to A\in(0,+\infty)$ and by \eqref{c7} we have that $A=\frac1{\pi|\nabla\mathcal{R}_\O(0)|}$ which proves the claim.

 \end{proof}
  \begin{rem}\label{cre}
Let us point out that
\begin{equation}\label{cre1}
r_\e\in \left(\frac{1}{|\ln \e|}, \frac{1}{\sqrt{|\ln \e|}}\right).
\end{equation}
In fact, taking $g(r)=r-\frac{\ln r}{\ln \e}$, then
$$g'(r)=1-\frac{1}{r\ln \e}>0,~\mbox{for any}~r\in (0,\infty),$$
and
$$g(\frac{1}{|\ln \e|})= \frac{1-\ln |\ln \e|}{|\ln \e|}<0,~~g(\frac{1}{\sqrt{|\ln \e|}})= \frac{1}{\sqrt{|\ln \e|}} -\frac{\ln |\ln \e|}{2 |\ln \e| }>0.$$
\end{rem}
 \begin{rem}\label{Rc7} Proposition \ref{cor} implies that the critical points of $\mathcal{R}_{\O_\e}$ that converges to $0$ belong to the ball $$\mathcal{C}_\e=
\begin{cases}
B\left(y_0\e^\frac{N-2}{2N-3},\d\e^\frac{N-2}{2N-3}\right)&\mbox{if}\ N\ge3,\\[2mm]
B(y_0r_\e,\d r_\e)&\mbox{if}\ N=2.
\end{cases}$$
where $\delta>0$ is a small fixed constant.
So,  if $x_\e$ is a critical point of $\mathcal{R}_{\O_\e}$ then either
\vskip0.2cm\noindent
$x_\e\in \Omega_\e\setminus \mathcal {C}_\e$ and then $x_\e$ converge to a critical point of $\mathcal {R}_\O$
\vskip0.2cm\noindent
or
\vskip0.2cm\noindent
 $x_\e\to0$ and then $x_\e\in \mathcal{C}_\e$.
\end{rem}
In the next lemma we introduce a function $F$ which plays a crucial role in the proof of the main results.
\begin{lem}\label{G9}
Let us consider the function $F:\R^N\setminus\{0\}\to\R$ as
\begin{equation*}
F(y)=
\begin{cases}
\displaystyle\sum_{j=1}^N\frac{\partial\mathcal{R}_\O(0)}{\partial x_j}y_j-\frac{D_N}{4-2N}\frac1{|y|^{2N-4}}~&\mbox{if}~N\geq 3,\\[2mm]
\displaystyle\sum_{j=1}^2\frac{\partial\mathcal{R}_\O(0)}{\partial x_j}y_j~-\frac1\pi\log|y|&\mbox{if}~N=2,
\end{cases}
\end{equation*}
with
\begin{equation}\label{c20}
D_N=\frac2{N\omega_N}\quad\hbox{for }N\ge2.
\end{equation}
Then, the function $F$ has a unique critical point $y_0$ (see \eqref{y0}) which is non-degenerate and satisfies
\begin{equation}\label{det}
\det \big(Hess\big(F(y_0)\big)\big)=(-1)^N\frac{|\nabla \mathcal{R}_\O(0)|^\frac{N(2N-2)}{2N-3}}{D_N^\frac N{2N-3}}(3-2N)\neq 0.
\end{equation}
\end{lem}
\begin{rem}\label{rem-delta}
The non-degeneracy of $y_0$, together with \eqref{det}, implies that there exists $\bar\delta>0$ such that $\overline{B(y_0,\bar\delta)}\subset \R^N\setminus \{0\}$ and $det \big(Hess\big(F(y)\big)\big)\neq 0$ in $B(y_0,\bar\delta)$. If necessary, choosing a smaller $\bar \delta$ we can assume that $|y|\geq c>0$ for every $y\in B(y_0,\bar\delta)$. From now we fix $\delta=\bar\delta$ in the definition of  $\mathcal{C}_\e$.
\end{rem}
\begin{proof}
By a straightforward computation we have that $y_0$ is the unique critical point of $F(y)$. Next we have that, for $N\ge2$, the Hessian matrix of $F(y)$ computed at $y_0=(y_{1,0},..,y_{N,0})$ is given by
$$Hess\big(F(y_0)\big)=-D_N\frac 1{|y_0|^{2N-2}}\underbrace{\left(\delta^i_j+
(2-2N)\frac{y_{i,0}y_{j,0}}{|y_0|^{2}}\right)_{1\leq i,j\leq N}}_{A_{ij}}.
$$
Note that $\lambda_1=3-2N$ is the first eigenvalue of the matrix $A_{ij}$  with associated eigenvector $v_1=y_0$. Next, we observe that any vector of the space $X:=\{x\in \R^N, x\bot y_0\}$ is an eigenvector of the matrix $A_{ij}$ corresponding to the eigenvalue $\lambda=1$, so that $det(A_{ij})=3-2N$. Hence we have that
\begin{equation*}
\begin{split}
 \det Hess\big(F(y_0)\big)=&det\left[-\frac{|\nabla \mathcal{R}_\O(0)|^\frac{2N-2}{2N-3}}{D_N^\frac1{2N-3}}\left(\delta^i_j-(2N-2)\frac{y_{i,0}y_{j,0}}{|y_0|^2}\right)_{1\leq i,j\leq N}\right]\\=
 & (-1)^N\frac{|\nabla \mathcal{R}_\O(0)|^\frac{N(2N-2)}{2N-3}}{D_N^\frac N{2N-3}}
 \displaystyle\Pi^N_{i=1}\la_i=(-1)^N\frac{|\nabla \mathcal{R}_\O(0)|^\frac{N(2N-2)}{2N-3}}{D_N^\frac N{2N-3}}(3-2N)\neq 0,
\end{split}
\end{equation*}
which proves \eqref{det}.
\end{proof}
Now we are in position to prove Theorem \ref{I6}.
 \begin{proof}[\textbf{Proof of Theorem \ref{I6}}]
 Recall that  we assume that $P=0\in \Omega$. If $x_\e\in B(0,r)\setminus B(0,\e)$ is a critical point of $\mathcal{R}_{\O_\e}$, then,  since $B(0,r)\subset\O$ is chosen not containing any critical  point of $\mathcal{R}_\O$, by Proposition \ref{cor}, we necessarly have that $x_\e\to 0$ and the expansion in \eqref{??} holds.  Then by Proposition \ref{cor} and  Remark \ref{Rc7} it is enough to prove the existence and the uniqueness of the critical point $x_\e$ in the ball $\mathcal{C}_\e$.
\vskip 0.2cm
Let us introduce the function $F_\e:B(y_0,\delta)\to\R$ as
$$F_\e(y)=\begin{cases}
\frac1{\e^\frac{N-2}{2N-3}}\mathcal{R}_{\O_\e}\left(\e^\frac{N-2}{2N-3}y\right)~&\mbox{if}~N\geq 3,\\[3mm]
\frac1{r_\e}\mathcal{R}_{\O_\e}(r_\e y)~&\mbox{if}~N=2,
\end{cases}
$$
where $y_0$ is defined in \eqref{y0} and $\delta$ is chosen as in Remark \ref{rem-delta}. Let us show that
$$\nabla F_\e\to \nabla F~\hbox{ in }~ C^1\big(\overline{B(y_0,\delta)}\big),$$
where $F$ is the function defined in Lemma \ref{G9}. Indeed, using Proposition \ref{prnab} we have that
 \begin{equation*}
\nabla F_\e(y)=\nabla \mathcal{R}_{\O_\e}(\e^\frac{N-2}{2N-3}y) =\nabla  \mathcal{R}_\O(0)-\frac2{N\omega_N}\frac y{\left(|y|^2+o(1)\right)^{N-1}}+o(1),
\end{equation*}
which gives the uniform convergence of $\nabla F_\e$ to $\nabla F$ in $B(y_0,\delta)$.

\vskip 0.1cm

Concerning $N=2$ we have that, again by Proposition \ref{prnab}, \eqref{re} and since $\frac {\e^k}{r_\e}=o(1)$ for any $k>0$ by Remark \ref{cre},
\begin{equation*}
\begin{split}
\nabla F_\e(y)=&\nabla \mathcal{R}_{\O_\e}(r_\e y) =\nabla  \mathcal{R}_\O(0)+o(1)-\frac y{\pi|y|^2}\frac{ \ln |y|+\ln r_\e }{r_\e\ln \e}-\underbrace{\frac{\e^2y }{\pi r_\e^3\big(|y|^2+o(1)\big)|y|^2}}_{=o(1)}\\=
&\nabla  \mathcal{R}_\O(0)-\frac y{\pi|y|^2}\left(1+
\frac{ \ln |y|}{\ln r_\e}\right)+o(1),
\end{split}
\end{equation*}
which gives the uniform convergence of $\nabla F_\e$ to $\nabla F$ in $B(y_0,\delta)$ for $N=2$.

\vskip 0.1cm
Let us show the $C^1$ convergence.
By Remark \ref{rem-delta} we can apply Corollary \ref{corC2} with $q=\frac{N-2}{2N-3}$ and a suitable $c>0$  such that, for $N\ge3$
 \begin{equation*}
 \begin{split}
\frac{\partial^2 F_\e(y)}{\partial x_i\partial x_j}=&\e^\frac{N-2}{2N-3}\frac{\partial^2 \mathcal{R}_{\O_\e}(\e^\frac{N-2}{2N-3}y)}{\partial x_i\partial x_j}\\=& o(1)
+\e^\frac{N-2}{2N-3}\frac{\partial^2  \mathcal{R}_{B^c_\e}(\e^\frac{N-2}{2N-3}y)}{\partial x_i\partial x_j} +\underbrace{O(\e)+O\left( \e^\frac{(N-2)^2}{2N-3}\right)}_{=o(1)}\\=
&\underbrace{2C_N(2-N)}_{=-\frac2{N\omega_N}}\left(\frac{\delta^i_j}{|y|^{2N-2}}+(2-2N)\frac{y_iy_j}{|y^{2N}}\right)+o(1),
 \end{split}
\end{equation*}
which gives the claim. In the same way, if $N=2$,  using \eqref{re} and again the Corollary \ref{corC2} with any  $q<1$ and a suitable $c>0$,
\begin{equation*}
 \begin{split}
\frac{\partial^2 F_\e(y)}{\partial x_i\partial x_j}=& r_\e\frac{\partial^2 \mathcal{R}_{\O_\e}(r_\e y)}{\partial x_i\partial x_j}=-\frac1\pi\frac{\ln|y|+\ln r_\e}{r_\e\ln\e}\frac{-2y_iy_j+\delta_i^j|y|^2}{|y|^4}+O\left(\frac{r_\e^{2-2q}}{r_\e|\ln\e|}\right) \\=
&-\frac1\pi\left(1+\frac{\ln|y|}{\ln r_\e}\right)\frac{-2y_iy_j+\delta_i^j|y|^2}{|y|^4}+O\left(\frac{r_\e^{2-2q}}{|\ln r_\e|}\right)\\=& \big(\hbox{since }r_\e\to 0\big)=
-\frac1\pi\frac{-2y_iy_j+\delta_i^j|y|^2}{|y|^4}+o(1),
 \end{split}
\end{equation*}
which gives the claim.
\vskip 0.1cm
Finally, the $C^1$ convergence of $\nabla F_\e$ to $\nabla F$ and \eqref{det} gives that
$$deg\left(\nabla F_\e, B(y_0,\delta),0\right)=deg\left(\nabla F, B(y_0,\delta),0\right)\neq 0,$$
which, jointly with the non-degeneracy of $y_0$, implies the existence and uniqueness of a  critical point $y_\e\in B(y_0,\delta)$ of $F_\e$. By the definition of $F_\e$ this implies  the existence of a unique critical point $x_\e$ for $\mathcal{R}_{\Omega_\e}$ in $\mathcal{C}_\e$.  Finally, by the definition of $\mathcal{C}_\e$, $x_\e\to 0$ and by  \eqref{??} of Proposition \ref{cor}  we get \eqref{11-1-01}.
 Moreover
\[
\begin{split}
 index_{x_\e}\big(\mathcal{R}_{\Omega_\e}(x)\big)&=index_{y_\e}\big(F_\e(y)\big)=
sgn\Big(\det Jac\big(F_\e(y_\e)\big)\Big)\\
&= sgn\Big(\det Jac\big(F(y_0)\big)\Big)=(-1)^{N+1}.
\end{split}\]
\vskip 0.2cm
We end the proof showing that $\mathcal{R}_{\O_\e}(x_\e)\to\mathcal{R}_\O(0)$. By Proposition \ref{Re1} we have that, for $N\ge3$,
\begin{equation}\label{conv1}
\begin{split}
\mathcal{R}_{\O_\e}(x_\e)=&\underbrace{\mathcal{R}_{\O}(x_\e)}_{=\mathcal{R}_{\O}(0)+o(1)}+ \mathcal{R}_{B^c_\e}(x_\e)+
O\underbrace{\left(\frac{\e^{N-2}}{|x_\e|^{N-2}}\right)}_{=O\left(\e^\frac{(N-2)(N-1)}{2N-3}\right)}+O(\e)
\\=
&\mathcal{R}_{\O}(0)+\underbrace{C_N\frac{\e^{N-2}}{\big(|x_\e|^2-\e^2\big)^{N-2}}}_{=O\left(\e^\frac{N-2}{2N-3}\right)}+o(1)=\mathcal{R}_{\O}(0)+o(1),
\end{split}
\end{equation}
which gives the claim. For $N=2$ we have, by Remark \ref{rorem},
\begin{equation}\label{conv2}
\begin{split}
&
\mathcal{R}_{\O_\e}(x_\e)=\underbrace{\mathcal{R}_{\O}(x_\e)}_{=\mathcal{R}_{\O}(0)+o(1)}+\underbrace{\frac1{2\pi}\ln\left(1-\frac{\e^2}{|x_\e|^2}\right)}_{=\frac1{2\pi}\ln\big(1+o(1)\big)\hbox{ by \eqref{4.7-bis}}}+\underbrace{O\left(\left|\frac{\ln |x_\e|}{\ln\e}\right|\right)}_{=o(1)\hbox{ by \eqref{4.7-bis}}}=\mathcal{R}_{\O}(0)+o(1),
\end{split}
\end{equation}
which gives the claim.
\end{proof}
\vskip0.2cm

Next we prove Corollary \ref{th1-2dd}.
\vskip0.2cm
\begin{proof}[\bf{Proof of  Corollary \ref{th1-2dd}}]
\vskip 0.2cm
Set ${\mathcal B}=\Big\{x\in\O \hbox{ such that }\nabla \mathcal{R}_\O(x)=0\Big\}$.
Since $\nabla \mathcal{R}_\O(P)\ne0$ we get that ${\mathcal B}\bigcap B(P,r)=\varnothing$ for any small fixed $r>0$.
Now we write $\Omega_\varepsilon=\mathcal{C}_1\bigcup \mathcal{C}_2 \bigcup \Big(B(P,r)\backslash B(P,\e)\Big)$ with $$\mathcal{C}_1=\{x,~dist~(x,{\mathcal B})\leq r\}~\mbox{and}~\mathcal{C}_2:=\Omega_\varepsilon\backslash \big(\mathcal{C}_1\cup B(P,r)\big),$$
where $r$ is such that $det\big(Hess\big(\mathcal {R}_\Omega(x)\big)\big)\neq 0$ in $\mathcal{C}_1$.

By Proposition \ref{prnab} and Corollary \ref{corC2} we have that $\mathcal{R}_{\O_\e}\to\mathcal{R}_\O$ in $\mathcal{C}_1$ and so by the choice of $r$, the non-degeneracy of the critical points of $\mathcal{R}_\O$ and the $C^2$ convergence of $\mathcal{R}_{\Omega_\e}$ to $\mathcal{R}_\O$ in $\mathcal{C}_1$ we have
$$\sharp\big\{\mbox{critical points of $\mathcal{R}_{\O_\e}(x)$ in $\mathcal{C}_1$}\big\}=\sharp\{{\mathcal B}\},$$
while the $C^1$ convergence of $\mathcal{R}_{\Omega_\e}$ to $\mathcal{R}_\O$ in $\overline{\Omega_\e\setminus B(P,r)}$ gives
$$\sharp\big\{\mbox{critical points of $\mathcal{R}_{\O_\e}(x)$ in $\mathcal{C}_2$}\big\}=0.$$
Finally from Theorem \ref{I6}, we get that
$$\sharp\big\{\mbox{critical points of $\mathcal{R}_{\O_\e}(x)$ in $B(P,r)\backslash B(P,\e)$}\big\}=1,$$
which proves the claim.
\end{proof}
\section{The case $\nabla\mathcal{R}_\O(P)=0$, proof of  Theorem \ref{th1-2}.}\label{s6}
As in the previous sections we assume that $P=0$. We will follow the line of the proof of Theorem \ref{I6}. So we start with
a necessary condition satisfied by the critical points of $\mathcal{R}_{\O_\e}(x)$.
\begin{prop}\label{d1}
If $0$ is a non-degenerate critical point for  $ \mathcal{R}_\O(x) $ and $x_\e$ is a critical point of $\mathcal{R}_{\O_\e}(x)$, then for $\e\to0$ we have that either
  \begin{equation*}
 x_\e\to x_0\ne0\hbox{ with }\nabla \mathcal{R}_\O(x_0)=0
 \end{equation*}
 or
\begin{equation}\label{d11-16-01}
x_\varepsilon=
\begin{cases}
\left(\frac{ \big(2+o(1)\big)}{N\omega_N\la}\right)^\frac{1}{2N-2}
\varepsilon^{\frac{N-2}{2N-2}}v~&\mbox{if}~N\geq 3,\\[3mm]
\widehat r_\e\Big(1+o(1)\Big)v_l~&\mbox{if}~N=2,
\end{cases}
\end{equation}
where $\lambda$ is a positive eigenvalue of the Hessian matrix $Hess\big(\mathcal{R}_\O(0)\big)$,  $v$ is an associated eigenvector with $|v|=1$ and $\widehat r_\e$ is the
 unique solution of
 \begin{equation}\label{d11-06-01b}
  r^2-\frac{\ln r}{\lambda \pi \ln \e}=0~~\mbox{in}~(0,\infty).
 \end{equation}
\end{prop}
 \begin{rem}
It is immediate to check that \eqref{d11-06-01b} admits only one solution $\widehat r_\e$ verifying $\widehat r_\e\in \Big(\frac{1}{(\lambda \pi|\ln \e|)^{\frac{1}{2}}}, \frac{1}{(\lambda \pi|\ln \e|)^{\frac{1}{4}}}\Big)$. This can be seen observing that
 $r_{\e,1}=\frac{1}{(\lambda \pi|\ln \e|)^{\frac{1}{2}}}$ and $r_{\e,2}= \frac{1}{(\lambda \pi|\ln \e|)^{\frac{1}{4}}}$ satisfy
$$h(r_{\e,1})= \frac{1-\ln \sqrt{\lambda \pi|\ln \e|}}{\lambda \pi|\ln \e|}<0,~~h(r_{\e,2})= \frac{1}{\sqrt{\lambda \pi|\ln \e|}} -\frac{\ln (\lambda \pi|\ln \e|)}{4\lambda \pi|\ln \e|}>0.$$
 \end{rem}
 \begin{rem}
Unlike when $\nabla\mathcal{R}_\O(P)\ne0$, here in general we do not have the uniqueness of the limit point of $x_\e$. It depends on the simplicity of the eigenvalue $\la$. This will play a crucial role in the next results.
 \end{rem}

\begin{proof}[Proof of Proposition \ref{d1}]
The first assertion follows arguing as in Proposition \ref{cor} (see \eqref{c3} or \eqref{super}).
On the other hand, if $x_\e\to 0$, using \eqref{c4} and \eqref{gradiente-robin} for $N\ge3$ and \eqref{4.7-bis}-\eqref{4.7-ter} for $N=2$, we get again
$$\frac{\e}{|x_\e|}\to 0.$$
Furthermore from Proposition \ref{prnab} and  $\nabla \mathcal{R}_\O(0)=0$, we get
\begin{equation}\label{11-16-01}
0=\sum^N_{j=1}\left(\frac{\partial^2 \mathcal{R}_\O(0)}{\partial x_i\partial x_j}+o(1)\right)x_{j,\e}+
 \begin{cases}
 -\frac{2x_{i,\e}\e^{N-2}}{N\omega_N |x_{\e}|^{2N-2}}\Big(1+o(1)\Big) &~\mbox{for}~N\geq 3,\\[4mm]
 \frac{x_{i,\e}\ln |x_\e|}{\pi |x_{\e}|^2|\ln \e|}\Big(1+o(1)\Big) &~\mbox{for}~N=2.
\end{cases}
\end{equation}
From \eqref{11-16-01}, we immediately get that, as $\e\to0$,
\begin{equation*}
\begin{cases}
 \frac{2 \e^{N-2}}{N\omega_N |x_{\e}|^{2N-2}} \to\lambda ~&\mbox{for}~N\geq 3,\\[2mm]
 -\frac{\ln |x_\e|}{\pi |x_{\e}|^2|\ln \e|}\to\lambda ~&\mbox{for}~N=2,
\end{cases}\end{equation*}
where  $\lambda$ is a positive eigenvalue of the hessian matrix $Hess\big(\mathcal{R}_\O(0)\big)$.
Hence \eqref{d11-16-01} follows by dividing \eqref{11-16-01} by $|x_\e|$ and passing to the limit.
\end{proof}
Analogously  to the previous section let us introduce ``the limit function'' of a suitable rescaling of  $\mathcal{R}_{\O_\e}$.
\begin{lem}
Let us consider the function $\widehat F:\R^N\setminus\{0\}\to\R$ as
\begin{equation}\label{m1}
\widehat F(y)=
\begin{cases}
\frac12\displaystyle\sum^N_{i,j=1}\frac{\partial^2 \mathcal{R}_\O(0)}{\partial x_i\partial x_j}y_iy_j -\frac{D_N} {(4-2N)|y|^{2N-4}}\,\,\,~&\mbox{for}~N\geq 3,\\[4mm]
\frac12\displaystyle\sum^2_{i,j=1}\frac{\partial^2 \mathcal{R}_\O(0)}{\partial x_i\partial x_j}y_iy_j -D_2 \ln |y| \,\,\,~&\mbox{for}~N=2,
\end{cases}\end{equation}
where $D_N$ ($N\geq 2$) is the same as in \eqref{c20}. Suppose that $Hess\big(\mathcal{R}_\O(0)\big)$ has $m\le N$ positive  eigenvalues
$0<\lambda_1\leq \lambda_{2} \leq \cdots \leq \lambda_m$. Then we have that
\begin{equation}\label{B26}
\hbox{if  $\nabla\widehat{F}(\bar{y})=0$  then } \bar{y}=\bar{y}^{(l)}=
\left(\frac{D_N}{ \lambda_l}\right)^{\frac{1}{2N-2}}v^{(l)}~\mbox{for some}~l\in\{1,\cdots,m\},
\end{equation}
where $v^{(l)}$ is an eigenvector of the matrix $\left(\frac{\partial^2 \mathcal{R}_\O(0)}{\partial x_i\partial x_j}\right)_{i,j=1,..N}$ associated to  $\lambda_l$ such that $|v^{(l)}|=1$.\\
Moreover it holds
\begin{equation}\label{B27}
\det \big(Hess\left(\widehat{F}\big(\bar y^{(l)}\big)\right)\big)=(2N-2)\la_l\prod_{s=1\atop s\ne l}^N(\lambda_s-\lambda_l).
\end{equation}
\end{lem}
\begin{rem}\label{rsim}
If the eigenvalue $\la_l$ is {\bf simple} we get two corresponding eigenvector $v_\pm^{(l)}$ with $|v_\pm^{(l)}|=1$ and then two critical points $\bar{y}_\pm^{(l)}$ in \eqref{B26}. Moreover by \eqref{B27} $\bar{y}_\pm^{(l)}$ are  non-degenerate critical points
and then it is possible to select $\delta^{(l)}>0$ such that in $B\big(\bar{y}_\pm^{(l)},\delta^{(l)}\big)$ we have $det \big(Hess\big(\widehat{F}\big(y\big)\big)\big)\neq 0$.
\end{rem}

\begin{proof}
Observe that  $\nabla\widehat{F}(y)$ is given by
\begin{equation}\label{F1}
\frac{\partial\widehat{F}(y)}{\partial y_i}=\sum^N_{j=1}\frac{\partial^2 \mathcal{R}_\O(0)}{\partial x_i\partial x_j}y_j -\frac{D_Ny_i} {|y|^{2N-2}}\,\,\,~\mbox{for}~N\geq 2
\end{equation}
and then if $\nabla\widehat{F}(y)=0$ we immediately get that $\frac{D_N} {|y|^{2N-2}}=\la_l$ for some $l\in\{1,\cdots,m\}$ proving \eqref{B26}.

The claim \eqref{B27} will be proved by diagonalizing the matrix $Hess\big(\mathcal{R}_\O(0)\big)$. Here we consider the case $N\ge3$ ($N=2$ can be handled  in the same way). Let $\textbf{P}$ the orthogonal matrix such that
$$\textbf{P}^THess\big(\mathcal{R}_\O(0)\big)\textbf{P}=diag(\la_1,\cdots,\la_N).$$
Taking $Z=\textbf{P}^T(y)$ we get that the system $\nabla\widehat{F}(y)=0$ becomes
\begin{equation*}
\la_iZ_i-D_N\frac{Z_i}{|Z|^{2N-2}}=0,~~\mbox{for}~N\geq 2\hbox{ and }i=1,\cdots,N.
\end{equation*}
Note that these zeros are critical point of the function
\begin{equation}\label{m2}
T(y)=\frac12\sum^N_{j=1}\lambda_jy_j^2-\frac{D_N} {(4-2N)|y|^{2N-4}}\,\,\,~\mbox{for}~N\geq 3.
\end{equation}
Next step is the computation of the determinant of the hessian matrix of  $\widehat F$. We know
\begin{equation*}
\begin{split}
\det Hess\left(\widehat{F}\big(y^{(l)}\big)\right)=&
\det \left( \textbf{P}^THess\Big(T
\big(y^{(l)}\big) \Big)\textbf{P}\right)= \det  \left( Hess\Big(T
\big(y^{(l)}\big) \Big) \right)\\=&
\det \left( diag\big(\lambda_1,\cdots,\lambda_N\big)+
\Big(-\frac{D_N}{|y^{(l)}|^{2N-2}}\delta_{ik}+ (2N-2) \frac{D_Ny^{(l)}_iy^{(l)}_k}{|y^{(l)}|^{2N}} \Big)_{1\leq i,k\leq N}
\right)\\=&
\det \left(\underbrace{diag\big(\lambda_1-\lambda_l,\cdots,\lambda_N-\lambda_l\big)+
(2N-2) \lambda_l  \Big(  \frac{y^{(l)}_iy^{(l)}_k}{|y^{(l)}|^{2}} \Big)_{1\leq i,k\leq N}}_{:=\bf{M_0}}
\right).
\end{split}
\end{equation*}
By the basic theory of linear algebra, we can find that the eigenvalues of
$\bf{M_0}$ are $\lambda_s-\lambda_l$ for $s=1,\cdots,l-1,l+1,\cdots,N$ and $(2N-2)\lambda_l$.
Hence we have
\begin{equation*}
\det Hess\left(\widehat{F}\big(y^{(l)}\big)\right)=(2N-2)\la_l\prod_{s=1\atop s\ne l}^N(\lambda_s-\lambda_l).
\end{equation*}
\end{proof}
Assume that $\la_l$ is a {\em simple} eigenvalue. Then, following the notations of Remark \ref{rsim},
analogously to the previous section we define
 $$\mathcal{C}_{\e,\pm}^{(l)}=
\begin{cases}
B\left(y_\pm^{(l)}\e^\frac{N-2}{2N-2},\d^{(l)}\e^\frac{N-2}{2N-2}\right)&\mbox{if}\ N\ge3,\\[2mm]
B(y_\pm^{(l)}\widehat r_\e,\d^{(l)}\widehat r_\e)&\mbox{if}\ N=2,
\end{cases}$$
and
\begin{equation*}
\widehat{\mathcal{C}_\e}=\underset{l}\cup\ \mathcal{C}_{\e,\pm}^{(l)}.
\end{equation*}
Proposition \ref{d1} implies that, under the assumption that all the eigenvalues of $\left(\frac{\partial^2 \mathcal{R}_\O(0)}{\partial x_i\partial x_j}\right)_{i,j=1,..N}$ are positive, the critical points $x_\e$  verify either
\vskip0.2cm\noindent
$x_\e\in \Omega_\e\setminus\widehat{\mathcal{C}_\e}$ and then $x_\e$ converge to a critical point of $\mathcal {R}_\O$
\vskip0.2cm\noindent
or
\vskip0.2cm\noindent
 $x_\e\to0$ and then $x_\e\in\widehat{\mathcal{C}_\e}$.

\begin{proof}[\textbf{Proof of Theorem \ref{th1-2}}]
 Recall that  we assume that $P=0\in \Omega$. As in the proof of Theorem \ref{I6}, if $x_\e\in B(0,r)\setminus B(0,\e)$ is a critical point of $\mathcal{R}_{\O_\e}$, then,  since $B(0,r)\subset\O$ is chosen not containing any critical  point of $\mathcal{R}_\O$ different from $0$, by Proposition \ref{d1} we necessarily have that $x_\e\to 0$.  Following the notation of Remark \ref{rsim}, denote by $\bar{y}_+^{(l)}$ the critical point of $\widehat{F}$ associated to $\bar v_+^{(l)}$ and let us show the existence of $one$ critical point of $\mathcal{R}_{\O_\e}$ in $B(0,r)$ which by \eqref{d11-16-01} verifies \eqref{11-1-02}. The same will hold for $\bar{y}_-^{(l)}$, giving the proof of the first claim of Theorem  \ref{th1-2}.
\vskip 0.2cm
Let us introduce the function $\widehat{F}_\e:B\left(\bar{y}_+^{(l)},\delta^{(l)}\right)\to\R$ as
$$\widehat{F}_{\varepsilon}(y)=
\begin{cases}
\frac{1}{\e^\frac{N-2}{2N-2}} \mathcal{R}_{\O_\e}\left(\e^\frac{N-2}{2N-2}y\right)~&\mbox{if}~N\geq 3,\\[2mm]
\frac{1}{\widehat r_\e} \mathcal{R}_{\O_\e}\left(\widehat r_\e y\right)~&\mbox{if}~N=2.
\end{cases}
$$
Furthermore, using Corollary \ref{corC2} with $q=\frac{N-2}{2N-2}$ for $N\ge3$ and $q<1$ for $N=2$ and arguing as in the proof of Theorem \ref{I6} we get that
$$\nabla \widehat{F}_\e\to \nabla \widehat{F}~\mbox{in} ~C^1\left(B\left(\bar{y}_+^{(l)},\delta^{(l)}\right)\right).$$
Since $\bar y_+^{(l)}$ is a  non-degenerate critical point of  $\widehat{F}$ then $\nabla\widehat{F}_\varepsilon(y)$ admits a unique critical point $y^{(l)}_{\e,+}\to y^{(l)}_+$ in $B\left(\bar{y}_+^{(l)},\delta^{(l)}\right)$ and also $y^{(l)}_{\e,+}$ is a non-degenerate critical point of $\mathcal{R}_{\O_\e}(0)$.

Finally, \eqref{11-1-02} follows by \eqref{d11-16-01} and \eqref{Irob} can be proved repeating step by step the proof of \eqref{conv1}--\eqref{conv2} in Theorem \ref{I6}.
\vskip 0.2cm\noindent
Next let us assume that all the eigenvalues of the Hessian matrix $Hess\big(\mathcal{R}_\O(0)\big)$ are simple. Again by Proposition \ref{d1}, Remark \ref{rsim} and the discussion  above, we have that if  $x_\e$ is a critical point belonging to $B(0,r)\setminus B(0,\e)$  then $x_\e\in\widehat{\mathcal{C}_\e}$. The simplicity of the eigenvalue and the previous claim prove \eqref{In}.
\end{proof}

\begin{proof}[Proof of Corollary \ref{Ic}]
In this case the Robin function $\mathcal{R}_\O(x)$ has only one critical point $P=0$ which is a non-degenerate minimum point (see \cite{g2,ct}). This means that $Hess\big(\mathcal{R}_\O(0)\big)$ has $N$ positive eigenvalues
and, if they are simple, the Robin function $\mathcal{R}_{\O\setminus B(0,\e)}(x)$ has exactly $2N$ critical points for $\e$ small enough. So the claim follows by Theorems \ref{I6} and \ref{th1-2}.
\end{proof}
 \vskip 0.2cm
\begin{rem}[What about multiple eigenvalues of $Hess\big(\mathcal{R}_\O(0)\big)$?]
~\vskip 0.2cm
In this case we are not able to give a complete description of the critical points of
$\mathcal{R}_{\O_\e}$.
Suppose that we have a multiple eigenvalue satisfying
$$\lambda_j=\lambda_{j+1}=\cdots=\lambda_{j+k},\quad j,k\ge1.$$
In this case, the function $T(y)$ defined in \eqref{m2} admits a manifold of critical points given by
$$S^k=\left\{x_j^2+\cdots+x_{j+k}^2=\left(\frac{D_N}{ \lambda_j}\right)^{\frac{1}{2N-2}}\right\}.$$
By \eqref{B27} we have that the Hessian matrix is non degenerate in the directions different from $x_j,\cdots,x_{j+k}$ and so it is a $non-degenerate$ manifold of critical points for $N$ in the sense of Morse-Bott theory. However, even in this explicit case, it seems hard to get existence results of critical points for $T(y)$ under non radial small perturbations. Of course no possible information can be deduced about the non-degeneracy.

Without additional assumptions,  as pointed out in the introduction,  it is even possible to have {\em infinitely many} solutions (the radial case). For these reasons the case of multiple eigenvalues is unclear.
\end{rem}

Before the close of this section, we give a  partial result when $\O$ is a symmetric domain.
\begin{teo}\label{th1-2a}
Let $\O\subset\R^N$ be convex and symmetric with respect $x_1,\cdots,x_N$ for $N\ge2$.
Then we have that $\mathcal{R}_{\O_\e}(x)$
has at least $2N$ critical points which are located on the coordinate axis, i.e.
$$x_{i,\e}^\pm= \Big(\underbrace{0,\cdots,0}_{i-1},\pm \big(\frac{2}{N\omega_N \lambda_i}+o(1)\big)^{\frac{1}{N-2}}\e^{\frac{N-2}{2N-2}},\underbrace{0,\cdots,0}_{N-i}\Big)
~~~~\mbox{for}~~~~i=1,\cdots,N.$$
Moreover we have $\frac{\partial^2 \mathcal{R}_{\O_\e}(x_{i,\e}^\pm)}{\partial x_i^2}\neq 0$ for $i=1,\cdots,N$.
\end{teo}
\begin{proof}
We prove the claim by constructing $2N$ zeros for $\nabla \mathcal{R}_{\O_\e}(x)$.
For the case $N\ge3$, as in the previous theorem we study the equation
\begin{equation*}
0=\big(\la_i+o(1)\big)Z_i-
\Big(\frac{2}{N\omega_N}+o(1)\Big)\frac{Z_i}{|Z|^{2N-2}}\e^{N-2},~~\mbox{for}~N\geq 3.
\end{equation*}
Here we  introduce the points
$$Y_{1,\e}=
\big((1-a)\alpha_1\e^\frac{N-2}{2N-2},0,\cdots,0\big)~\mbox{and}~ Y_{2,\e}=\big((1+a)\alpha_1\e^\frac{N-2}{2N-2},0,\cdots,0\big)$$ for any $a\in(0,1)$ and $\alpha_1=\left(\frac{D_N}{ \lambda_1}\right)^{\frac{1}{2N-2}}$.
We have that $\O_\e=\O\setminus B(0,\e)$ is symmetric and then
\begin{equation*}
\frac{\partial \mathcal{R}_{\O_\e}(Y_{1,\e})}{\partial y_i}=\frac{\partial \mathcal{R}_{\O_\e}(Y_{2,\e})}{\partial y_i}=0\,\,~\mbox{for any}~i=2,\cdots,N.
\end{equation*}
Also we have
\begin{equation*}
\begin{split}
\frac{\partial \mathcal{R}_{\O_\e}(Y_{1,\e})}{\partial y_1}=&\textbf{P}^T\frac{\partial \mathcal{R}_{\O_\e}(Y_{1,\e})}{\partial x_1}=
\big(\la_1+o(1)\big)Y_{1,\e}-\big(\frac{2}{N\omega_N}+o(1)\big)\frac{Y_{1,\e}}{|Y_{1,\e}|^{2N-2}}\e^{N-2}\\=
&(1-a)\alpha_1\la_1\e^\frac{N-2}{2N-2}\left[1-\frac1{(1-a)^{2N-2}}+o(1)\right]>0,
 \end{split}\end{equation*}
and similarly,
$$\frac{\partial \mathcal{R}_{\O_\e}(Y_{2,\e})}{\partial y_1}=(1+a)\alpha_1\la_1\e^\frac{N-2}{2N-2}\left[1-\frac1{(1+a)^{2N-2}}+o(1)\right]<0.$$
This implies that there exists $Z_{1,\e}=\big(b\e^\frac{N-2}{2N-2},0,\cdots,0\big)$ with $b\in\big((1-a)\alpha_1,(1+a)\alpha_1\big)$ satisfying  $\nabla \mathcal{R}_{\O_\e}(Z_{1,\e})=0$.
Taking $a$ small, we can write $b=\alpha_1+o(1)$.
Hence there exists $Z_{1,\e}=\Big(\big(\alpha_1+o(1)\big)\e^\frac{N-2}{2N-2},0,\cdots,0\Big)$ satisfying  $\nabla \mathcal{R}_{\O_\e}(Z_{1,\e})=0$.

The same computation holds if we replace
$Y_{1,\e},Y_{2,\e}$ by $-Y_{1,\e},-Y_{2,\e}$ getting the existence of a $second$ critical point $-Z_{1,\e}$.
Since $Hess\big(\mathcal{R}_\O(0)\big)$ has $N$ positive  eigenvalues
$0<\lambda_1\leq \lambda_{2} \leq \cdots \leq \lambda_N$, then repeating the argument  above for any positive eigenvalue,
 we have that $\mathcal{R}_{\O_\e}(x)$
has at least $2N$ critical points $x_{i,\e}^\pm(i=1,\cdots,N)$ in $B(0,r)\backslash B(0,\e)$, where
$B(0,r)\backslash\{0\}\subset\O$ does not contain any critical  point of $\mathcal{R}_\O(x)$.
Moreover we can write
$$x_{i,\e}^\pm= \pm Z_{i,\e} = \Big(\underbrace{0,\cdots,0}_{i-1},\pm \big(\frac{2}{N\omega_N \lambda_i}+o(1)\big)^{\frac{1}{N-2}}\e^{\frac{N-2}{2N-2}},\underbrace{0,\cdots,0}_{N-i}\Big),
~~~~\mbox{for}~~~~i=1,\cdots,N.$$
Also using \eqref{3-28} for $N\geq 3$, we compute
\begin{equation*}
\begin{split}
\frac{\partial^2 \mathcal{R}_{\O_\e}(x_{i,\e}^\pm)}{\partial x_i^2}
=&\lambda_i +\left(\frac2{N\omega_N}+o(1)\right)\Big(2N-3\Big)\frac{ 1}{|x_{i,\e}^\pm|^{2N-2}}\e^{N-2}+o\big(1\big)\\=&
\Big(2N-2\Big)\lambda_i+o\big(1\big)\neq 0,~\,~\mbox{for}~\,~i=1,\cdots,N.
\end{split}
\end{equation*}
For the case $N=2$, as in the previous theorem we study the equation
\begin{equation*}
0=\big(\la_i+o(1)\big)Z_i-
\Big(\frac{1}{\pi}+o(1)\Big)\frac{Z_i\ln|Z|}{|Z|^{2}\ln \e}.
\end{equation*}
Then similar the idea in proving the case $N\geq 3$, we can write
$$x_{1,\e}^\pm= \Big(\big(1+o(1)\big)\widehat{r}_{\e,1},0\Big)
~~~~\mbox{and}~~~~x_{2,\e}^\pm= \Big(0,\big(1+o(1)\big)\widehat{r}_{\e,2}\Big).$$
Finally using \eqref{3-28} for $N=2$, we compute
\begin{equation*}
\begin{split}
\frac{\partial^2 \mathcal{R}_{\O_\e}(x_{i,\e}^\pm)}{\partial x_i^2}
= \lambda_i +  \frac{ 1}{\pi |x_{i,\e}^\pm|^{2N-2}}\frac{\ln \widehat{r}_{\e,i}}{\ln \e}+o\big(1\big) =2\lambda_i+o\big(1\big)\neq 0,~\,~\mbox{for}~\,~i=1,2.
\end{split}
\end{equation*}
\end{proof}
\begin{rem}
Here we point out that the positive eigenvalues may be multiple in Theorem \ref{th1-2a}.
\end{rem}
\section{Examples on which  the conditions Theorem \ref{th1-2} hold}~\label{es}

\vskip 0.2cm

In Theorem \ref{th1-2} we proved that if all positive eigenvalues of the Hessian matrix of $\mathcal{R}_\Omega$ are simple, then we can give the precise number and non-degeneracy of the critical points of Robin function $\mathcal{R}_{\O_\e}(x)$.
In this section, we exhibit some domains on which this assumption is verified.

We recall that the regular part of  Green's function in $B_1=B_1(0)$ is

\begin{equation*}
H_{B_1}(x,y)=
\begin{cases}
\frac{1}{2\pi} \ln \frac{1}{|y|\cdot \big|x-\frac{y}{|y|^2}\big|}&~~\mbox{for}~~y\in B_1(0)~~\mbox{and}~~N=2,\\
\frac{1}{N(N-2)\omega_N}  \frac{1}{ |y|^{N-2}\cdot \big|x-\frac{y}{|y|^2}\big| ^{N-2}}&~~\mbox{for}~~y\in B_1(0)~~\mbox{and}~~N\geq 3.
\end{cases}
\end{equation*}
Then we have following result,
\begin{teo}\label{th1-3a}
Let $N\geq 2$ and
$$\Omega_\delta=\Big\{x\in \R^N,~ \displaystyle\sum^N_{i=1}x_i^2\Big(1+\alpha_i \delta\Big)^2 <1~~\mbox{with}~~\delta>0\hbox{ and }0<\alpha_1\leq \alpha_2\le\cdots \leq \alpha_N\Big\}.$$
Then the Robin function $\mathcal{R}_{\O_\delta}$ has a unique critical point $P=0$ and is even with respect to $x_1,\cdots,x_N$.
\vskip 0.1cm
Moreover if $\alpha_1\leq \alpha_2\le\cdots <\alpha_{l_0+1}= \alpha_{l_0+2} =\cdots= \alpha_{l_0+k}< \alpha_{l_0+k+1}\leq \cdots \leq \alpha_N$, then
 $\mathcal{R}_{\O_{\delta}}(x)$ is radial function with respect to $x_{l_0+1}\cdots x_{l_0+k}$, i.e. $$\mathcal{R}_{\O_{\delta}}(x)=
 \mathcal{R}_{\O_{\delta}}\Big(x_1,\cdots,x_{l_0},\sqrt{x^2_{l_0+1}+\cdots+ x^2_{l_0+k}},
 x_{l_0+k+1},\cdots,x_N\Big),$$
and we have that $\lambda_{l_0+1}= \lambda_{l_0+2} =\cdots= \lambda_{l_0+k}$ for any $\delta>0$.
 \vskip 0.1cm
Moreover if there exists some $k$ such that $\alpha_k\neq \alpha_j$  for any $j\in \{1,\cdots,N\}$ with $j\neq k$,
 then the $k$-th eigenvalue of $Hess\big(\mathcal{R}_{\O_{\delta}}(0)\big)$ is simple for $\delta>0$ small.
\vskip 0.1cm

Finally if  $\alpha_i\neq \alpha_j$ for any  $i,j\in\{1,\cdots,N\}$ with $i\neq j$, then all the eigenvalues of $Hess\big(\mathcal{R}_{\O_{\delta}}(0)\big)$ are simple.
\end{teo}

\begin{proof}
Firstly, suppose that $A$ is a reflection or a rotation such that $A\Omega=\Omega$. We have\\
\textup{(i)} $H(y,x)$ and $H(Ay,Ax)$ are harmonic in $y$.\\
\textup{(ii)} For any $y\in\partial\Omega$ and $x\in\Omega$ it holds $H(y,x)=H(Ay,Ax)$.

\vskip 0.1cm

Hence we have $H(y,x)=H(Ay,Ax)$ for any $x,y\in \Omega$, which gives $$\mathcal{R}(x)=\mathcal{R}(Ax).$$
Now since $\Omega_\delta$ is symmetric with respect to $x_1,\cdots,x_N$, then
the Robin function $\mathcal{R}_{\O_\delta}$  is even function with respect to $x_1,\cdots,x_N$. Also from the fact that $\Omega_\delta$ is convex and \cite{cf2,ct}, we have that  $\mathcal{R}_{\O_\delta}$ has a unique critical point $P=0$.
Moreover by \cite{g2}
$Hess\big(\mathcal{R}_{\O_{\delta}}(0)\big)$  is a diagonal matrix with $i$-th eigenvalue $\lambda_i=\frac{\partial^2 \mathcal{R}_{\O_\delta}(0)}{\partial x_i^2}$ for $i=1,\cdots,N$.
If $\alpha_1\leq \alpha_2\cdots <\alpha_{l_0+1}= \alpha_{l_0+2} =\cdots= \alpha_{l_0+k}< \alpha_{l_0+k+1}\leq \cdots\leq \alpha_N$, then $\mathcal{R}_{\O_{\delta}}(x)$ is a radial function with respect to $x_{l_0+1}\cdots x_{l_0+k}$ and  $\lambda_{l_0+1}= \lambda_{l_0+2} =\cdots= \lambda_{l_0+k}$ for any $\delta>0$.
\vskip 0.1cm
Set $\alpha_1\leq \alpha_2\cdots \leq \alpha_N$ and consider the general case for $\delta>0$ small.
Let
$g_\delta(x,y)=H_\delta(x,y)-{H}_{B(0,1)}(x,y)$, then we have
\begin{equation}\label{9-19-21}
\begin{cases}
\Delta_x g_\delta(x,y)=0,~~&\mbox{in}~~ \Omega_\delta,\\[1mm]
g_\delta(x,y)=
\begin{cases}
\frac{1}{N(N-2)\omega_N} \Big(\frac{1}{ |x-y|^{N-2}}-  \frac{1}{|y|^{N-2}\cdot \big|x-\frac{y}{|y|^2}\big|^{N-2}}\Big)&\mbox{for}~~{N\geq 3},\\[1mm]
\frac{1}{2\pi} \ln \frac{1}{ |x-y|}-\frac{1}{2\pi} \ln \frac{1}{|y|\cdot \big|x-\frac{y}{|y|^2}\big|}
&\mbox{for}~~{N=2},
\end{cases}
~~&\mbox{on}~~ \partial \Omega_\delta.
\end{cases}
\end{equation}
Next, on $\partial \Omega_\delta$, we have
\begin{equation}\label{9-19-22}
\begin{split}
\frac{1}{ |x-y|^{N-2}}&-  \frac{1}{|y|^{N-2}\cdot \big|x-\frac{y}{|y|^2}\big|^{N-2}}=
\frac{1}{ |x-y|^{N-2}}\Big( 1-\frac{ |x-y|^{N-2}} { |y|^{N-2}\cdot \big|x-\frac{y}{|y|^2} \big|^{N-2} }   \Big)\\=&
\frac{1}{ |x-y|^{N-2}}\Big( 1-\frac{ 1} {\big(1-\alpha(x,y)\big)^{\frac{N-2}{2}} }   \Big)
=
\frac{(N-2)\alpha(x,y)}{ 2|x-y|^{N-2}}+O\big(\alpha^2(x,y)\big),~~\mbox{for}~~N\geq 3,
\end{split}\end{equation}
and
\begin{equation}\label{9-19-23}
\begin{split}
 \ln \frac{1}{ |x-y|}&- \ln \frac{1}{|y|\cdot \big|x-\frac{y}{|y|^2}\big|}
 =\frac12\ln \frac{|y|^2\cdot \big|x-\frac{y}{|y|^2}\big|^2}{ |x-y|^2}\\ =& \frac12\ln\big(1-\alpha(x,y)\big)= -\frac12\alpha(x,y)+O\big(\alpha(x,y)^2\big),
\end{split}\end{equation}
where
\begin{equation*}
\alpha(x,y):=\frac{|x-y|^2-|y|^2\cdot \big|x-\frac{y}{|y|^2}\big|^2 }{|x-y|^2}.
\end{equation*}
Also for $x\in \partial \Omega_\delta$, it holds
\begin{equation}\label{9-19-24}
\begin{split}
|y|^2\cdot \big|x-\frac{y}{|y|^2}\big|^2=& |y|^2\cdot |x|^2-2\langle x,y\rangle+1\\=&
 -2\delta (|y|^2-1) \sum^N_{i=1}\alpha_ix^2_i +|x-y|^2+ O\big(\delta^2\big).
\end{split}
\end{equation}
Hence for $x\in \partial \Omega_\delta$, from \eqref{9-19-22}, \eqref{9-19-23} and \eqref{9-19-24}, we have
\begin{equation}\label{9-19-25}
\begin{split}
g_\delta(x,y)=&
-  \frac{\delta(|y|^2-1)}{N \omega_N |x-y|^N} \sum^N_{i=1}\alpha_ix_i^2+O\big(\delta^2\big),~~\mbox{for}~~N\geq 2.
\end{split}
\end{equation}
Now using \eqref{9-19-21} and \eqref{9-19-25}, we deduce
\begin{equation*}
\begin{split}
H_\delta(x,y)= {H}_{B(0,1)}(x,y)
- \frac{\delta(|y|^2-1)}{N\omega_N} v_\delta(x,y)+O\big(\delta^2\big),~~\mbox{for}~~N\geq 2,
\end{split}
\end{equation*}
where $v_\delta(x,y)$ is the solution of
\begin{equation}\label{9-19-27}
\begin{cases}
\Delta_xv_\delta(x,y) =0~~&\mbox{in}~~ \Omega_\delta,\\[1mm]
v_\delta(x,y) =\frac{1}{ |x-y|^N} \displaystyle\sum^N_{i=1}\alpha_ix_i^2~~&\mbox{on}~~ \partial \Omega_\delta.
\end{cases}
\end{equation}
Then it follows
\begin{equation}\label{9-19-28}
\mathcal{R}_\delta(x)=\mathcal{R}(x)-\frac{\delta(|x|^2-1)}{N \omega_N}v_\delta(x,x)+O\big(\delta^2\big),~~\mbox{for}~~N\geq 2.
\end{equation}
Also for $x\in \partial \Omega_\delta$ and $|y|$ small, by Taylor's expansion, it holds
\begin{equation}\label{9-19-29}
\begin{split}
\frac{ 1 }{ |x-y|^N}=&\Big( \frac{1}{ |x|^2+|y|^2-2\langle x,y \rangle }\Big)^{\frac{N}{2}}\\=&
\frac{1}{ |x|^N}\Big( 1+\frac{2\langle x,y \rangle}{|x|^2}-\frac{|y|^2}{|x|^2}
+\frac{4\langle x,y \rangle^2}{|x|^4}+O\big(|y|^3\big)\Big)^{\frac{N}{2}}\\=&
\frac{1}{ |x|^N}\Big( 1+\frac{N\langle x,y \rangle}{|x|^2}-\frac{N}{2}\frac{|y|^2}{|x|^2}
+\frac{N(N+2)}{2}\frac{\langle x,y \rangle^2}{|x|^4}+O\big(|y|^3\big)\Big).
\end{split}
\end{equation}
From \eqref{9-19-27} and \eqref{9-19-29}, we get \begin{equation*}
v_\delta(x,y)=v^{(1)}_\delta(x,y)+Nv^{(2)}_\delta(x,y)
-\frac{N}{2}v^{(3)}_\delta(x,y)+\frac{N(N+2)}{2}v^{(4)}_\delta(x,y)+ v^{(5)}_\delta(x,y),\end{equation*}
where
\begin{equation*}
\begin{cases}
\Delta_x v^{(1)}_\delta(x,y)=0 &~~\mbox{in}~~\Omega_\delta,\\[2mm]
v^{(1)}_\delta(x,y) =\frac{1}{ |x|^N} \displaystyle\sum^N_{i=1}\alpha_ix_i^2  &~~\mbox{on}~~\partial \Omega_\delta,
\end{cases} ~~~~~
\begin{cases}
\Delta_x v^{(2)}_\delta(x,y)=0 &~~\mbox{in}~~\Omega_\delta,\\[2mm]
v^{(2)}_\delta(x,y) =\frac{\langle x,y \rangle }{ |x|^{N+2}}\displaystyle\sum^N_{i=1}\alpha_ix_i^2  &~~\mbox{on}~~\partial \Omega_\delta,
\end{cases}
\end{equation*}
\begin{equation*}
\begin{cases}
\Delta_x v^{(3)}_\delta(x,y)=0 &~~\mbox{in}~~\Omega_\delta,\\[2mm]
v^{(3)}_\delta(x,y) = \frac{|y|^2}{ |x|^{N+2}}\displaystyle\sum^N_{i=1}\alpha_ix_i^2  &~~\mbox{on}~~\partial \Omega_\delta,
\end{cases}~~~~~
\begin{cases}
\Delta_x v^{(4)}_\delta(x,y)=0 &~~\mbox{in}~~\Omega_\delta,\\[2mm]
v^{(4)}_\delta(x,y) =\frac{\langle x,y \rangle^2}{ |x|^{N+4}}\displaystyle\sum^N_{i=1}\alpha_ix_i^2  &~~\mbox{on}~~\partial \Omega_\delta,
\end{cases}
\end{equation*}
and
\begin{equation*}
\begin{cases}
\Delta_x v^{(5)}_\delta(x,y)=0 &~~\mbox{in}~~\Omega_\delta,\\[2mm]
v^{(5)}_\delta(x,y) =O\big(|y|^3\big)  &~~\mbox{on}~~\partial \Omega_\delta.
\end{cases}
\end{equation*}
We know that as $\delta\to 0$, $v_\delta^{(1)}(x,y)\rightarrow v^{(1)}(x,y)$, where
\begin{equation}\label{9-1-1a}
\Delta v^{(1)}(x,y)=0~~\mbox{in}~~B_1(0),~~~\, v^{(1)}(x,y)=\displaystyle\sum^N_{i=1}\alpha_ix_i^2~~\mbox{on}~~\partial B_1(0).
\end{equation}
Solving \eqref{9-1-1a}, we can get
\begin{equation*}
v^{(1)}(x,y)= -\frac{(N-1)}{2}\big(|x|^2-1\big)+ \displaystyle\sum^N_{i=1}\alpha_ix_i^2.
\end{equation*}
Then it holds
\begin{equation*}
\begin{split}
(|x|^2-1)v^{(1)}(x,x)=&-\frac{(N-1)}{2}\big(|x|^2-1\big)^2+ \big(|x|^2-1\big)\displaystyle\sum^N_{i=1}\alpha_ix_i^2,
\end{split}\end{equation*}
and
\begin{equation*}
\begin{split}
\nabla^2 &\Big( (|x|^2-1) v^{(1)}(x,x)\Big)\Big|_{x=0}=
\nabla^2 \Big( \big(N-1\big) |x|^2- \displaystyle\sum^N_{i=1}\alpha_ix_i^2\Big)\Big|_{x=0}\\=&
  2(N-1) \textbf{E}_N-  2 ~\mbox{diag}\big(\alpha_1,\alpha_2,\cdots,\alpha_N\big),
\end{split}\end{equation*}
where $\textbf{E}_N$ is the unit matrix. In the same way we have that, for  $\delta\to 0$,
$v_\delta^{(i)}(x,y)\rightarrow v^{(i)}(x,y)$, for $i=2,3,4$ where
\begin{equation*}
\begin{split}
v^{(2)}(x,y)=& \langle x,y\rangle \displaystyle\sum^N_{i=1}\alpha_ix_i^2
-\frac{N(N-1)}{2(N+2)}\big(|x|^2-1\big) \langle x,y\rangle
-\frac{2(|x|^2-1)}{N+2} \sum^N_{i=1}\alpha_ix_iy_i,
\end{split}\end{equation*}
\begin{equation*}
v^{(3)}(x,y)=
 |y|^2 v^{(1)}(x,y)
= |y|^2\left(-\frac{(N-1)}{2}\big(|x|^2-1\big)+ \displaystyle\sum^N_{i=1}\alpha_ix_i^2\right).
\end{equation*}
and
\begin{equation*}
\begin{split}
 {v}^{(4)}(x,y)=&-2|y|^2
\Big( \displaystyle\sum^N_{i=1}\alpha_ig_i(x)\Big) -\sum^N_{i=1}
\Big(N(N-1)+8\alpha_i\Big)g_i(x)y_i^2\\&
-
\sum^N_{i=1}\sum_{j\neq i}
\Big(N(N-1)+8\alpha_i\Big)f_{ij}(x)y_i y_j+\Big( \displaystyle\sum^N_{i=1}\alpha_ix_i^2\Big)\langle x,y \rangle^2,
\end{split}
\end{equation*}
with
$$f_{ij}(x)=\frac{x_ix_j(|x|^2-1)}{12}~~\mbox{and}~~g_{i}(x)=\frac{(|x|^2-1)x_i^2}{2(N+4)} -\frac{(|x|^4-1)}{4(N+2)(N+4)}+\frac{(|x|^2-1)}{2N(N+4)}.$$
For future aims it will be useful to remark that
\begin{equation*}
\begin{split}
v^{(4)}(x,x)=&\frac{N(N-1)|x|^2+4\displaystyle\sum^N_{i=1}\alpha_i x_i^2}{2N(N+2)}  +O\big(|x|^4\big).
\end{split}\end{equation*}
Finally since $v_\delta^{(5)}(x,y)=O\big(|y|^3\big)$ as  $\delta\to 0$ we get
\begin{equation*}
\nabla^2 \Big( (|x|^2-1)v^{(5)}_\delta(x,x)\Big)\Big|_{x=0}
\to \textbf{O}_N,
\end{equation*}
where $\textbf{O}_N$ is the zero matrix in $\mathbb{R}^N$.
Hence by the explicit form of $v_\delta^{(1)},..,v_\delta^{(4)}$, a straightforward (and tedious) computation gives that
\begin{equation}\label{9-19-45}
\begin{split}
\nabla^2 &\Big( (|x|^2-1)v_\delta(x,x)\Big)\Big|_{x=0}
\\&\to
-\frac{\big(N-1\big)\big( N^2-2N-4\big)}{N+2}  \textbf{E}_N- \frac{8\big(N+1\big)}{N+2} ~\mbox{diag}\big(\alpha_1,\cdots,\alpha_N\big).
\end{split}\end{equation}
Observe that
\begin{equation}\label{9-19-46}
\frac{\partial^2 \mathcal{R}_{B_1}(x)}{\partial x_i\partial x_j}
=\frac{2}{N\omega_N}\delta^i_j,~~\mbox{for}~~1\leq i,j\leq N.
\end{equation}
Finally, let $\lambda_i$ for $i=1,\cdots,N$ be the eigenvalues of $Hess\big(\mathcal{R}_{\O_\d}(0)\big)$; for $N\geq 2$, from \eqref{9-19-28}, \eqref{9-19-45} and \eqref{9-19-46}, we find
\begin{equation*}
\lambda_i=\frac{2}{N\omega_N}+\frac{(N-1)(N^2-2N-4)+8(N+1)\alpha_i}{N(N+2)\omega_N} \delta+o\big(\delta\big),~~\mbox{for}~~i=1,\cdots,N.
\end{equation*}

Hence we deduce that  if there exists some $k$ such that $\alpha_k\neq \alpha_j$  for any $j\in \{1,\cdots,N\}$ with $j\neq k$,
 then the corresponding $k$-th eigenvalue of $Hess\big(\mathcal{R}_{\O_{\delta}}(0)\big)$ is simple for $\delta>0$ small.
\vskip 0.1cm

Furthermore, if  $\alpha_i\neq \alpha_j$ for any  $i,j\in\{1,\cdots,N\}$ with $i\neq j$, then all the eigenvalues of $Hess\big(\mathcal{R}_{\O_{\delta}}(0)\big)$ are simple.
\end{proof}

\vskip 0.2cm
\appendix\section{Some useful computations on the Green function in the exterior of the ball}\label{s2a}

We recall that Green function $G_{B_\e^c}(x,y)$ of $\R^N\backslash B_\e$ is given by (see \cite{bf} for example)
\begin{equation*}
G_{B_\e^c}(x,y)=
\begin{cases}
-\frac{1}{2\pi}\left(\ln \big|{x-y}\big|-\ln\sqrt{\frac {|x|^2|y|^2}{\e^2}+\e^2-2x\cdot y}\right)&~\mbox{if}~N=2,\\[2mm]
C_N
\left(\frac{1}{|x-y|^{N-2}}-\frac{\e^{N-2}}{\big||x|y-\e^2\frac{x}{|x|}\big|^{N-2}}\right) &~\mbox{if}~N\geq 3,
\end{cases}
\end{equation*}
 and, for $i=1,\cdots,N$,
\begin{equation}\label{ader}
\frac{\partial G_{B_\e^c}(x,y)}{\partial y_i}=
\frac1{N\omega_N}\left(\frac{x_i-y_i}{|x-y|^{N}}+\e^{N-2}\frac{|x|^2y_i-\e^2x_i}{\big(|x|^2|y|^2-2\e^2x\cdot y+\e^4\big)^{\frac N2}}\right).
\end{equation}
The {\em Robin function} of $\R^N\backslash B_\e$ is given by (see \cite{bf})
\begin{equation}\label{robin-esterno-palla}
\mathcal{R}_{B_\e^c}(x)=
\begin{cases}
\frac{1}{2\pi}\ln\frac {\e}{|x|^2-\e^2}&~\mbox{if}~N=2,\\[2mm]
C_N\frac{\e^{N-2}}{\big(|x|^2-\e^2\big)^{N-2}} &~\mbox{if}~N\geq 3.
\end{cases}
\end{equation}
It will be also useful the well-known representation formula for harmonic function in the ball $B_\e$: if $u$ is harmonic in $B_\e$ we have that
\begin{tcolorbox}[colback=white,colframe=black]
\begin{equation}\label{aar}
u(x)=-\int_{\partial B_\e}\frac{\partial G_{B_\e}(x,y)}{\partial\nu_y}u(y)ds_y=\frac{1}{N\omega_N}\frac{\e^2-|x|^2}{\e}\int_{\partial B_\e}\frac{u(y)}{|x-y|^N}ds_y.
\end{equation}
\end{tcolorbox}

In next lemma we prove some identities which use in our computations.
\begin{lem}
We have that, for any $x\in B^c_\e$ and $N\ge2$ the following equalities hold:
\begin{equation}\label{ap1}
\int_{\partial B_\e}\frac{\partial G_{B^c_\e}(x,y)}{\partial\nu_y}d\sigma_y=-\frac{\e^{N-2}}{|x|^{N-2}},
\end{equation}
\begin{equation}\label{ap2}
\int_{\partial B_\e}y_j\frac{\partial G_{B^c_\e}(x,y)}{\partial\nu_y}d\sigma_y=-\frac{x_j}{|x|^N}\e^N
\end{equation}
and
\begin{equation}\label{ap10}
\int_{\partial B_\e}y_iy_j\frac{\partial G_{B^c_\e}(x,y)}{\partial\nu_y}d\sigma_y=
-\frac{\e^N}{|x|^N}\left[\e^2\frac{x_ix_j}{|x|^2}+
\frac{\delta^i_j}N\big(|x|^2-\e^2\big)\right].
\end{equation}

\end{lem}
\begin{proof}
In order to prove \eqref{ap1} and  \eqref{ap2} we will use the well known facts that, for $x\in B_\e$,
\begin{equation}\label{ap3}
\int_{\partial B_\e}\frac{\partial G_{B_\e}(x,y)}{\partial\nu_y}d\sigma_y=-1\quad\hbox{and }
\int_{\partial B_\e}y_j\frac{\partial G_{B_\e}(x,y)}{\partial\nu_y}d\sigma_y=-x_j,
\end{equation}
(these can be deduced by \eqref{aar}).
Let us prove \eqref{ap1} for $N\geq 3$. We have that, for for $x\in B_\e^c$,
\begin{equation*}
\begin{split}
\int_{\partial B_\e}& \frac{\partial G_{B^c_\e}(x,y)}{\partial\nu_y}d\sigma_y =\frac{\e^2-|x|^2}{N\omega_N\e}\int_{\partial B_\e}\frac1{|x-y|^N}d\sigma_y=(\hbox{setting $x=\e^2\frac t{|t|^2}$, and then $|t|\le\e$})
\\
&\e\frac{|t|^2-\e^2}{N\omega_N|t|^2}\int_{\partial B_\e}\frac{|t|^{2N}}{|\e^2t-|t|^2y|^N}d\sigma_y=
\frac{(|t|^2-\e^2)|t|^{N-2}}{N\omega_N\e^{N-1}}\int_{\partial B_\e}\frac1{|t-y|^N}d\sigma_y\\& \hbox{ (using \eqref{ap3} for }t\in B_\e)= \frac{|t|^{N-2}}{\e^{N-2}}\int_{\partial B_\e}\frac{\partial G_{B_\e}(t,y)}{\partial\nu_y}=-\frac{\e^{N-2}}{|x|^{N-2}}.
 \end{split}
\end{equation*}
In the same way we get \eqref{ap2},
\begin{equation*}
\begin{split}
\int_{\partial B_\e}& y_j\frac{\partial G_{B^c_\e}(x,y)}{\partial\nu_y}d\sigma_y =\frac{\e^2-|x|^2}{N\omega_N\e}\int_{\partial B_\e}\frac{y_j}{|x-y|^N}d\sigma_y=(\hbox{setting $x=\e^2\frac t{|t|^2}$, and then $|t|\le\e$})\\
 &=\e\frac{|t|^2-\e^2}{N\omega_N|t|^2}\int_{\partial B_\e}\frac{|t|^{2N}y_j}{|\e^2t-|t|^2y|^N}d\sigma_y=
\frac{(|t|^2-\e^2)|t|^{N-2}}{N\omega_N\e^{N-1}}\int_{\partial B_\e}\frac{y_j}{|t-y|^N}d\sigma_y\\
&\hbox{ (using \eqref{ap3} for }t\in B_\e)=\frac{|t|^{N-2}}{\e^{N-2}}t_j\int_{\partial B_\e}\frac{\partial G_{B_\e}(x,y)}{\partial\nu_y}=-\frac{x_j}{|x|^N}\e^N.
 \end{split}
\end{equation*}
Let us prove \eqref{ap10}. We have that the function
$$a(x)=x_ix_j+\frac{\delta^i_j}N(\e^2-|x|^2)$$
is harmonic and $a\big|_{\partial B_\e}=x_ix_j$. So
\begin{equation}\label{ap14}
\int_{\partial B_\e}y_jy_j\frac{\partial G_{B_\e}(x,y)}{\partial\nu_y}=-\left(x_ix_j+\frac{\delta^i_j}N(\e^2-|x|^2)\right)
\end{equation}
and arguing as before we that
\begin{equation*}
\begin{split}
\int_{\partial B_\e}& y_iy_j\frac{\partial G_{B^c_\e}(x,y)}{\partial\nu_y}d\sigma_y =\frac{\e^2-|x|^2}{N\omega_N\e}\int_{\partial B_\e}\frac{y_iy_j}{|x-y|^N}d\sigma_y=(\hbox{setting $x=\e^2\frac t{|t|^2}$, and then $|t|\le\e$})
\\
&=\e\frac{|t|^2-\e^2}{N\omega_N|t|^2}\int_{\partial B_\e}\frac{|t|^{2N}y_iy_j}{|\e^2t-|t|^2y|^N}d\sigma_y=
\frac{(|t|^2-\e^2)|t|^{N-2}}{N\omega_N\e^{N-1}}\int_{\partial B_\e}\frac{y_iy_j}{|t-y|^N}d\sigma_y\\
&\hbox{ (using \eqref{ap14} for }t\in B_\e)=-\frac{|t|^{N-2}}{\e^{N-2}}\left(t_it_j+\frac{\delta^i_j}N(\e^2-|t|^2)\right)\\
&=
-\frac{\e^{N-2}}{|x|^{N-2}}\left[\e^4\frac{x_ix_j}{|x|^4}+\frac{\delta^i_j}N
\e^2\left(1-\frac{\e^2}{|x|^2}\right)\right]=-\frac{\e^N}{|x|^N}
\left[\e^2\frac{x_ix_j}{|x|^2}+\frac{\delta^i_j}N\big(|x|^2-\e^2\big)\right],
 \end{split}
\end{equation*}
which proves the claim.
\end{proof}
Next lemma concerns some identities on the Robin function.
\begin{lem}\label{Lap}
We have that, for any $x\in B^c_\e$ and $N\ge2$ the following identities hold:
\begin{equation}\label{lap1}
\int_{\partial B_\e}\frac y\e\left(\frac{\partial G_{B^c_\e}(x,y)}{\partial\nu_y}\right)^2 d\sigma_y=-\nabla\mathcal{R}_{B_\e^c}(x)= \frac 2{N\omega_N}\e^{N-2}\frac x{ (|x|^2-\e^2)^{N-1}}
\end{equation}
and
\begin{equation}\label{lap2}
-\int_{\partial B_\e}\frac{\partial G_{B_\e^c}(x,y)}{\partial\nu_y}S(x,y)d\sigma_y=
\begin{cases}
C_N\frac{\e^{N-2}}{\left(|x|^2-\e^2\right)^{N-2}}=\mathcal{R}_{B_\e^c}(x)&\mbox{ if }N\ge3,\\[3mm]
\frac1{2\pi}\left(-\ln\frac{|x|^2}{|x|^2-\e^2}+\ln|x|\right)=\mathcal{R}_{B_\e^c}(x)+
\frac1{2\pi}\ln\frac{|x|}\e&\mbox{ if }N=2.
\end{cases}
\end{equation}

\end{lem}
\begin{proof}
Let us prove \eqref{lap1}. For $t\in B_\e$, using the expression of the Green and the Robin function in $B_\e$ and recalling \cite{bp},  p.  170, we have
\begin{equation*}
\begin{split}
\int_{\partial B_\e}\nu(y)\left(\frac{\partial G_{B_\e}(t,y)}{\partial\nu_y}\right)^2 d\sigma_y=\nabla\mathcal{R}_{B_\e}(t)=\frac 2{N\omega_N}\e^{N-2}\frac t{ (\e^2-|t|^2)^{N-1}}
 \end{split}
\end{equation*}
and we deduce the following formula,
\begin{equation}\label{Res}
\frac 1{(N\omega_N)^2}\int_{\partial B_\e}\frac y\e\frac1{|t-y|^{2N}}d\sigma_y=\frac 2{N\omega_N}\e^{N-2}\frac t{ (\e^2-|t|^2)^{N-1}}\quad\hbox{for any }t\in B_\e.
\end{equation}
Now, let $x\in B^c_\e$, then using \eqref{ader}
\begin{equation*}
\begin{split}
\int_{\partial B_\e}&\frac y\e\left(\frac{\partial G_{B^c_\e}(x,y)}{\partial\nu_y}\right)^2d\sigma_y=
\frac 1{(N\omega_N)^2}
\frac{\left(\e^2-|x|^2\right)^2}{\e^2}\int_{\partial B_\e}\frac y\e\frac1{|x-y|^{2N}}d\sigma_y\\
&(\hbox{setting $x=\e^2\frac t{|t|^2}$ and so $|t|\le\e$})=
\frac 1{(N\omega_N)^2}
\frac{\e^2}{|t|^4}(|t|^2-\e^2)^2\int_{|y|=\e}\frac y\e\frac1{\left|\frac{\e^2t-|t|^2y}{|t|^2}\right|^{2N}}d\sigma_y\\=
&\frac 1{(N\omega_N)^2}(|t|^2-\e^2)^2\frac{|t|^{2N-4}}{\e^{2N-2}}\int_{|y|=\e}\frac y\e\frac1{|t-y|^{2N}}d\sigma_y\\
&(\hbox{using  \eqref{Res} for $|t|\le\e$})=\frac 2{N\omega_N}
\frac{|t|^{2N-4}}{\e^{N-2}}\frac t{\left(\e^2-|t|^2\right)^{N-1}}\\=
&\frac 2{N\omega_N}
\e^{N-2}\frac x{ (|x|^2-\e^2)^{N-1}}=-\nabla\mathcal{R}_{B_\e^c}(x),
 \end{split}
\end{equation*}
which gives the claim.

\vskip 0.1cm
Next let us prove \eqref{lap2}. If $N\ge3$ we have
 \begin{equation*}
  \begin{split}
 \int_{\partial B_\e}& \frac{\partial G_{B_\e^c}(x,y)}{\partial\nu_y}S(x,y)d\sigma_y=
\frac{\e^2-|x|^2}{N\omega_N\e}C_N\int_{\partial B_\e}\frac1{|x-y|^{2N-2}}d\sigma_y\\
&(\hbox{setting $x=\e^2\frac t{|t|^2}$})=\frac{|t|^2-\e^2}{N\omega_N\e}C_N\frac{|t|^{2N-4}}{\e^{2N-4}}\int_{\partial B_\e}\frac1{|t-y|^{2N-2}}d\sigma_y\\=&\frac{|t|^{2N-4}}{\e^{2N-4}}\int_{\partial B_\e}\frac{\partial G_{B_\e}(t,y)}{\partial\nu_y}S(t,y)d\sigma_y= -C_N\frac{|t|^{2N-4}}{\e^{2N-4}}\frac{\e^{N-2}}{\left(\e^2-|t|^2\right)^{N-2}}=
-C_N\frac{\e^{N-2}}{\left(|x|^2-\e^2\right)^{N-2}}.
\end{split}
 \end{equation*}
If $N=2$ we have, arguing as above,
 \begin{equation*}
  \begin{split}
\int_{\partial B_\e}&\frac{\partial G_{B_\e^c}(x,y)}{\partial\nu_y}S(x,y)d\sigma_y=
-\frac{\e^2-|x|^2}{4\pi^2\e}\int_{\partial B_\e}\frac{\ln|x-y|}{|x-y|^2}d\sigma_y\\=
&-\e^2\frac{|t|^2-\e^2}{4\pi^2\e|t|^2}\int_{\partial B_\e}\frac{\ln|\e^2\frac t{|t|^2}-y|}{|\e^2\frac t{|t|^2}-y|^2}d\sigma_y=-\frac{|t|^2-\e^2}{4\pi^2\e}\int_{\partial B_\e}\frac{\ln|t-y|+\ln\e-\ln|t|}{|t-y|^2}d\sigma_y\\=
&\int_{\partial B_\e}\frac{\partial G_{B_\e}(t,y)}{\partial\nu_y}S(t,y)d\sigma_y-\frac{|t|^2-\e^2}{4\pi^2\e}\ln\frac\e{|t|}\int_{\partial B_\e}\frac1{|t-y|^2}d\sigma_y\\=
&-\mathcal{R}_{B_\e}(t)+\frac1{2\pi}\ln\frac\e{|t|}=-\frac 1{2\pi}\ln \frac \e{\e^2-|t|^2}+\frac1{2\pi}\ln\frac\e{|t|}\\
=&-\mathcal{R}_{B_\e^c}(x) -\frac1{2\pi}\ln\frac{|x|}\e,
\end{split}
 \end{equation*}
which ends the proof.
\end{proof}
The final lemma computes some useful integrals.
\begin{lem}
We have that, for any $x\in B^c_\e$ and $N\ge2$ the following identities hold,
 \begin{equation}\label{ap5}
\int_{\partial B_\e}\frac{C_N}{|x-y|^{N-2}}d\sigma_y=\frac1{N-2}\frac{\e^{N-1}}{|x|^{N-2}},
\end{equation}
\begin{equation}\label{ap6}
\int_{\partial B_\e}\ln|x-y|d\sigma_y=2\pi\e\ln|x|,
\end{equation}
and
\begin{equation}\label{ap7}
\int_{\partial B_\e}\big|\ln|x-y|\big|d\sigma_y=O\left(\e\big|\ln|x|\big|\right).
\end{equation}
\end{lem}
\begin{proof}
For $x\in B_\e^c$, using \eqref{ap1} and \eqref{ap2},  we have
\begin{equation*}
\begin{split}
\int_{\partial B_\e}&\frac{C_N}{|x-y|^{N-2}}d\sigma_y=C_N\int_{\partial B_\e}\frac{|x|^2+\e^2-2x\cdot y}{|x-y|^N}d\sigma_y
\\&=\frac \e{N-2}\frac{|x|^2+\e^2}{\e^2-|x|^2}\int _{\partial B_\e}\frac{\partial G_{B_\e^c}(x,y)}{\partial \nu_y}d\sigma_y-\frac{2\e}{N-2}\frac{x}{\e^2-|x|^2}\cdot \int _{\partial B_\e}y\frac{\partial G_{B_\e^c}(x,y)}{\partial \nu_y}d\sigma_y\\
&=\frac {\e^{N-1}}{N-2}\frac 1{|x|^{N-2}}.
\end{split}
\end{equation*}
Let us prove \eqref{ap6}. Denoting by $F(x)=\displaystyle\int_{\partial B_\e}\ln|x-y|d\sigma_y$, by \eqref{aar} we get, for any $x\in B_\e$,
\begin{equation*}
\begin{split}
\frac{\partial F(x)}{\partial x_i}=& x_i\int_{\partial B_\e}\frac1{|x-y|^2}d\sigma_y-\int_{\partial B_\e}\frac{y_i}{|x-y|^2}d\sigma_y=0,
\end{split}\end{equation*}
and then $F(x)=F(0)=\displaystyle\int_{\partial B_\e}\ln|y|d\sigma_y=2\pi\e\ln\e.$ Next, if $x\in B_\e^c$,
\begin{equation*}
\begin{split}
 \int_{\partial B_\e}\ln|x-y|d\sigma_y& \left(\hbox{setting $x=\e^2\frac t{|t|^2}$}\right)=\int_{\partial B_\e}\big(\ln|x-y|+\ln\frac\e{|t|}\big)d\sigma_y\\=
&2\pi\e\ln\e+2\pi\e\ln\frac{|x|}\e=2\pi\e\ln|x|,
\end{split}
\end{equation*}
which proves \eqref{ap6}.\\
Finally to estimate $\int_{\partial B_\e}\big|\ln|x-y|\big|d\sigma_y$ we consider the alternative either $\e<|x|\le\frac12$ or $|x|\ge\frac12$.\\
If  $|x|\ge\frac12$ then
\begin{equation*}
\begin{split}
\int_{\partial B_\e}\big|\ln|x-y|\big|d\sigma_y&\leq\int_{\partial B_\e}\left[\big|\ln|x|\big|+\frac12\Big|\underbrace{\ln\left(1-2\frac{x\cdot y}{|x|^2}+\frac{\e^2}{|x|^2}\right)}_{=o(1)}\Big|\right]d\sigma_y\\
&=\big(2\pi\big|\ln|x|\big|+o(1)\big)\e.
\end{split}
\end{equation*}
On the other hand, if $\e<|x|\le\frac12$, for $y\in\partial B_\e$ we have $|x-y|<1$. So
\begin{equation*}
\int_{\partial B_\e}\big|\ln|x-y|\big|d\sigma_y=-\int_{\partial B_\e}\ln|x-y|d\sigma_y=-2\pi\e\ln|x|,
\end{equation*}
which ends the proof.

\end{proof}
\noindent\textbf{Acknowledgments} ~Part of this work was done while Peng Luo  was visiting the Mathematics Department
of the University of Rome ``La Sapienza" whose members he would like to thank for their warm hospitality.  Massimo Grossi was supported
by INDAM-GNAMPA project. Francesca Gladiali was supported by FdS-Uniss 2017 and by INDAM-GNAMPA project. Peng Luo  was supported by NSFC grants (No.12171183,11831009).
Shusen Yan was supported by  NSFC grants (No.12171184).
\bibliographystyle{abbrv}
\bibliography{GladialiGrossiLuoYan}

\begin{thebibliography}{10}

\bibitem{blr}
A.~Bahri, Y.~Li, and O.~Rey.
\newblock On a variational problem with lack of compactness: the topological
  effect of the critical points at infinity.
\newblock {\em Calc. Var. Partial Differential Equations}, 3(1):67--93, 1995.

\bibitem{bf}
C.~Bandle and M.~Flucher.
\newblock Harmonic radius and concentration of energy; hyperbolic radius and
  {L}iouville's equations {$\Delta U=e^U$} and {$\Delta U=U^{(n+2)/(n-2)}$}.
\newblock {\em SIAM Rev.}, 38(2):191--238, 1996.

\bibitem{bn}
H.~Brezis and L.~Nirenberg.
\newblock Positive solutions of nonlinear elliptic equations involving critical
  sobolev exponents.
\newblock {\em Communications on Pure and Applied Mathematics}, 36(4):437--477,
  1983.

\bibitem{bp}
H.~Brezis and L.~A. Peletier.
\newblock Asymptotics for elliptic equations involving critical growth.
\newblock In {\em Partial differential equations and the calculus of
  variations, {V}ol. {I}}, volume~1 of {\em Progr. Nonlinear Differential
  Equations Appl.}, pages 149--192. Birkh\"{a}user Boston, Boston, MA, 1989.

\bibitem{cf}
L.~A. Caffarelli and A.~Friedman.
\newblock Asymptotic estimates for the plasma problem.
\newblock {\em Duke Math. J.}, 47(3):705--742, 1980.

\bibitem{cf2}
L.~A. Caffarelli and A.~Friedman.
\newblock Convexity of solutions of semilinear elliptic equations.
\newblock {\em Duke Math. J.}, 52(2):431--456, 1985.

\bibitem{CGPY2019}
D.~Cao, Y.~Guo, S.~Peng, and S.~Yan.
\newblock Local uniqueness for vortex patch problem in incompressible planar
  steady flow.
\newblock {\em J. Math. Pures Appl. (9)}, 131:251--289, 2019.

\bibitem{CPY2010}
D.~Cao, S.~Peng, and S.~Yan.
\newblock Multiplicity of solutions for the plasma problem in two dimensions.
\newblock {\em Adv. Math.}, 225(5):2741--2785, 2010.

\bibitem{CPY2015}
D.~Cao, S.~Peng, and S.~Yan.
\newblock Planar vortex patch problem in incompressible steady flow.
\newblock {\em Adv. Math.}, 270:263--301, 2015.

\bibitem{ct}
P.~Cardaliaguet and R.~Tahraoui.
\newblock On the strict concavity of the harmonic radius in dimension
  {$N\ge3$}.
\newblock {\em J. Math. Pures Appl. (9)}, 81(3):223--240, 2002.

\bibitem{DIP}
F.~De~Marchis, I.~Ianni, and F.~Pacella.
\newblock Asymptotic profile of positive solutions of {L}ane-{E}mden problems
  in dimension two.
\newblock {\em J. Fixed Point Theory Appl.}, 19(1):889--916, 2017.

\bibitem{f}
A.~Friedman.
\newblock {\em Variational principles and free-boundary problems}.
\newblock A Wiley-Interscience Publication. John Wiley \& Sons, Inc., New York,
  1982.

\bibitem{gsc}
P.~R. Garabedian and M.~Schiffer.
\newblock On estimation of electrostatic capacity.
\newblock {\em Proc. Amer. Math. Soc.}, 5:206--211, 1954.

\bibitem{gnn}
B.~Gidas, W.~M. Ni, and L.~Nirenberg.
\newblock Symmetry and related properties via the maximum principle.
\newblock {\em Comm. Math. Phys.}, 68(3):209--243, 1979.

\bibitem{gt}
D.~Gilbarg and N.~S. Trudinger.
\newblock {\em Elliptic partial differential equations of second order}.
\newblock Classics in Mathematics. Springer-Verlag, Berlin, 2001.
\newblock Reprint of the 1998 edition.

\bibitem{Glangetas}
L.~Glangetas.
\newblock Uniqueness of positive solutions of a nonlinear elliptic equation
  involving the critical exponent.
\newblock {\em Nonlinear Anal.}, 20(5):571--603, 1993.

\bibitem{g2}
M.~Grossi.
\newblock On the nondegeneracy of the critical points of the {R}obin function
  in symmetric domains.
\newblock {\em C. R. Math. Acad. Sci. Paris}, 335(2):157--160, 2002.

\bibitem{GILY}
M.~Grossi, I.~Ianni, P.~Luo, and S.~Yan.
\newblock Non-degeneracy and local uniqueness of positive solutions to the
  {L}ane-{E}mden problem in dimension two.
\newblock {\em J. Math. Pures Appl. (9)}, 157:145--210, 2022.

\bibitem{gu}
B.~Gustafsson.
\newblock Vortex motion and geometric function theory: the role of connections.
\newblock {\em Philos. Trans. Roy. Soc. A}, 377(2158):20180341, 27, 2019.

\bibitem{h}
Z.~C. Han.
\newblock Asymptotic approach to singular solutions for nonlinear elliptic
  equations involving critical sobolev exponent.
\newblock {\em Annales de l'I.H.P. Analyse non linéaire}, 8(2):159--174, 1991.

\bibitem{mipi}
A.~M. Micheletti and A.~Pistoia.
\newblock Non degeneracy of critical points of the {R}obin function with
  respect to deformations of the domain.
\newblock {\em Potential Anal.}, 40(2):103--116, 2014.

\bibitem{O}
S.~Ozawa.
\newblock Singular variation of domains and eigenvalues of the {L}aplacian.
\newblock {\em Duke Math. J.}, 48(4):767--778, 1981.

\bibitem{Rey}
O.~Rey.
\newblock The role of the {G}reen's function in a nonlinear elliptic equation
  involving the critical {S}obolev exponent.
\newblock {\em J. Funct. Anal.}, 89(1):1--52, 1990.

\bibitem{sc}
M.~Schiffer.
\newblock Hadamard's formula and variation of domain-functions.
\newblock {\em Amer. J. Math.}, 68:417--448, 1946.

\bibitem{ss}
M.~Schiffer and D.~C. Spencer.
\newblock {\em Functionals of finite {R}iemann surfaces}.
\newblock Princeton University Press, Princeton, N. J., 1954.

\end{thebibliography}
\end{document}